\documentclass[11pt]{amsart}
\usepackage[T1]{fontenc}
\usepackage{newtxtext,newtxmath,verbatim,color}

\usepackage[margin=1in]{geometry}
\usepackage{latexsym} %,amssymb,amsmath}
\usepackage{hyperref}
\usepackage{comment}
\usepackage{mathrsfs}
\usepackage[all]{xy}
\usepackage{tabls}

\newtheorem{theorem}{Theorem}[section]
\newtheorem{lemma}[theorem]{Lemma}
\newtheorem{proposition}[theorem]{Proposition}
\newtheorem{corollary}[theorem]{Corollary}

\newtheorem{question}[theorem]{Question}

\theoremstyle{definition}
\newtheorem{definition}[theorem]{Definition}
\newtheorem{example}[theorem]{Example}

\theoremstyle{remark}
\newtheorem{remark}[theorem]{Remark}

\numberwithin{equation}{section}

\DeclareMathOperator{\Hom}{Hom}
\DeclareMathOperator{\End}{End}
\DeclareMathOperator{\GL}{GL}
\DeclareMathOperator{\Mat}{Mat}
\renewcommand\top{\mathrm{top}\,\!}

\newcommand\ZZ{{\mathbb Z}}
\newcommand\QQ{{\mathbb Q}}
\newcommand\RR{{\mathbb R}}
\newcommand\CC{{\mathbb C}}
\newcommand\FF{{\mathbb F}}
\renewcommand\SS{\mathfrak{S}}
\newcommand{\spann}{\operatorname{span}_{\ZZ}}

  \newcommand\rad{\operatorname{rad}}
\newcommand{\leftast}[1]{\left.^*{#1}\right.}
\newcommand\C{\mathcal{C}}
\DeclareMathOperator{\id}{id}

\DeclareMathOperator{\im}{im}

\DeclareMathOperator{\FPdim}{\mathsf{FPdim}}

\DeclareMathOperator{\opp}{{opp}}
\DeclareMathOperator{\rank}{{rank}}
\DeclareMathOperator{\adj}{{adj}}
\DeclareMathOperator{\Ann}{{Ann}}
\DeclareMathOperator{\LKer}{{LKer}}

\newcommand\groth{{G_0(A)}}

\newcommand\xx{\mathbf{x}}
\newcommand\sss{\mathbf{s}}
\newcommand\ppp{\mathbf{p}}
\newcommand\nnn{\mathbf{n}}
\newcommand\GG{\mathbf{G}}

\newcommand\NN{\mathbb{N}}

\newcommand{\card}[1]{\# #1}
\newcommand\arxiv[1]{\href{http://www.arxiv.org/abs/#1}{\texttt{arXiv:#1}}}
\newenvironment{verlong}{}{}
\newenvironment{vershort}{}{}
\newenvironment{noncompile}{}{}
\excludecomment{verlong}
\includecomment{vershort}
\excludecomment{noncompile}

\begin{document}
\title{Critical groups for Hopf algebra modules}

\author{Darij Grinberg}
\address{School of Mathematics,
University of Minnesota,
206 Church St. SE,
Minneapolis, MN 55455}
\email{darijgrinberg@gmail.com}

\author{Jia Huang}
\address{Department of Mathematics and Statistics, 
University of Nebraska at Kearney,
Kearney, NE 68849}
\email{huangj2@unk.edu}

\author{Victor Reiner}
\address{School of Mathematics,
University of Minnesota,
206 Church St. SE,
Minneapolis, MN 55455}
\email{reiner@math.umn.edu}

\date{13 June 2026 (corrected version)}


\thanks{Third author supported by NSF grants DMS-1148634, 1601961.}


\maketitle

\begin{abstract} 
This paper considers an invariant of modules
over a finite-dimensional Hopf algebra, called the critical
group.  This generalizes the critical groups of
complex finite group representations studied in
\cite{BenkartKlivansReiner, Gaetz}.
A formula is given for the cardinality of 
the critical group generally,
and the critical group for the regular representation
is described completely.
A key role in the formulas is played by the greatest common
divisor of the dimensions of the indecomposable projective
representations.
\end{abstract}

%\tableofcontents

%%%%%%%%%%%%%%%%%%%%%%%%
\begin{comment}

%$A$-Mod <-> finite abelian categories

%Tensor product -> monoidal category 

%Dual -> rigid 

%Grothendieck ring -> $\mathbb Z_+$-ring, dim -> FPdim

%ordinary group representation -> finite fusion ring 

\[
\xymatrix{
&\text{Finite tensor category}& \\
\text{Finite semisimple tensor category} \ar[ur]& &
    \text{Finite-dimensional Hopf algebra} \ar[ul] \\
& & \text{Group algebra}\ar[u]\\
&\text{Semisimple group algebra}\ar[ur]\ar[uul]&
}
\]

\end{comment}
%%%%%%%%%%%%%%%%%%%%%%

%%%%%%%%%%%%%%%%%%%%%%%%%%%%%%%%%%%%%%%%%%%%%%%%%%%%%%%%%
\section{Introduction}

Every connected finite graph has an interesting
isomorphism invariant, called its \emph{critical} or \emph{sandpile group}.
%It is a finite abelian group, the cokernel of
%its (reduced) Laplacian matrix,
This is a finite abelian group, defined as the cokernel of
the (reduced) Laplacian matrix of the graph. Its cardinality is
the number of spanning trees in the graph, and it has
distinguished coset representatives related to
the notion of \emph{chip-firing} on
graphs (\cite{chipfiring}, \cite{sandpile-ag}).
In recent work
motivated by the classical McKay correspondence,
a similar critical group was defined by Benkart, Klivans
and the third author \cite{BenkartKlivansReiner} (and studied
further by Gaetz \cite{Gaetz}) for
complex representations of a finite group.
They showed that the critical group of such a representation
has many properties in common with that of a graph.

The current paper was motivated in trying to understand the role
played by semisimplicity for the group representations.  In fact,
we found that much of the theory generalizes not only to
arbitrary finite group representations in any characteristic,
but even to representations of finite-dimensional Hopf algebras\footnote{And even
further to finite tensor categories, 
although we will not emphasize this;  
see Remark~\ref{tensor-cat-remark} below.}.

Thus we start in Section~\ref{algebras-and-modules-section}
by reviewing modules $V$ for a Hopf algebra $A$ which is finite-dimensional
over an algebraically closed field $\FF$.
This section also defines the critical group $K(V)$ as follows:
if $n:=\dim V$, and if $A$ has $\ell+1$ simple modules, 
then the cokernel of the map $L_V$ on the Grothendieck group
$G_0(A) \cong \ZZ^{\ell+1}$ which multiplies by $n-[V]$ has abelian group
structure $\ZZ \oplus K(V)$.

To develop this further, in Section~\ref{eigenspaces-section} 
we show that the vectors in $\ZZ^{\ell+1}$ giving the
dimensions of the simple and indecomposable projective $A$-modules
are left- and right-nullvectors for
the map $L_V$.  In the case of a group algebra $A=\FF G$ 
for a finite group $G$, we extend results from
\cite{BenkartKlivansReiner} and show
that the columns in the \emph{Brauer character tables} for the
simple and indecomposable projective modules give
complete sets of left- and right-eigenvectors for $L_V$.

Section \ref{left-regular-section} uses this to prove the
following generalization of a result of Gaetz \cite[Ex. 9]{Gaetz}. 
Let $d:=\dim A$, and let
$\gamma$ be the greatest common divisor of the dimensions of the $\ell+1$
indecomposable projective $A$-modules.

\begin{theorem}
\label{regular-rep-theorem}
%Let $d := \dim A$.
If $\ell=0$ then $K(A)=0$, else
$
K(A) \cong \left( \ZZ/\gamma\ZZ \right)
             \oplus \left( \ZZ/d\ZZ \right)^{\ell-1}.
$
\end{theorem}

Section ~\ref{Lorenzini-Gaetz-section}
proves the following formula for 
$\# K(V)$, analogous to one for critical groups of graphs.

\begin{theorem}
\label{Lorenzini-style-formula}
%Let $d := \dim A$.
Assume $K(V)$ is finite, so that $L_V$ has nullity one.  If
the characteristic polynomial of $L_V$ factors as 
$\det(xI-L_V)=x\prod_{i=1}^\ell (x-\lambda_i)$, then 
$
\card{K(V)} = \left| \frac{\gamma}{d} 
                     (\lambda_1 \lambda_2 \cdots \lambda_\ell)\right|.
$
\end{theorem}

Section~\ref{Lorenzini-Gaetz-section} makes this
much more explicit in the case of a group algebra
$\FF G$ for a finite group $G$,
generalizing another result of Gaetz \cite[Thm. 3(i)]{Gaetz}.
Let $p\ge0$ be the characteristic of the field $\FF$.
Let $p^a$ be the order of the $p$-Sylow subgroups of $G$
(with $p^a$ to be understood as $1$ if $p = 0$),
and denote by $\chi_V(g)$ the \emph{Brauer character} value for $V$ on
a $p$-regular element $g$ in $G$;
see Section~\ref{eigenspaces-section} for definitions.

\begin{corollary}
\label{Gaetz-formula}
For any $\FF G$-module $V$ of dimension $n$ with $K(V)$ finite,
one has 
$$
\card{K(V)} = \frac{p^a}{\card{G}} 
\displaystyle\prod_{g \neq e} \left(n-\chi_V(g)\right),
$$
where the product runs through a set of representatives $g$ for
the non-identity $p$-regular $G$-conjugacy classes.
In particular, the quantity on the right is a positive integer.
\end{corollary}

The question of when the abelian group $K(V)$ is \emph{finite},
as opposed to having a free part, occupies
Section~\ref{M-matrix-section}.  
The crucial condition is a generalization of \emph{faithfulness}
for semisimple finite group representations:
one needs the $A$-module $V$ to
be \emph{tensor-rich} in the sense that
every simple $A$-module occurs in at least one
of its tensor powers $V^{\otimes k}$.  In fact, we show
that tensor-richness implies something much stronger about the
map $L_V$:  its submatrix $\overline{L_V}$ obtained by
striking out the row and column indexed by the trivial
$A$-module turns out to be a \emph{nonsingular $M$-matrix},
that is, the inverse $\left(\overline{L_V}\right)^{-1}$ 
has all nonnegative entries.  

\begin{theorem}
\label{tensor-rich-equivalences-theorem}
The following are equivalent for an $A$-module $V$.
\begin{enumerate}
\item[(i)] $\overline{L_V}$ is a nonsingular $M$-matrix. 
\item[(ii)] $\overline{L_V}$ is nonsingular.
\item[(iii)] $L_V$ has rank $\ell$, so nullity $1$.
\item[(iv)] $K(V)$ is finite.
\item[(v)] $V$ is tensor-rich.
\end{enumerate}
\end{theorem}

The question of which $A$-modules $V$ are
tensor-rich is answered completely for group algebras $A=\FF G$ 
via a result of Brauer in Section~\ref{group-algebra-section}.
We suspect that the many questions on
finite-dimensional Hopf algebras raised here 
(Questions~\ref{quest.left-eigenvectors.A},
\ref{gcd-question},
\ref{Sylow-subalgebra-question},
\ref{tensor-rich-questions})
have good answers in general, 
not just for group algebras.

\subsection{Notations and standing assumptions}
\label{subsect.conventions}
Throughout this paper, $\FF$ will be an algebraically closed field,
and $A$ will be a finite-dimensional algebra over $\FF$.
Outside of 
Section~\ref{subsect.findim-alg}, we will further 
assume that $A$ is a Hopf algebra.
We denote by $\dim V$
the dimension of an $\FF$-vector space $V$.
Only finite-dimensional $A$-modules $V$ will
be considered.
All tensor products are over $\FF$.  

Vectors $v$ in $R^m$ for various rings $R$ are regarded as column
vectors, with $v_i$ denoting their $i^{th}$ coordinate.  
The  $(i,j)$ entry of a matrix $M$ will be denoted $M_{i,j}$.
(Caveat lector: Most of the matrices appearing in this paper belong
to $\ZZ^{m \times m'}$ or $\CC^{m \times m'}$, even when they are
constructed from $\FF$-vector spaces. In particular, the rank of such
a matrix is always understood to be its rank over $\QQ$ or $\CC$.)

Let $S_1,S_2,\ldots,S_{\ell+1}$ (resp., $P_1,P_2,\ldots,P_{\ell+1}$)
be the inequivalent simple (resp., indecomposable projective)
$A$-modules, with $\top(P_i):=P_i/\rad P_i=S_i$.
Define two vectors $\sss$ and $\ppp$ in $\ZZ^{\ell+1}$ as follows:
\[
\begin{aligned}
\sss&:=[ \dim(S_1),\ldots,\dim(S_{\ell+1}) ]^T,\\
\ppp&:=[ \dim(P_1),\ldots,\dim(P_{\ell+1}) ]^T.
\end{aligned}
\]

%%%%%%%%%%%%%%%%%%%%%%%%%%%%%%%%%%%%%%%%%%%%%%%%%%%%%%%%%
\section{Finite-dimensional Hopf algebras}
\label{algebras-and-modules-section}

\subsection{Finite-dimensional algebras}
\label{subsect.findim-alg}

Let $A$ be a finite-dimensional algebra over an algebraically closed 
field $\FF$.  Unless explicitly mentioned otherwise, we will only
consider left $A$-modules $V$, with $\dim V:=\dim_\FF V$ finite,
and all tensor products $\otimes$ will be over the field $\FF$.
We recall several facts about such modules; see, e.g., 
Webb \cite[Chap. 7]{Webb} and particularly \cite[Thm. 7.3.9]{Webb}.
The left-regular $A$-module $A$ has a decomposition
\begin{equation}
\label{left-regular-decomposition}
A
\cong \bigoplus_{i=1}^{\ell+1} P_i^{\dim{S_i}}.
\end{equation}
For an $A$-module $V$, if $[V: S_i]$ denotes the multiplicity of $S_i$ as a 
composition factor of $V$, then  
\begin{equation}
\label{eq.alg-reps.V:Si}
[V:S_i] = \dim \Hom_A ( P_i, V ).
\end{equation}

There are two \emph{Grothendieck groups}, $G_0(A)$ and $K_0(A)$:
\begin{itemize}
\item The first one, $G_0(A)$, is defined as the 
quotient of the free abelian group on the set of
all isomorphism classes $[V]$ of $A$-modules $V$,
subject to the relations $[U]-[V]+[W]$ for each
short exact sequence $0 \to U \to V \to W \to 0$
of $A$-modules. This group has a $\ZZ$-module basis consisting of
the classes $[S_1],\ldots,[S_{\ell+1}]$, due to the Jordan-H\"older
theorem.
\item The second one, $K_0(A)$, is defined as the
quotient of the free abelian group on the set of all isomorphism
classes $[V]$ of {\bf projective} $A$-modules $V$,
subject to the relations $[U]-[V]+[W]$ for each
direct sum decomposition $V=U \oplus W$
of $A$-modules. This group has a $\ZZ$-module basis consisting
of the classes $[P_1],\ldots,[P_{\ell+1}]$, due to the 
Krull-Remak-Schmidt theorem.
\end{itemize}

Note that \eqref{left-regular-decomposition} 
implies the following.

\begin{proposition} 
\label{K0-expansion-of-regular-rep}
For a finite-dimensional algebra $A$ over an algebraically
closed field $\FF$, in $K_0(A)$, the class $[A]$
of the left-regular $A$-module has the expansion
$
[A] =\sum_{i=1}^{\ell+1} (\dim S_i) [P_i].
$
\end{proposition}

The two bases of $G_0(A)$ and $K_0(A)$ give rise to
group isomorphisms $G_0(A) \cong \ZZ^{\ell+1} \cong K_0(A)$.
There is also a 
$\ZZ$-bilinear pairing $K_0(A) \times G_0(A) \rightarrow \ZZ$
induced from $\langle [P], [S] \rangle :=\dim \Hom_A(P,S)$. This
is a perfect pairing 
since the $\ZZ$-basis elements satisfy
\[
\langle [P_i], [S_j] \rangle =\dim \Hom_A(P_i,S_j)= [S_j:S_i]
= \delta_{i,j}
\]
(where \eqref{eq.alg-reps.V:Si} was used for the second equality). 
\begin{verlong}
More generally,
\begin{equation}
\langle [P_i], [V] \rangle = \dim \Hom_A (P_i, V) = [V : S_i]
\label{eq.alg-reps.V:Si2}
\end{equation}
for any $A$-module $V$.
\end{verlong}
There is also a $\ZZ$-linear map $K_0(A) \rightarrow G_0(A)$ 
which sends the class $[P]$ of a projective $A$-module $P$
in $K_0(A)$ to the class $[P]$
in $G_0(A)$.  This map is expressed in the usual bases 
by the \emph{Cartan matrix} $C$ of $A$; this is the integer
$\left(\ell+1\right)\times\left(\ell+1\right)$-matrix having entries
\begin{equation}
\label{Cartan-matrix-definition}
C_{i,j}:= [ P_j: S_i]=\dim \Hom_A(P_i,P_j).
\end{equation}
If one chooses orthogonal idempotents $e_i$ in $A$
for which $P_i \cong Ae_i$ as $A$-modules, then
one can reformulate
\begin{equation}
\label{Cartan-matrix-reformulation}
C_{i,j}= \dim \Hom_A(P_i,P_j)
= \dim \Hom_A \left(A e_i, A e_j\right)
= \dim \left(e_i A e_j\right)
\end{equation}
where the last equality used the isomorphism
$\Hom_A\left(Ae, V\right) \cong eV$ sending $\varphi \mapsto \varphi(e)$,
for any $A$-module $V$ and any idempotent $e$ of $A$;
see, e.g., \cite[Prop. 7.4.1 (3)]{Webb}.

Taking dimensions of both sides in
\eqref{left-regular-decomposition}
identifies the dot product of $\sss$ and $\ppp$.

\begin{proposition} 
\label{prop.important-dot-product}
If $A$ is a finite-dimensional algebra over an algebraically
closed field, then $\sss^T \ppp=\dim(A)$.
\end{proposition}

\begin{verlong}
\begin{proof}
From $\sss = [ \dim(S_1),\ldots,\dim(S_{\ell+1}) ]^T$ and
$\ppp = [ \dim(P_1),\ldots,\dim(P_{\ell+1}) ]^T$, we obtain
\begin{align*}
\sss^T \ppp
&= \sum_{i=1}^{\ell+1} \dim\left(S_i\right) \dim\left(P_i\right)
= \dim \left(\bigoplus_{i=1}^{\ell+1} P_i^{\dim{S_i}}\right)
= \dim A
\qquad \left(\text{by \eqref{left-regular-decomposition}}\right) .
\qedhere
\end{align*}
\end{proof}
\end{verlong}

On the other hand, the definition \eqref{Cartan-matrix-definition}
of the Cartan matrix $C$ immediately yields the following:

\begin{proposition} 
\label{prop.AA.cartan.2}
If $A$ is a finite-dimensional algebra over an algebraically
closed field, then $\ppp^T=\sss^T C$.
\end{proposition}

\subsection{Hopf algebras}
\label{Hopf-algebra-subsection}

Let $A$ be a finite-dimensional 
\emph{Hopf algebra} over an algebraically closed field $\FF$, with 
\begin{itemize}
\item counit $\epsilon:A\to\FF$, 
\item coproduct $\Delta: A \to A \otimes A$,
\item antipode $\alpha:A\to A$.
\end{itemize}
%
%The cocommutativity of $A$ implies $S^2=\mathrm{id}$.
%

\begin{example}
Our main motivating example is the \emph{group algebra}
$A=\FF G=\{\sum_{g \in G} c_g g: c_g \in \FF\}$, 
for a finite group $G$, with $\FF$ of arbitrary characteristic.
For $g$ in $G$, the corresponding basis
element $g$ of $\FF G$ has 
\[
\begin{aligned}
\epsilon(g)&=1,\\
\Delta(g)&=g \otimes g,\\
\alpha(g)&=g^{-1}.
\end{aligned}
\]
\end{example}

\begin{example}
\label{Taft-algebra-example}
For integers $m, n>0$ with $m$ dividing $n$,
the \emph{generalized Taft Hopf algebra} $A=H_{n,m}$
is discussed in Cibils~\cite{Cibils} and in Li and Zhang \cite{LiZhang}.
As an algebra, it is a skew group ring \cite[Example 4.1.6]{Montgomery}
$$
H_{n,m} =\FF[\ZZ/n\ZZ] \ltimes \FF[x]/(x^m)
$$
for the cyclic group $\ZZ/n\ZZ=\{e,g,g^2,\ldots,g^{n-1}\}$
acting on coefficients in a truncated polynomial algebra $\FF[x]/(x^m)$,
via $gx g^{-1}=\omega^{-1} x$, with
$\omega$ a primitive $n^{th}$ root of unity in $\FF$.
That is, the algebra $H_{n,m}$ is the 
quotient of the free associative $\FF$-algebra on two generators
$g,x$, subject to the relations $g^n=1, x^m=0$ and $xg = \omega gx$.
It has dimension $mn$, with $\FF$-basis 
$\{g^i x^j : 0 \leq i < n \text{ and }0 \leq j <m\}$.

The remainder of its Hopf structure is determined by these choices:
\[
\begin{array}{rclcrcl}
\epsilon(g)&=&1,&            &\epsilon(x)&=&0,\\
\Delta(g)&=&g \otimes g,&    &\Delta(x)&=&1\otimes x + x \otimes g,\\
\alpha(g)&=&g^{-1},& &\alpha(x)&=&-\omega^{-1}g^{-1}x.
\end{array}
\]
\end{example}

%%%%%%%%%%%%%%%
\begin{comment}
\begin{example}
\label{Nichols-Hopf-algebra-example}
The \emph{Nichols Hopf algebra $A=H_{2^{n+1}}$} of dimension $2^{n+1}$
is a twisted group algebra 
$$
H_{2^{n+1}}=\FF[\ZZ/2\ZZ] \ltimes \Lambda[x_1,\ldots,x_n]
$$
with coefficients in an exterior algebra $\Lambda[x_1,\ldots,x_n]$,
where $\ZZ/2\ZZ=\{e,g\}$  acts via $gx_i g^{-1}=-x_i$.
That is, as an algebra, $H_{2^{n+1}}$ is the quotient of the 
free associative $\FF$-algebra
on $g,x_1,\ldots,x_n$, subject to the relations 
$g^2=1, x_i^2=0, x_i x_j=-x_j x_i, gx_i g^{-1}=-x_i$;
see Nichols \cite{Nichols} and Etingof \emph{et al.} \cite[Example 5.5.7]{Tensor}.
The remainder of its Hopf structure is determined by these choices:
\[
\begin{array}{rclcrcl}
\epsilon(g)&=&1,&            &\epsilon(x_i)&=&0,\\
\Delta(g)&=&g \otimes g,&    &\Delta(x_i)&=&1\otimes x_i + x_i \otimes g,\\
\alpha(g)&=&g^{-1}=g,& &\alpha(x_i)&=&-x_ig=gx_i.
\end{array}
\]
In particular, $H_{2^2}$ is isomorphic to the generalized 
Taft Hopf algebra $H_{2,2}$, via $g \mapsto g$ and $x_i \mapsto x$.
\end{example}
\end{comment}
%%%%%%%%%%%%%%%%

\begin{example}
\label{asymmetric-Nichols-Hopf-algebra-example}
Radford defines in \cite[Exercise 10.5.9]{Radford}
a further interesting Hopf algebra,
which we will denote $A(n,m)$.
Let $n > 0$ and $m \geq 0$ be integers such that $n$ is
even and $n$ lies in $\FF^\times$.
Fix a primitive $n^{th}$ root of unity $\omega$ in $\FF$.
As an algebra, $A(n,m)$ is again a skew group ring 
$$
A(n,m) =\FF[\ZZ/n\ZZ] \ltimes \bigwedge_\FF[x_1,\ldots,x_m],
$$
for the cyclic group $\ZZ/n\ZZ=\{e,g,g^2,\ldots,g^{n-1}\}$
acting this time on coefficients in an exterior algebra 
$\bigwedge_\FF[x_1,\ldots,x_m]$,
via $gx_i g^{-1}=\omega x_i$.
That is, $A(n,m)$
is the quotient of the 
free associative $\FF$-algebra
on $g,x_1,\ldots,x_m$, subject to relations 
$g^n=1, x_i^2=0, x_i x_j=-x_j x_i,$ and $gx_i g^{-1}= \omega x_i$.
It has dimension $n 2^m$ and an $\FF$-basis
$
\left\{
g^i x_J:  0 \leq i < n, \  J \subseteq \{1,2,\ldots,m\}
\right\}
$
where $x_J:=x_{j_1} x_{j_2} \cdots x_{j_k}$ if $J=\{j_1 < j_2<\cdots< j_k\}$. 
The remainder of its Hopf structure is determined by these choices:
\begin{align*}
\epsilon\left(g\right) &= 1,
\qquad
& \epsilon\left(x_i\right) &= 0, \\
\Delta\left(g\right) &= g \otimes g,
\qquad
& \Delta\left(x_i\right) &= 1 \otimes x_i + x_i \otimes g^{n/2}, \\
\alpha\left(g\right) &= g^{-1},
\qquad
& \alpha\left(x_i\right) &= - x_i g^{n/2}.
\end{align*}
In the special case where $n=2$, the Hopf algebra $A(2,m)$ is
the \emph{Nichols Hopf algebra} of dimension $2^{m+1}$ defined
in Nichols \cite{Nichols}; see also 
Etingof \emph{et al.} \cite[Example 5.5.7]{Tensor}.
\end{example}

\begin{example}
\label{restricted-Lie-algebra}
When $\FF$ has characteristic $p$, a \emph{restricted 
Lie algebra} is a Lie algebra $\frak{g}$ over $\FF$, together with
a \emph{$p$-operation} $x \mapsto x^{[p]}$ on $\frak g$ satisfying certain
properties; see Montgomery \cite[Defn. 2.3.2]{Montgomery}.  
The \emph{restricted universal enveloping algebra}
$\frak{u}(\frak{g})$ is then the quotient of the usual
universal enveloping algebra $\frak{U}(\frak{g})$ by the two-sided
ideal generated by all elements $x^p-x^{[p]}$ for $x$ in $\frak g$.  
Since this two-sided ideal is also a Hopf ideal,
the quotient $\frak{u}(\frak{g})$ becomes a Hopf algebra over $\FF$.
The dimension of $\frak{u}(\frak{g})$ is $p^{\dim \frak g}$,
as it has a PBW-style $\FF$-basis of monomials 
$\{x_1^{i_1} x_2^{i_2} \cdots x_m^{i_m}\}_{0 \leq i_j <p}$
corresponding to a choice of ordered $\FF$-basis $(x_1,\ldots,x_m)$ 
of $\frak{g}$.
\end{example}

We return to discussing general finite-dimensional
Hopf algebras $A$ over $\FF$.

The counit $\epsilon: A \rightarrow \FF$ gives rise to
the $1$-dimensional \emph{trivial $A$-module} $\epsilon$, which is the
vector space $\FF$ on which $A$ acts through $\epsilon$.
Furthermore, for each $A$-module $V$, we can define its subspace of
\emph{$A$-fixed points}:
\[
V^A:=\{v \in V: av=\epsilon(a)v \text{ for all }a \in A\}.
\]

The coproduct $\Delta$ gives rise to the \emph{tensor product} 
$V\otimes W$ of two $A$-modules $V$ and $W$, defined via 
$a(v\otimes w) := \sum a_1 v \otimes a_2 w$,
using the Sweedler notation $\Delta(a)=\sum a_1 \otimes a_2$ for $a\in A$
(see, e.g., \cite[Sect. 2.1]{Radford} for an introduction to
the Sweedler notation).
With this definition, the canonical isomorphisms 
\begin{equation}
\label{canonical-triv-tensor-isomorphisms}
\epsilon\otimes V \cong V \cong V\otimes \epsilon
\end{equation}
are $A$-module isomorphisms.  The following lemma appears,
for example, as \cite[Prop. 7.2.2]{DascalescuEtAl}.

\begin{lemma}
\label{lem.VtimesA}
Let $V$ be an $A$-module.
\begin{enumerate}
\item[(i)] Then, $V \otimes A \cong A^{\oplus \dim V}$
as $A$-modules.
\item[(ii)] Also, $A \otimes V \cong A^{\oplus \dim V}$
as $A$-modules.
\end{enumerate}
\end{lemma}

The antipode $\alpha: A \rightarrow A$ of the Hopf algebra $A$ is
bijective,
since $A$ is finite-dimensional; see, e.g., \cite[Thm. 2.1.3]{Montgomery},
\cite[Thm. 7.1.14 (b)]{Radford}, \cite[Prop. 4]{Pareigis},
or \cite[Prop. 5.3.5]{Tensor}.  
% [DG] Please keep in the Pareigis ref! It's open-access and proves a
% generalization.
Hence $\alpha$ is an algebra and coalgebra \emph{anti-automorphism}.
\begin{verlong}
In particular, as an algebra, $A \cong A^{\opp}$.
\end{verlong}
For each $A$-module $V$, the antipode gives rise to two
$A$-module structures on $\Hom_\FF(V,\FF)$:
the \emph{left-dual} $V^*$  and the \emph{right-dual} $\leftast{V}$ of $V$.
They are defined as follows:  For $a \in A$, $f \in \Hom_\FF(V,\FF)$ and
$v \in V$, we set
\[
(af)(v) := 
\begin{cases}
f(\alpha(a)v), &\text{ when regarding }f\text{ as an element of }V^*,\\
f(\alpha^{-1}(a)v), &\text{ when regarding }f\text{ as an element of }\leftast{V}. \\
\end{cases}
\]

The following two facts are straightforward exercises in the definitions.
\begin{lemma}
\label{lem.epsilonstar}
We have $A$-module isomorphisms
$\epsilon^* \cong \leftast{\epsilon} \cong \epsilon$.
\end{lemma}

\begin{lemma}
\label{lem.Vstarstar}
Let $V$ be an $A$-module. We have canonical $A$-module isomorphisms
$\leftast{(V^*)} \cong V \cong (\leftast{V})^*$.
\end{lemma}

For any two $A$-modules $V$ and $W$, we define an $A$-module structure
%\footnote{At some point, we re-defined this from the other choice, $(a\varphi)(v) := \sum a_2 \varphi(\alpha(a_1)v)$, so it should be checked how this affects left- and right-duals in other statements later!} 
on
$\Hom_\FF(V,W)$ via 
\[
(a\varphi)(v) := \sum a_1 \varphi(\alpha(a_2)v)
\]
for all $a \in A$, $\varphi \in \Hom_\FF(V,W)$ and $v \in V$.
The following result appears, for example, as \cite[Lemma 2.2]{Witherspoon}.

\begin{lemma}
\label{lem.hom-tensor-iso}
Let $V$ and $W$ be two $A$-modules.
Then, we have an $A$-module isomorphism
\begin{equation}
\label{hom-tensor-iso}
\Phi: W \otimes V^* \overset{\cong}{\rightarrow} \Hom_\FF(V,W)
\end{equation}
sending $w \otimes f$ to the linear map $\varphi \in \Hom_\FF(V,W)$
that is defined by
$\varphi(v)=f(v)w$ for all $v\in V$.

In particular, when $W=\epsilon$, this shows $V^* \cong\Hom_\FF(V,\epsilon)$.
\end{lemma}

Next, we shall use a result that is proven in Schneider
\cite[Lemma 4.1]{Schneider}\footnote{Schneider makes various
assumptions that are not used in the proof.}

\begin{lemma}\label{lem:AHom}
Let $V$ and $W$ be two $A$-modules. Then,
$\Hom_A (V, W) = \Hom_\FF (V, W)^A$.
\end{lemma}

The next four results are proven in Appendix~\ref{Hopf-appendix}.

\begin{lemma}\label{lem:AHom2}
Let $V$ and $W$ be two $A$-modules. Then,
$\Hom_A (V, W) \cong \Hom_A \left(W^* \otimes V, \epsilon\right)$.
\end{lemma}

\begin{lemma}\label{lem:tensor-dual}
Let $U$ and $V$ be $A$-modules.
Then, $(U\otimes V)^* \cong V^*\otimes U^*$ and
$\leftast{(U\otimes V)} \cong \leftast{V} \otimes \leftast{U}$.
\end{lemma}

\begin{lemma}
\label{lem:tensor-dual-four-isos}
For $A$-modules $U$, $V$, and $W$, one has isomorphisms 
\begin{eqnarray}
\label{eq:TensorDual1}
\Hom_A(U\otimes V, W) \xrightarrow\sim \Hom_A(U, W\otimes V^*),\\
\label{eq:TensorDual2}
\Hom_A(V^*\otimes U, W) \xrightarrow\sim \Hom_A(U, V\otimes W),\\
\label{eq:TensorDual3}
\Hom_A(U\otimes \leftast{V}, W) \xrightarrow\sim \Hom_A(U, W\otimes V),\\
\label{eq:TensorDual4}
\Hom_A(V\otimes U, W) \xrightarrow\sim \Hom_A(U, \leftast{V}\otimes W).
\end{eqnarray}
\end{lemma}

\begin{proposition} \phantomsection \label{prop.AA}
Any $A$-module $V$ has $\dim \Hom_A\left(V, A\right) = \dim V$.
\end{proposition}

Proposition~\ref{prop.AA} implies the following two Hopf algebra facts,
to be compared with the two ``transposed'' algebra facts, Propositions~\ref{K0-expansion-of-regular-rep} and \ref{prop.AA.cartan.2}.

\begin{corollary} \label{cor.AA.cartan}
Let $A$ be a finite-dimensional Hopf algebra over an algebraically
closed field $\FF$. Let $P_i$, $S_i$, $\ppp$, $\sss$ and $C$ be as in
Subsection~\ref{subsect.findim-alg}.
\begin{enumerate}
\item[(i)] The class $[A]$ of the left-regular $A$-module expands in $G_0(A)$
as 
$
[A] =\sum_{i=1}^{\ell+1} (\dim P_i) [S_i].
$
\item[(ii)] The Cartan matrix $C$ has $C \sss = \ppp$. 
\end{enumerate}
\end{corollary}

\begin{proof}
The assertion in (i) follows by noting that for each 
$i=1,2,\ldots,\ell+1$, one has
$$
\left[A : S_i\right] = \dim \Hom_A (P_i, A) = \dim\left(P_i\right),
$$
where the first equality applied  \eqref{eq.alg-reps.V:Si}
and the second equality applied Proposition~\ref{prop.AA} with $V = P_i$.

This then helps to deduce assertion (ii), since 
\begin{noncompile}
% was verlong
\begin{equation}
\bigoplus_{j=1}^{\ell+1} P_j^{\dim{S_j}}
= \bigoplus_{i=1}^{\ell+1} P_i^{\dim{S_i}}
\cong A
\qquad \left(\text{by \eqref{left-regular-decomposition}}\right) .
\label{pf.cor.AA.cartan.decomp}
\end{equation}
\end{noncompile}
for each $i =1,2,\ldots \ell+1$, one has
$$
\left(C \sss\right)_i
%&= \sum_{j=1}^{\ell+1}
%        \underbrace{C_{i, j}}_{= \left[P_j : S_i\right]}
%        \underbrace{\sss_j}_{= \dim S_j}
= \sum_{j=1}^{\ell+1} C_{ij} \sss_j
= \sum_{j=1}^{\ell+1} \left[P_j : S_i\right] \dim S_j
= \left[ \bigoplus_{j=1}^{\ell+1} P_j^{\dim{S_j}} : S_i \right]
= \left[ A : S_i \right]
= \ppp_i 
$$
where the second-to-last equality used \eqref{left-regular-decomposition},
and the last equality is assertion (i).
Thus, $C \sss = \ppp$.
\end{proof}

Note that Corollary~\ref{cor.AA.cartan} (ii) follows from
Proposition~\ref{prop.AA.cartan.2} whenever the Cartan matrix $C$ is
symmetric. However, $C$ is not always symmetric, as illustrated by the following example.

\begin{example}
\label{asymmetric-Nichols-Cartan-example}
Consider Radford's Hopf algebra $A=A(n,m)$  from 
Example~\ref{asymmetric-Nichols-Hopf-algebra-example}, 
whose algebra structure is the skew group ring $\FF[\ZZ/n\ZZ] \ltimes \bigwedge_\FF[x_1,\ldots,x_m]$. In this case, it is not hard to see that the radical of $A$ is the two-sided ideal $I$
generated by $x_1,\ldots,x_m$, with $A/I \cong \FF[\ZZ / n \ZZ]$, and that $A$ has a system of orthogonal primitive idempotents
$
\left\{ e_k:=\frac{1}{n} \sum_{i=0}^{n-1} \omega^{ki} g^i \right\}_{k=0,1,\ldots,n-1},
$
where the subscript $k$ can be regarded as an element of $\ZZ/n\ZZ$.
This gives $n$ 
indecomposable projective $A$-modules
$\{P_k\}_{k=0,1,\ldots,n-1}$ with $P_k \cong Ae_k$, 
whose corresponding simple $A$-modules $\{S_k\}_{k=0,1,\ldots,n-1}$  
are the simple modules for the cyclic group algebra 
$A/I \cong \FF[\ZZ / n \ZZ]$, regarded as $A$-modules by
inflation.  

We compute here the Cartan matrix $C$ for $A$,
using the formulation $C_{i, j}=\dim \left(e_i A e_j\right)$
from \eqref{Cartan-matrix-reformulation}.  Recall 
that $A$ has $\FF$-basis
$
\left\{
g^k x_J:  0 \leq k < n, \  J \subseteq \{1,2,\ldots,m\}
\right\}.
$
Using the fact that the $e_0,\ldots,e_{n-1}$ are orthogonal idempotents,
and easy calculations such as
$e_i g^k = \omega^{-ki} e_i$
%g^k x_J &= x_J \omega^{k\#J}, \\
and 
$x_J e_j = e_{j-\#J} x_J$,
one concludes that
$$
e_i \left( g^k x_J \right) e_j 
= \omega^{-ki} e_i x_J e_j 
= \omega^{-ki} e_i e_{j-\#J} x_J
=
\begin{cases} 
\omega^{-ki} e_i x_J, & \text{if }i \equiv j-\#J \bmod{n}, \\
0,&  \text{otherwise.}
\end{cases}
$$
Therefore 
$
C_{i, j}
=\dim \left(e_i A e_j\right)
= \#\{ J \subseteq \{1,2,\ldots,m\}: \#J \equiv j-i \bmod{n} \}.
$
This matrix $C$ will not be symmetric in general; e.g. for
$n = 4$ and $m=1$, if one indexes rows and columns by $e_0,e_1,e_2,e_3$,
then
$
C=\left[\begin{smallmatrix}
 1 & 1 & 0 & 0\\
 0 & 1 & 1 & 0\\
 0 & 0 & 1 & 1\\
 1 & 0 & 0 & 1\\
\end{smallmatrix}\right].
$
\end{example}

% \begin{question} Obsolete question (answered above):
% \label{cartan-symmetry-question}
% Does every finite-dimensional Hopf algebra $A$ over an algebraically
% closed field $\FF$
% have symmetric Cartan matrix $C=C^T$?
% Equivalently, is $\dim (eAf) = \dim (fAe)$ for
% all idempotents $e, f$ in $A$ ?
% \end{question}

% The examples in \cite{Lorenz} seem to suggest that this is true, but
% a proof has eluded us. 

\subsection{The Grothendieck ring and the critical group}

The group $\groth$ also has an associative (not necessarily commutative)
augmented $\ZZ$-algebra structure:
\begin{itemize}
\item the multiplication is induced from $[V]\cdot [W]:=[V \otimes W]$
(which is well-defined, since the tensor bifunctor over $\FF$ is exact,
and is associative since tensor products are associative),
\item the unit element is $1=[\epsilon]$, 
the class of the trivial $A$-module $\epsilon$, and
\item the augmentation (algebra) map 
$\groth \rightarrow \ZZ$ is induced
from $[V] \mapsto \dim(V)$. 
\end{itemize}
In many examples that we consider, $A$ will be cocommutative,
so that $V \otimes W \cong W \otimes V$, and hence $\groth$ is also
commutative.  However, Lemma~\ref{lem:tensor-dual} shows that there is
a ring homomorphism $\groth \to \groth^{\opp}$ sending each $[V]$
to $[V^*]$.
Lemma~\ref{lem.Vstarstar} furthermore shows that this homomorphism is
an isomorphism. Thus, $\groth \cong \groth^{\opp}$ as rings.
% antipode $\alpha$ gives an algebra and coalgebra anti-automorphism of $A$, it also 
% gives rise to an \emph{algebra} anti-automorphism
% of $\groth$, that is $\groth \cong \groth^{\opp}$.
Consequently, when discussing
constructions involving $\groth$ that involve multiplication on the right,
we will omit the discussion of the same construction on the left.

The kernel $I$ of the augmentation map, defined by
the short exact sequence
\begin{equation}
\label{augmentation-ideal-sequence}
0 \to I \longrightarrow  \groth \longrightarrow  \ZZ \to 0,
\end{equation}
is the (two-sided) \emph{augmentation ideal} of $\groth$.
Recalling that the vector $\sss$ gave the dimensions of the
simple $A$-modules, then 
under the additive isomorphism $\groth \cong \ZZ^{\ell+1}$,
the augmentation map $\groth \cong \ZZ^{\ell+1} \to \ZZ$
corresponds to the map $\xx \mapsto \sss^T \xx$ that takes 
dot product with $\sss$.  Therefore
the augmentation ideal $I \subset \groth$ corresponds
to the perp sublattice 
$$
I=\sss^\perp:=\{ \xx \in \ZZ^{\ell+1} : \sss^T \xx = 0\}.
$$

We come now to our main definition.

\begin{definition}
Given an $A$-module $V$ of dimension $n$, define its
\textit{critical group}
%\footnote{also called its \textit{sandpile group}}
as the quotient (left-)$\groth$-module of $I$
modulo the principal (left-)ideal generated by $n-[V]$:
\[
K(V):= \left. I \middle/ \groth(n-[V]) \right. .
\]
\end{definition}

We are interested in the abelian group structure of
$K(V)$, which has some useful matrix reformulations.
First, note that the short exact sequence of abelian groups 
\eqref{augmentation-ideal-sequence}
is split, since $\ZZ$ is free abelian.  This gives
a direct decomposition
$
\groth = \ZZ \oplus I
$
as abelian groups,
which then induces a decomposition
\[
\left. \groth \middle/ 
%\left(
 \groth(n-[V]) 
%\right) 
\right.= \ZZ \oplus K(V).
\]

Second, note that in the ordered $\ZZ$-basis $([S_1],\ldots,[S_{\ell+1}])$
for $\groth$, one expresses multiplication on the right by $[V]$
via the \emph{McKay matrix} 
$M=M_V$ in $\ZZ^{(\ell+1) \times (\ell+1)}$
where
$
M_{i,j}=[S_j \otimes V: S_i].
$
Consequently multiplication on the right by $n-[V]$
is expressed by the matrix $L_V:=n I_{\ell+1}-M_V$.
Thus the abelian group structure of $K(V)$
can alternately be described in terms of the cokernel of $L_V$:
\begin{align}
\ZZ \oplus K(V) 
     &\cong \ZZ^{\ell+1} / \im L_V,
\label{eq.K(V).def2}     \\
K(V) &\cong \sss^\perp / \im L_V.
\label{eq.K(V).def3}
\end{align}

We will sometimes be able to reformulate $K(V)$ further as the
cokernel of an $\ell \times \ell$ submatrix of $L_V$ 
(see the discussion near the end of Section~\ref{M-matrix-section}).
For this and other purposes, it is important to 
know about the left- and right-nullspaces of $L_V$, explored next.

%%%%%%%%%%%%%%%%%%%%%%%%%%%%%%%%%%%%%%%%%%%%%%%%%%%%%%%%%
\section{Left and right eigenspaces}
\label{eigenspaces-section}

A goal of this section is to record the observation that, for any $A$-module $V$, 
the vectors $\sss$ and $\ppp$ introduced earlier
are always left- and right-eigenvectors for $M_V$, both
having eigenvalue $n=\dim(V)$,
and hence left- and right-nullvectors for $L_V=nI_{\ell+1}-M_V$.
When $A=\FF G$ is the
group algebra of a finite group $G$,
we complete this to a full set of left- and right-eigenvectors
and eigenvalues: the eigenvalues of $M_V$ turn out to be 
the \emph{Brauer character values}
$\chi_V(g)$, while the left- and right-eigenvectors are the
\emph{columns of the Brauer character table} for the simple $A$-modules
and indecomposable projective $A$-modules, respectively.  This interestingly
generalizes a well-known story from the McKay correspondence
in characteristic zero; see \cite[Prop. 5.3, 5.6]{BenkartKlivansReiner}.

Let us first establish terminology: 
a \emph{right-eigenvector} (resp. \emph{left-eigenvector})
of a matrix $U$ is a vector $v$ such that $Uv = \lambda v$ 
(resp. $v^T U = \lambda v^T$) for some scalar $\lambda$;
notions of left- and right-nullspaces and
left- and right-eigenspaces should be interpreted similarly.

We fix an $A$-module $V$ throughout Section~\ref{eigenspaces-section};
we set $n = \dim(V)$.

\subsection{Left-eigenvectors}
Left-eigenvectors of $M_V$ and $L_V$ will arise from the simple $A$-modules.

\begin{proposition}
\label{left-nullvector-prop}
The vector $\sss$ is a left-eigenvector for $M_V$ with eigenvalue $n$, and
a left-nullvector for $L_V$.
\end{proposition}
\begin{proof}
Letting $M:=M_V$, for each $j=1,2,\ldots,\ell+1$, one has
\[
n \sss_j
=\dim(S_j) \dim(V) 
 = \dim (S_j \otimes V)
 = \sum_{i=1}^{\ell+1} [S_j \otimes V: S_i] \dim(S_i)
 = \sum_{i=1}^{\ell+1}\dim(S_i) M_{i,j} 
 = (\sss^T M)_j.
\qedhere
\]
\end{proof}

The full left-eigenspace decomposition for $M_V$ and $L_V$,  
when $A=\FF G$ is a group algebra, requires the notions of $p$-regular elements and Brauer characters, recalled here.

\begin{definition}
Recall that for a finite group $G$ and a field $\FF$ of characteristic $p\ge0$, an element $g$ in $G$ 
is \emph{$p$-regular} if its multiplicative order lies in $\FF^\times$.
That is, $g$ is $p$-regular if it has
order coprime to $p$ when $\FF$ has characteristic $p > 0$,
and \emph{every} $g$ in $G$ is $p$-regular when $\FF$ has characteristic $p=0$.  
Let $p^a$ be the order of the $p$-Sylow subgroups of $G$,
so that $\card{G}=p^a q$ with $\gcd(p,q)=1$.
(In characteristic zero, set $p^a := 1$ and $q:=\# G$.)
The order of any $p$-regular element of $G$ divides $q$.

To define Brauer characters for $G$, one first
fixes a (cyclic) group isomorphism $\lambda \mapsto \widehat{\lambda}$ between
the $q^{th}$ roots of unity in the algebraic closure
$\overline{\FF}$ of $\FF$ and
the  $q^{th}$ roots of unity in $\CC$.  Then for
each $\FF G$-module $V$ of dimension $n$, and each $p$-regular element $g$ in $G$,
the \emph{Brauer character} value $\chi_V(g) \in \CC$ can be defined as follows.
Since $g$ is $p$-regular, it will act semisimply on $V$ by Maschke's theorem,
and have eigenvalues $\lambda_1,\lambda_2,\ldots,\lambda_n$ in 
$\overline{\FF}$ which are $q^{th}$ roots
of unity when acting on $V$ 
(or, strictly speaking, when $1 \otimes g$
acts on $\overline{\FF} \otimes_\FF V$).  This lets one define
$
\chi_V(g) := \sum_{i=1}^{n} \widehat{\lambda_i},
$
using the isomorphism fixed earlier.
This $\chi_V(g)$ depends only on the conjugacy class of $g$ (not on
$g$ itself), and so is also called the \textit{Brauer character value}
of $V$ at the conjugacy class of $g$.
\end{definition}

Brauer showed  \cite[Theorem 9.3.6]{Webb} that the 
number $\ell+1$ of simple $\FF G$-modules is the same as the number
of $p$-regular conjugacy classes.
He further showed 
that the map sending $V \longmapsto \chi_V$ induces a ring isomorphism
from the Grothendieck ring $G_0(A)$ to 
$\CC^{\ell+1}$, where $\CC^{\ell+1}$ is the ring of $\CC$-valued class functions 
on the $\ell+1$ distinct $p$-regular $G$-conjugacy classes, under
pointwise addition and multiplication; see \cite[Prop. 10.1.3]{Webb}. 
%This ring homomorphism is furthermore injective \cite[Thm. 10.2.2]{Webb} (although we shall not use this).
One has the accompanying notion of the
\emph{Brauer character table} for $G$, 
an invertible $(\ell+1) \times (\ell+1)$
matrix \cite[Theorem 10.2.2]{Webb}
having columns indexed by the $p$-regular conjugacy classes of $G$,
rows indexed by the simple $\FF G$-modules $S_i$, 
and entry $\chi_{S_i}(g_j)$ in the row for $S_i$ and column indexed by 
the conjugacy class of $g_j$.

\begin{definition}
Given a $p$-regular element $g$ in $G$, let $\sss(g)=[ \chi_{S_1}(g),\ldots,\chi_{S_{\ell+1}}(g)]^T$
be the Brauer character values of the simple $\FF G$-modules at $g$, that is,
the column indexed by the conjugacy class of $g$ in the
Brauer character table of $G$.
In particular, $\sss(e)=\sss$, where $e$ is the identity element of $G$.
\end{definition}

\begin{proposition}
\label{left-eigenvector-prop}
For any $p$-regular element $g$ in $G$, the vector $\sss(g)$ is a left-eigenvector for $M_V$ and $L_V$,
with eigenvalues $\chi_V(g)$ and $n-\chi_V(g)$, respectively.
\end{proposition}
\begin{proof}
Generalize the calculation from Proposition~\ref{left-nullvector-prop},
using the fact that $[V] \mapsto \chi_V$ is a ring map:
\[
\chi_V(g) \cdot \sss(g)_j
=\chi_V(g) \chi_{S_j}(g)
 = \chi_{S_j \otimes V}(g)
 = \sum_{i=1}^{\ell+1} [S_j \otimes V: S_i] \chi_{S_i}(g)
 = \sum_{i=1}^{\ell+1} \sss(g)_i M_{i,j} 
 = (\sss(g)^T M)_j.
\qedhere
\]
\end{proof}

\subsection{Right-eigenvectors}
Right-eigenvectors for $L_V$ and $M_V$ will come from the indecomposable
projective $A$-modules, as we will see below
(Proposition~\ref{right-nullvector-prop} and, for group algebras, the
stronger Proposition~\ref{right-eigenvector-prop}). First, we shall
show some lemmas.

\begin{lemma}\label{lem:DualSym}
For any $A$-module $V$, and $i,j\in\left\{1,2,\ldots,\ell+1\right\}$,
one has
\begin{equation}\label{eq.lem:DualSym.1}
[ V \otimes P_j^* : S_i ] = [\,\leftast{\!P_i} \otimes V : S_j ]. 
\end{equation}
In particular, taking $V = \epsilon$ gives a ``dual symmetry'' for the Cartan matrix $C$ of $A$:
\begin{equation}\label{eq2}
[ P_j^* : S_i ] = [\,\leftast{\!P_i} : S_j ].
\end{equation}
\end{lemma}

\begin{proof}
The result follows upon taking dimensions
in the following consequence of \eqref{eq:TensorDual1} and \eqref{eq:TensorDual4}:
\[ 
\Hom_A (P_i, V \otimes P_j^*) \cong \Hom_A ( P_i\otimes P_j, V )
\cong \Hom_A (P_j, \leftast{\!P_i}\otimes V).  \qedhere
\]
%\[ \Hom_A (P_j \otimes \leftast{V}, \leftast{P_i}) = \Hom_C ( P_i\otimes P_j, V ) = \Hom_C (V^*\otimes P_i, P_j^*) \]
\end{proof}

\begin{lemma}
\label{lem:TensorProj}
The following equality holds in the Grothendieck group $G_0(A)$
for any $[V] \in G_0(A)$:
\[ [V\otimes P_j^*] = \sum_{i=1}^{\ell+1}
[ S_i\otimes V : S_j ] [P_i^*], \quad \forall j\in\left\{1,2,\ldots,\ell+1\right\}.\]
\end{lemma}

\begin{proof}
%We freely use $[U:S_k]=\langle [P_k],[U]\rangle$, 
%and pair $\langle [P_k], - \rangle$ with the left side in the lemma for each $k$:
%and the biexactness of the tensor functor
By \eqref{eq.lem:DualSym.1}, the multiplicity of $S_k$ in the left hand side is
\begin{equation}
\label{start-of-inner-prod-calculation}
%\langle [P_k],V\otimes P_j^*\rangle =
[ V\otimes P_j^* : S_k ]
= [\,\leftast{\!P_k} \otimes V : S_j ].
%= \langle [P_j], [\,\leftast{\!P_k} \otimes V] \rangle.
\end{equation}
However, one also has in $G_0(A)$ that
$$
[\leftast{\!P_k} \otimes V]
= [\leftast{\!P_k}] \cdot [V]
= \sum_i [\leftast{\!P_k} : S_i] [S_i] \cdot [V]
= \sum_i [\leftast{\!P_k} : S_i] [S_i \otimes V] 
$$
and substituting this into \eqref{start-of-inner-prod-calculation} gives
$$
%\langle [P_k],V\otimes P_j^*\rangle
[ V\otimes P_j^* : S_k ]
%= \sum_i \langle [P_j], [S_i \otimes V] \rangle
= \sum_i [ S_i\otimes V: S_j ]
[\,\leftast{\!P_k} : S_i]
=  \sum_i [S_i \otimes V:S_j] 
  [\ P_i^* : S_k]
$$
where we have used \eqref{eq2} in the last equality.
One can then recognize this last expression as the 
multiplicity of $S_k$ in the right hand side of the desired equation.
%pairing $\langle [P_k],-\rangle$ with the right side in the lemma.
\end{proof}

\begin{lemma}\label{lem:dual-proj-ind}
Any indecomposable projective $A$-module $P_i$
has its left-dual $P_i^*$ and right-dual $\leftast{\!P_i}$ also indecomposable projective.
Consequently, $P_1^*,\ldots,P_{\ell+1}^*$ form a permutation of $P_1,\ldots,P_{\ell+1}$.
\end{lemma}
\begin{proof}
Lemma~\ref{lem.Vstarstar} shows that
$V\mapsto V^*$ is an equivalence of categories
from the category of (finite-dimensional) $A$-modules
to its opposite category.
Since we furthermore have $A$-module isomorphisms 
$\left(\bigoplus_i V_i\right)^* \cong \bigoplus V_i^*$
(for finite direct sums) and similarly for
right-duals,
we thus see that
indecomposability is preserved under taking left-duals.
It is well-known \cite[Prop. 6.1.3]{Tensor} that 
projectivity is preserved under taking left-duals.
Thus $P_i^*$ is also indecomposable projective and so is $\leftast{\!P_i}$ by the same argument.
Then $P_1^*,\ldots,P_{\ell+1}^*$ form a permutation of $P_1,\ldots,P_{\ell+1}$ since $\leftast{(V^*)}\cong V$.
\end{proof}

Since $\dim (P_i) = \dim (P_i^*)$, the definition of
the vector $\ppp$ can be rewritten as
\[
\ppp:=\left[ \dim (P_1), \ldots,\dim (P_{\ell+1}) \right]^T
     =\left[ \dim(P^*_1), \ldots,\dim(P^*_{\ell+1}) \right]^T.
\]

\begin{proposition}
\label{right-nullvector-prop}
This $\ppp$ is a right-eigenvector for $M_V$ with eigenvalue $n$, and
a right-nullvector for $L_V$.
\end{proposition}
\begin{proof}
Letting $M:=M_V$, for each $j=1,2,\ldots,\ell+1$, 
using Lemma~\ref{lem:TensorProj} one has
\begin{align*}
n \ppp_j
&= \dim(V) \dim(P_j^*)
 = \dim \left(V \otimes P^*_j\right)
 = \sum_{i=1}^{\ell+1} [S_i \otimes V: S_j] \dim(P^*_i)
 = \sum_{i=1}^{\ell+1} M_{j,i} \ppp_i
 = (M \ppp)_j.
\qedhere
\end{align*}
\end{proof}

In the case of a group algebra $A=\FF G$, one has
the analogous result to Proposition~\ref{left-eigenvector-prop}.

\begin{definition}
For a $p$-regular $g$ in $G$, let 
$\ppp^*(g)=[ \chi_{P^*_1}(g),\ldots,\chi_{P^*_{\ell+1}}(g)]^T$
be the Brauer character values of the (left-duals of the) 
indecomposable projective $A$-modules $P^*_i$ at $g$.
Note that $\ppp^*(g)$ is
a re-ordering of the column indexed by $g$ in the 
$(\ell+1) \times (\ell+1)$
table of Brauer characters of the indecomposable projective $\FF G$-modules,
whose $(i,j)$-entry is $\chi_{P_i}(g_j)$.
In particular, $\ppp^*(e)=\ppp$, where $e$ is the identity in $G$.
Note that this indecomposable projective Brauer character
table is also an invertible matrix \cite[Theorem 10.2.2]{Webb}.

\end{definition}

\begin{proposition}
\label{right-eigenvector-prop}
This $\ppp^*(g)$ is a right-eigenvector for $M_V$ and $L_V$,
with eigenvalues $\chi_V(g)$ and $n-\chi_V(g)$.
\end{proposition}
\begin{proof}
Generalize the calculation from Proposition~\ref{right-nullvector-prop}
using the fact that $[V] \mapsto \chi_V$ is a ring map:
\[
\chi_V(g) \cdot \ppp^*(g)_j
=\chi_V(g) \chi_{P^*_j}(g)
 = \chi_{V \otimes P^*_j}(g) 
 = \sum_{i=1}^{\ell+1} [S_i \otimes V: S_j] \chi_{P^*_i}(g)
 = \sum_{i=1}^{\ell+1} M_{j,i} \ppp^*(g)_i 
 = (M \ppp^*(g))_j. \qedhere
\]
\end{proof}

\begin{remark}
Note that since the Brauer character tables for the simple $\FF G$-modules
and for the indecomposable projective $\FF G$-modules are both invertible,
Propositions~\ref{left-eigenvector-prop} and~\ref{right-eigenvector-prop} 
yield full bases for $\CC^{\ell+1}$ consisting of 
right-eigenvectors for $M_V$ or $L_V$, and of left-eigenvectors for
$M_V$ or $L_V$.

\begin{question} \label{quest.left-eigenvectors.A}
Are there analogues of 
Propositions~\ref{left-eigenvector-prop}, ~\ref{right-eigenvector-prop} for all
finite-dimensional Hopf algebras? 
\end{question}

In particular, what plays the role of $p$-regular elements,
and Brauer characters?
\end{remark}

\begin{verlong}
One little step towards resolving
Question~\ref{quest.left-eigenvectors.A} is to observe that if
$a$ is a group-like element of a finite-dimensional $\FF$-Hopf algebra
$A$ (that is, $\Delta(a) = a \otimes a$ and $\epsilon(a) = 1$), then
the vector \newline
$\left[ \operatorname{Tr}_{S_1}\left(a\right),
\operatorname{Tr}_{S_2}\left(a\right),
\ldots,
\operatorname{Tr}_{S_{\ell+1}}\left(a\right) \right]^T
\in \FF^{\ell+1}$ (where $\operatorname{Tr}_W \left(a\right)$ stands
for the trace of the action of $a$ on an $A$-module $W$) is an
eigenvector (with eigenvalue $\operatorname{Tr}_V \left(a\right)$)
for the matrix $\left(M_V\right)_\FF$ obtained by mapping the matrix
$M_V \in \ZZ^{\left(\ell+1\right)\times\left(\ell+1\right)}$ into
$\FF^{\left(\ell+1\right)\times\left(\ell+1\right)}$. This gives us
some eigenvectors for this matrix $\left(M_V\right)_\FF$; but we do
not know whether they can be lifted to eigenvectors of
$M_V$, and how far they are away from yielding the diagonalization of
$\left(M_V\right)_\FF$ (if such a diagonalization even exists).
\end{verlong}

\begin{remark}
\label{tensor-cat-remark}
It is perhaps worth noting that 
many of the previous results which we have stated for 
a finite-dimensional Hopf algebra $A$,
including Propositions~\ref{left-nullvector-prop} and \ref{right-nullvector-prop}
on $\sss$ and $\ppp$ as left- and
right-nullvectors for $L_V$, 
hold in somewhat higher generality.
One can replace the category of $A$-modules with 
a \emph{finite tensor category} $\C$,
replace $G_0(A)$ with the Grothendieck ring $G_0(\C)$ of $\C$,
and replace the assignment $V \mapsto \dim V$ for $A$-modules $V$ with
the \emph{Frobenius-Perron dimension} 
as an algebra morphism $\FPdim: G_0(\C) \rightarrow \RR$; 
see \cite[Chapters 1--4]{Tensor}. Most of our arguments mainly use the
existence of left- and right-duals $V^*$ and $\leftast{V}$ for objects $V$ 
in such a category $\C$, and properties of $\FPdim$.

In fact, we feel that, in the same way that
Frobenius-Perron dimension $\FPdim(V)$ is an interesting
real-valued invariant of an object in a tensor category,
whenever $\FPdim(V)$ happens to be an integer, the
critical group $K(V)$ is another interesting invariant
taking values in abelian groups.
\end{remark}

%%%%%%%%%%%%%%%%%%%%%%%%%%%%%%%%%%%%%%%%%%%%%%%%%%%%%%%%%
\section{Proof of Theorem~\ref{regular-rep-theorem}}
\label{left-regular-section}

We next give the structure of the critical group $K(A)$
for the left-regular representation $A$.
We start with a description of its McKay matrix $M_A$
using the Cartan matrix $C$, and the vectors $\sss, \ppp$
from Subsection~\ref{subsect.findim-alg}.

\begin{proposition} \label{prop.MA}
Let $A$ be a finite-dimensional Hopf algebra over an algebraically
closed field $\FF$. Then the McKay matrix $M_A$ of the
left-regular representation $A$ takes the form
$$
M_A = C \sss \sss^T = \ppp \sss^T.
$$
\end{proposition}

\begin{proof}
For every $A$-module $V$ and any
$i \in \left\{1, 2, \ldots, \ell+1\right\}$, we obtain from\footnote{Instead 
of using
Lemma~\ref{lem.VtimesA} (i) here, we could also have used the weaker
result that $[V \otimes A] = \dim(V)[A]$ in $G_0(A)$; this weaker
result has the advantage of being generalizable to tensor categories
\cite[Prop. 6.1.11]{Tensor}.}
Lemma~\ref{lem.VtimesA}(i) the equality
\begin{equation}
\left[V \otimes A : S_i\right]
= \left[A^{\oplus \dim V} : S_i \right]
= (\dim V) \left[A : S_i\right].
\label{pf.prop.MA.1}
\end{equation}
Now, we can compute the entries of the McKay matrix $M_A$:
$$
(M_A)_{i,j}
= [S_j \otimes A:S_i] 
 = \dim(S_j) [A:S_i]                          
= \dim(S_j) \dim\left(P_i\right) = \sss_j \ppp_i
 = \left( \ppp \sss^T \right)_{i,j} 
$$
using \eqref{pf.prop.MA.1} in the second equality,
and Corollary~\ref{cor.AA.cartan} (i) in the third.
Thus $M_A = \ppp \sss^T$ and then $\ppp \sss^T=C \sss \sss^T$,
since $\ppp = C \sss$ from Corollary~\ref{cor.AA.cartan} (ii).
\end{proof}

We will deduce the description of $K(A)$ from
Proposition~\ref{prop.MA} and the following lemma from linear algebra:

\begin{lemma}
\label{lem.coker}
Let $\sss$ and $\ppp$ be column vectors in $\ZZ^{\ell+1}$ with
$\ell \geq 1$ and $\sss_{\ell+1}=1$.
%\footnote{\edit{I (Jia) changed $\sss_1=1$ to $\sss_{\ell+1}=1$ to be consistent with the way we label the trivial module later.}}
(In this lemma, $\sss$ and $\ppp$ are not required to be the vectors
from Subsection~\ref{subsect.conventions}.) %{subsect.findim-alg}
Set $d:=\sss^T \ppp$ and assume that $d \neq 0$.
Let $\gamma:=\gcd(\ppp)$.
Then the matrix $L := d I_{\ell+1} - \ppp \sss^T$
has cokernel
\[
\ZZ^{\ell+1} / \im L
\cong \ZZ \oplus \left(\ZZ / \gamma \ZZ\right)
      \oplus \left(\ZZ / d \ZZ\right)^{\ell-1} .
\]
\end{lemma}

\begin{verlong}
We shall now give a proof of Lemma~\ref{lem.coker} using the Smith
normal form of a matrix. For a second, more elementary proof, see
Section~\ref{sect.smith}.

\begin{proof}[First proof of Lemma~\ref{lem.coker}.]
\end{verlong}

\begin{vershort}
\begin{proof}%[Proof of Lemma~\ref{lem.coker}.]
\end{vershort}
Note that
$
\sss^T L=d \sss^T - \sss^T \ppp \sss^T=d\sss^T-d\sss^T=0.
$
This has two implications.  One is that $L$ is singular,
so its Smith normal form has diagonal entries $(d_1,d_2,\ldots,d_{\ell},0)$,
with $d_i$ dividing $d_{i+1}$ for each $i$.
Hence
$
\ZZ^{\ell+1}/\im L \cong \ZZ \oplus \left( \bigoplus_{i=1}^\ell
\ZZ/d_i\ZZ \right),
$
and our goal is to show that $(d_1,d_2,\ldots,d_\ell)=(\gamma,d,d,\ldots,d)$.

The second implication is that $\im L \subset \sss^\perp$,
which we claim lets us reformulate the cokernel of $L$ as follows:
\begin{equation}
\label{cokernel-of-L-reformulation}
\ZZ^{\ell+1}/\im L \cong \ZZ \oplus \sss^\perp/\im(L),
\qquad \left( \text{ so that }
\sss^\perp/\im(L) \cong \bigoplus_{i=1}^\ell \ZZ/d_i\ZZ \,\, \right). 
\end{equation}
To see this claim, note that $\xx \mapsto \sss^T \xx$ gives
a surjection $\ZZ^{\ell+1} \rightarrow \ZZ$, since $\sss_{\ell+1}=1$,
and hence a short exact sequence
$
0 \rightarrow \sss^\perp \rightarrow \ZZ^{\ell+1} \rightarrow \ZZ \rightarrow 0.
$
The sequence splits since $\ZZ$ is a free (hence projective) $\ZZ$-module,
and then the resulting direct sum decomposition $\ZZ^{\ell+1} = \ZZ
\oplus \sss^\perp$
induces the claimed decomposition in \eqref{cokernel-of-L-reformulation}.

Note furthermore that the abelian group 
$\sss^\perp/\im(L)$ is all $d$-torsion, since
for any $\xx$ in $\sss^\perp$, one has that $\im(L)$ contains
$
L\xx = d\xx- \ppp \sss^T \xx = d\xx.
$
Therefore each of $(d_1,d_2,\ldots,d_{\ell})$ must divide $d$.

Note that $\gamma=\gcd(\ppp)$ must divide $d=\sss^T \ppp$,
and hence we may assume without loss of generality that $\gamma=1$,
after replacing $\ppp$ with $\frac{1}{\gamma}\ppp$:  this
has the effect of replacing $d$ with $\frac{d}{\gamma}$, replacing
$\gamma$ with $1$,
replacing $L$ with $\frac{1}{\gamma}L$,
and $(d_1,d_2,\ldots,d_{\ell},0)$ with
$\frac{1}{\gamma}(d_1,d_2,\ldots,d_\ell,0)$
(since the Smith normal form of $\frac{1}{\gamma}L$ is obtained from
that of $L$ by dividing all entries by $\gamma$).

Once we have assumed $\gamma=1$, our goal is to show
$(d_1,d_2,\ldots,d_\ell)=(1,d,d,\ldots,d)$.
However, since we have seen that each $d_i$ divides $d$, it only
remains to show
that $d_1=1$, and $d$ divides each of $(d_2,d_3,\ldots,d_\ell)$.
To this end, recall (e.g., \cite[\S12.3 Exer. 35]{DummitFoote})
that if one defines $g_k$ as the gcd of all $k \times k$ minor
subdeterminants of $L$, then $d_k=\frac{g_k}{g_{k-1}}$.
Thus it remains only to show that $g_1=1$ and that
$g_2$ is divisible by $d$.

To see that $g_1=1$, we claim $1$ lies in the ideal $I$ of $\ZZ$
generated by the last column
$
[ -\ppp_1, -\ppp_2, \ldots, -\ppp_\ell,  d-\ppp_{\ell+1}]^T
$
of $L$
together with the $(1,1)$-entry $L_{1,1}=d-\ppp_1 \sss_1$.  To see
this claim, note that
$d= (d-\ppp_1 \sss_1) + \sss_1 \cdot \ppp_1$ lies in $I$,  hence
$\ppp_{\ell+1} = d - (d-\ppp_{\ell+1})$ lies in $I$,
and therefore $1=\gcd(\ppp_1,\ppp_2,\ldots,\ppp_{\ell+1})$ also lies in $I$.

To see $d$ divides $g_2$, we need only to show that
each $2 \times 2$ minor subdeterminant of $dI-\ppp \sss^T$ 
vanishes modulo $d$.  This holds, since working modulo $d$,
one can replace  $dI-\ppp \sss^T$ by  $-\ppp \sss^T$,
a rank one matrix.
\end{proof}

We can now prove Theorem~\ref{regular-rep-theorem}.  Recall that its statement
involves the number $\ell+1$ of simple $A$-modules, the dimension $d$ of $A$,
and the gcd $\gamma$ 
of the dimensions of the indecomposable projective $A$-modules.

\vskip.1in
\noindent
{\bf Theorem~\ref{regular-rep-theorem}.}
{\it 
Let $d := \dim A$ and $\gamma := \gcd(\ppp)$.
If $\ell=0$ then $K(A)=0$, else
$
K(A) \cong \left( \ZZ/\gamma\ZZ \right)
             \oplus \left( \ZZ/d\ZZ \right)^{\ell-1}.
$
}
\vskip.1in
\begin{proof}
The case $\ell=0$ is somewhat trivial, since
$M_A, L_A$ are the $1 \times 1$ matrices $[d],[0]$,
and $K(A)=0$.

When $\ell \geq 1$, note that
Proposition~\ref{prop.important-dot-product} gives
$\sss^T \ppp = \dim A = d$, and Proposition~\ref{prop.MA} yields
$M_A = \ppp \sss^T$, so that
$
L_A = d I_{\ell+1} - \ppp \sss^T.
$ 
Reindexing $S_1,\ldots,S_{\ell+1}$ so that $S_{\ell+1}=\epsilon$,
the result now follows from  Lemma~\ref{lem.coker}.
\end{proof}

%%%%%%%%%%%%%%%%%%%%%%%%%%%%%%%%%%%%%%%%%%%%%%%%%%%%%%%%%
\section{Proofs of Theorem~\ref{Lorenzini-style-formula}
and Corollary~\ref{Gaetz-formula}}
\label{Lorenzini-Gaetz-section}

We record here a key 
observation of Lorenzini \cite[Prop. 2.1]{Lorenzini} 
that leads to a formula
for the cardinality of the critical group $K(V)$.

\begin{proposition}[{\cite[Prop. 2.1]{Lorenzini}}]
\label{Lorenzini-prop}
Let $L$ be a matrix in $\ZZ^{(\ell+1)\times(\ell+1)}$,
regarded as a linear map $\ZZ^{\ell+1} \rightarrow \ZZ^{\ell+1}$,
 of rank $\ell$, 
 with characteristic polynomial $x\prod_{i=1}^\ell (x-\lambda_i)$,
and whose integer right-nullspace
\begin{verlong}
\footnote{The \textit{integer
right-nullspace} of $L$ is defined to be the $\ZZ$-module of all
column vectors $u \in \ZZ^{\ell+1}$ such that $L u = 0$.
The \textit{integer left-nullspace} of $L$ is defined to be the
$\ZZ$-module of all
column vectors $u \in \ZZ^{\ell+1}$ such that $L^T u = 0$.}
\end{verlong}
 (resp. left-nullspace) is 
spanned over $\ZZ$ by the primitive vector $\nnn$
(resp. $\nnn'$) in $\ZZ^{\ell+1}$. Assume that $\nnn^T \nnn' \neq 0$.

% \begin{enumerate}
% \item[(i)]
Then, the torsion part $K$ of the cokernel
$\ZZ^{\ell+1}/\im L \cong \ZZ \oplus K$
has cardinality
$
\card{K} = \left| \frac{1}{\nnn^T \nnn'} 
                  (\lambda_1 \lambda_2 \cdots \lambda_\ell) \right|.
$
% \item[(ii)] (THE FOLLOWING MAY NOT BE LITERALLY TRUE:)
% If a nonzero integer eigenvalue $\lambda$ of $L$
% has multiplicity $m=\dim_\QQ \ker(\lambda I_{\ell+1}-L)  \geq 1$,
% then $K$ contains a subgroup isomorphic to 
% $(\ZZ/\lambda \ZZ)^{m-1}$.
% \end{enumerate}
\end{proposition}
\begin{verlong}
\begin{proof}
This is a restatement of
$\lambda_1 \lambda_2 \cdots \lambda_\ell
= \pm \card{K} \cdot \left(\nnn^T \nnn'\right)$,
which is the displayed equation in
\cite[Prop. 2.1]{Lorenzini}. (The requirement $n > 1$ is
unnecessary.)
\end{proof}
\end{verlong}

This lets us prove Theorem~\ref{Lorenzini-style-formula} from the
Introduction,
whose statement we recall here.

\vskip.1in
\noindent
{\bf Theorem~\ref{Lorenzini-style-formula}.}
{\it
Let $d := \dim A$ and $\gamma := \gcd(\ppp)$.
Assume $K(V)$ is finite, so that $L_V$ has nullity one.  If
the characteristic polynomial of $L_V$ factors as 
$\det(xI-L_V)=x\prod_{i=1}^\ell (x-\lambda_i)$, then 
$
\card{K(V)} = \left| \frac{\gamma}{d} 
                     (\lambda_1 \lambda_2 \cdots \lambda_\ell)\right|.
$
}

\begin{proof}
From \eqref{eq.K(V).def2}, we see that
$K\left(V\right)$ is isomorphic to the torsion part of
$\ZZ^{\ell+1} / \im (L_V)$.

From Proposition~\ref{prop.important-dot-product}, we obtain
$\sss^T \ppp = d (\neq 0)$. 
%Thus,$\ppp^T \sss = \sss^T \ppp = d \neq 0$.
Propositions~\ref{left-nullvector-prop} and
\ref{right-nullvector-prop} exhibit $\sss$ and $\ppp$
as left- and right-nullvectors of $L_V$ in $\ZZ^{\ell+1}$.
Note that $\sss$ is primitive, since
one of its coordinates is $\dim(\epsilon)=1$,
while $\frac{1}{\gamma}\ppp$ is also primitive.
Since the integer left-nullspace and the integer right-nullspace of
$L_V$ are % saturated and
free of rank $1$ (because $L_V$ has nullity $1$),
this shows that $\sss$ and $\frac{1}{\gamma}\ppp$ span these two
nullspaces.
Then Proposition~\ref{Lorenzini-prop} (applied to
$\nnn = \frac{1}{\gamma}\ppp$ and $\nnn' = \sss$) implies 
\[
\card{K(V)}
= \left| \frac{1}{\left(\frac{1}{\gamma}\ppp\right)^T \sss} 
                 (\lambda_1 \lambda_2 \cdots \lambda_\ell) \right|
= \left| \frac{\gamma}{d} 
             (\lambda_1 \lambda_2 \cdots \lambda_\ell) \right|.\qedhere
%\qquad \left(\text{since } \ppp^T \sss = d \right) 
\]
\end{proof}

The important role played by $\gamma=\gcd(\ppp)$ in 
Theorem~\ref{regular-rep-theorem} and
Theorem~\ref{Lorenzini-style-formula} 
raises the following question.

\begin{question}
\label{gcd-question}
For a finite-dimensional Hopf algebra $A$ over an algebraically closed field,
what does the gcd of the dimensions of the 
indecomposable projective $A$-modules ``mean'' in terms of the structure of $A$?
\end{question}

We shall answer this question for some Hopf algebras $A$
in Remark~\ref{rmk.sylow-gcd} further below.
The following answer for group algebras may be known to experts,
but we did not find it in the literature.

\begin{proposition}
\label{group-algebra-gcd}
For $A=\FF G$ the group algebra of a finite group $G$,
the gcd $\gamma$ of the dimensions $\ppp$
of the indecomposable projective $\FF G$-modules equals 
\begin{itemize}
\item $1$ when $\FF$ has characteristic zero,
\item the order of a $p$-Sylow subgroup of $G$ 
when $\FF$ has characteristic $p > 0$.
\end{itemize}
\end{proposition}
\begin{proof}
The statement is obvious in characteristic $0$,
since $\gamma = 1$, as $\epsilon$ is a $1$-dimensional 
projective $A$-module.

Thus we may assume $\FF$ has positive characteristic $p$.
We first claim $\gamma=\gcd(\ppp)$ is a power of $p$.
To deduce this, let $C$ be the Cartan matrix of $A$.
Proposition~\ref{prop.AA.cartan.2} shows that $\ppp^T=\sss^T C$.
%By~\cite[Cor. 10.2.4]{Webb}, the Cartan matrix $C$ is invertible, so $\sss^T=\ppp^T C^{-1}$.  
Multiplying this equation on the right by the \emph{adjugate} matrix $\adj(C)$, whose entries are the cofactors of $C$,
one finds that
\begin{equation}
\label{Cartan-and-Brauer-consequence}
\ppp^T \adj(C)  = \sss^T C \adj(C) = \det(C) \sss^T. 
\end{equation}
%\footnote{The Cartan matrix $C$ of a finite group $G$ over an algebraically closed field $\FF$ is actually symmetric by \cite[Thm.~8.5.7]{Webb}}. 
The positive integer $\gamma$ divides every entry of $\ppp$, 
and hence divides every entry on the left of
\eqref{Cartan-and-Brauer-consequence}.  Note that
 $\det(C)$ occurs as an entry on the right of 
\eqref{Cartan-and-Brauer-consequence}, so $\gamma$ divides $\det(C)$,
which by a result of Brauer \cite[Thm. 1]{Brauer-Cartan-invariants}
(also proven in \cite[\S 16.1, Corollary 3]{Serre} and
\cite[Theorem (18.25)]{CurtisReiner1}) is a power of $p$.
That is, $\gamma = p^b$ for some $b \geq 0$.

All that remains is to apply a result of Dickson, asserting
that the $p$-Sylow order $p^a$ for $G$ is the minimum of the powers
of $p$ dividing the dimensions $\dim(P_i)$;
see Curtis and Reiner \cite[(84.15)]{CurtisReiner-old}.  We give
a modern argument for this here.  
Since the $p$-Sylow order $p^a$ for $G$ 
divides the dimension of every 
projective $\FF G$-module 
(see \cite[\S 18, Exer. 5]{CurtisReiner1}, \cite[Cor. 8.1.3]{Webb}),
it also divides $\gamma=p^b$, implying $b \geq a$.
For the opposite inequality, since $\card{G} = p^a q$ where $\gcd(p,q)=1$,
and $\card{G} = \dim \FF G = \dim A = \sss^T \ppp$
by Proposition~\ref{prop.important-dot-product}, the prime power
$p^b = \gamma$ divides $\sss^T \ppp = \card{G} = p^a q$,
and therefore $b \leq a$.  Thus $b=a$, so that $\gamma = p^b = p^a$.
\end{proof}

Since the number of simple $\FF G$-modules is the number of $p$-regular
$G$-conjugacy classes, the following is immediate from 
Theorem~\ref{regular-rep-theorem}
and Proposition~\ref{group-algebra-gcd}.

\begin{corollary}
For the group algebra $A=\FF G$ of a finite group $G$,
with $\ell+1 \geq 2$ different $p$-regular conjugacy classes, and
$p$-Sylow order $p^a$,
the regular representation $A$ has critical group 
\[
K(A)
\cong 
\left( \ZZ/p^a\ZZ \right)
   \oplus \left( \ZZ/ (\card G)\ZZ \right)^{\ell-1}.
\]
\end{corollary}

Since for group algebras, either of
Proposition~\ref{left-eigenvector-prop} or \ref{right-eigenvector-prop}
identified the eigenvalues of $L_V$ in terms of
the Brauer character values of $V$, one immediately
deduces Corollary~\ref{Gaetz-formula} from the
Introduction:

\vskip.1in
\noindent
{\bf Corollary~\ref{Gaetz-formula}.}
{\it 
For any $\FF G$-module $V$ of dimension $n$ with $K(V)$ finite,
one has 
$$
\card{K(V)} = \frac{p^a}{\card{G}} 
\displaystyle\prod_{g \neq e} \left(n-\chi_V(g)\right),
$$
where the product runs through a set of representatives $g$ for
the non-identity $p$-regular $G$-conjugacy classes.
In particular, the quantity on the right is a positive integer.
}
\vskip.1in

\begin{example}
Let us compute what some of the foregoing results
say when $A=\FF G$ for the symmetric group $G=\SS_4$,
and $\FF$ has characteristic $p$,
assuming some facts about modular $\SS_N$-representations
that can be found, e.g., in James and Kerber \cite{JamesKerber}.
Every field $\FF$ is a splitting field for each $\SS_N$,
so we may assume $\FF=\FF_p$.  Furthermore
one need only consider three cases, namely $p=2,3$ and $p \geq 5$, since
$\FF \SS_N$ is semisimple for $p > N$, and in that case,
the theory is the same as in characteristic zero.
The simple $A$-modules can be indexed $D^\lambda$ where
$\lambda$ are the $p$-regular partitions of $N=4$, that is,
those partitions having no parts repeated $p$ or more times.
For $p = 2, 3$, we have the following Brauer character tables and
Cartan matrices (see \cite[Example 10.1.5]{Webb}):
$$
p=2:
\qquad
\bordermatrix{
~ &e&(ijk) \cr
D^{4}&1&1 \cr
D^{31}&2&-1 
}
\qquad
C=\left(
\begin{matrix}
4&2 \\
2&3 
\end{matrix}
\right)
$$
\vskip.2in 
$$
p=3:
\qquad
\bordermatrix{
~ &e&(ij)&(ij)(kl)&(ijkl) \cr
D^{4}&1&1&1&1 \cr
D^{31}&3&1&-1&-1\cr 
D^{22}&1&-1&1&-1 \cr
D^{211}&3&-1&-1&1
}
\qquad
C=\left(
\begin{matrix}
2&0&1&0 \\
0&1&0&0 \\
1&0&2&0 \\
0&0&0&1 
\end{matrix}
\right)
$$
while for $p \geq 5$, the Brauer character table is the ordinary one
(and the Cartan matrix $C$ is the identity):
$$
\bordermatrix{
~ &e&(ij)&(ij)(kl)&(ijk)&(ijkl) \cr
D^{4}&1&1&1&1&1 \cr
D^{31}&3&1&0&-1&-1\cr 
D^{22}&2&0&-1&2&0 \cr
D^{211}&3&-1&0&-1&1\cr
D^{1111}&1&-1&1&1&-1}.
$$
In each case, $\sss$ is the first column
of the Brauer character table, 
$\ppp^T= \sss^T C$, and $\gamma=\gcd(\ppp)$:
\begin{center} 
\begin{tabular}{|c|c|c|c|}\hline
$p$&$\sss$&$\ppp$ &$\gamma$\\ \hline\hline
$2$&$[1,2]^T$& $[8,8]^T$&$8=2^3$\\ \hline
$3$&$[1,3,1,3]^T$&$[3,3,3,3]^T$&$3$ \\ \hline
$\geq 5$&$[1,3,2,3,1]^T$&$[1,3,2,3,1]^T$&$1$\\ \hline
\end{tabular}
\end{center}
Note that $\gamma$ is the order $p^a$ of the $p$-Sylow subgroups for
$G=\SS_4$ in each case.

In Section~\ref{M-matrix-section} we will show that 
the critical group $K(V)$ is finite if and only if $V$ is tensor-rich.
One can read off which simple $\FF\SS_4$-modules $V=D^\lambda$ are
tensor-rich using Theorem~\ref{Brauer-result} below: this holds
exactly when the only $g \in \SS_4$ satisfying $\chi_V(g)=n:=\dim V$
is $g=e$.  
Perusing the above tables, one sees
that in each case, the simple modules
labeled $D^{4}, D^{22}, D^{1111}$ (to the extent they are simple)
are the ones which are not tensor-rich.
However, the module $V=D^{31}$ is tensor-rich for each $p$, 
and one can use its character
values $\chi_V(g)$ to compute $M_V, L_V, K(V)$ and check
Corollary~\ref{Gaetz-formula} in each case as follows:
\begin{center}
\begin{tabular}{|c|c|c|c|c|l|}\hline
$p$ & $M_V$ for $V=D^{31}$ & $L_V=nI-M_V$ &Smith form of $L_V$& $K(V)$ &$\#K(V)=\frac{\gamma}{\#G}\prod_{g \neq e}(n-\chi_V(g))$\\ \hline
 & & & & & \\
$2$ &
$\left(\begin{smallmatrix}
0 & 2\\
1 & 1\\
\end{smallmatrix}\right)$
 &
$\left(\begin{smallmatrix}
2 & -2\\
-1 & 1\\
\end{smallmatrix}\right)$ & 
$\left(\begin{smallmatrix}
1 & 0\\
0 & 0\\
\end{smallmatrix}\right)$ & 
$0$ &
$1=\frac{8}{24}(2-(-1))$\\ 
 & & & & & \\ \hline
 & & & & & \\
$3$ &
$\left(\begin{smallmatrix}
0 & 2 & 0 & 1\\
1 & 1 & 0 & 1\\
0 & 1 & 0 & 2\\
0 & 1 & 1 & 1
\end{smallmatrix}\right)$
 &
$\left(\begin{smallmatrix}
3 & -2 & 0 & -1\\
-1 & 2 & 0 & -1\\
0 & -1 & 3 & -2\\
0 & -1 & -1 & 2
\end{smallmatrix}\right)$ & 
$\left(\begin{smallmatrix}
1 & 0 & 0 & 0\\
0 & 1 & 0 & 0\\
0 & 0 & 4 & 0\\
0 & 0 & 0 & 0
\end{smallmatrix}\right)$
 & 
$\ZZ/4\ZZ$ &
$4=\frac{3}{24}(3-1)(3-(-1))(3-(-1))$ \\ 
 & & & & & \\ \hline
 & & & & & \\
$\geq 5$ &
$\left(\begin{smallmatrix}
0 & 1 & 0 & 0 & 0\\
1 & 1 & 1 & 1 & 0\\
0 & 1 & 0 & 1 & 0\\
0 & 1 & 1 & 1 & 1\\
0 & 0 & 0 & 1 & 0\\
\end{smallmatrix}\right)$
 &
$\left(\begin{smallmatrix}
3 & -1 & 0 & 0 & 0\\
-1 & 2 & -1 & -1 & 0\\
0 & -1 & 3 & -1 & 0\\
0 & -1 & -1 & 2 & -1\\
0 & 0 & 0 & -1 & 3\\
\end{smallmatrix}\right)$
 & 
$\left(\begin{smallmatrix}
1 & 0 & 0 & 0 & 0\\
0 & 1 & 0 & 0 & 0\\
0 & 0 & 1 & 0 & 0\\
0 & 0 & 0 & 4 & 0\\
0 & 0 & 0 & 0 & 0
 \end{smallmatrix}\right)$
 & 
$\ZZ/4\ZZ$ &
$4=\frac{1}{24}(3-1)(3-0)(3-(-1))(3-(-1))$ \\
 & & & & & \\ \hline
\end{tabular}
\end{center}
The answer $K(D^{31}) \cong \ZZ/4\ZZ$ for $p \geq 5$
is also consistent with Gaetz \cite[Example 6]{Gaetz}.
\end{example}

\begin{example}
The above examples with $G=\SS_4$ are slightly deceptive,
in that, for each prime $p$, there exists an $\FF \SS_4$-module
$P_i$ having $\dim P_i =\gamma=\gcd(\ppp)$.  This fails for $G=\SS_5$,
e.g., examining $\FF_3 \SS_5$-modules, one finds that $\sss=(1,1,4,4,6)$ and
$\ppp=(6,6,9,9,6)$, so that $\gamma=3$, but $\dim P_i \neq 3$ for all $i$.
\end{example}

\begin{verlong}
\begin{example}
(This example at least illustrates that $\gamma$ is not always equal to $\dim(P_i)$ for some $i$.)
Let $A=\FF G$ be the group algebra of the symmetric group $G=\SS_5$ over the algebraic closure $\FF$ of $\FF_3$.
By computations in Magma, the Brauer character tables of the simple
$A$-modules and indecomposable projective $A$-modules are 
\[ \begin{pmatrix}
1 & 1 & 1 & 1 & 1 \\
1 & -1 & 1 & -1 & 1 \\
4 & 2 & 0 & 0 & -1 \\
4 & -2 & 0 & 0 & -1 \\
6 & 0 & -2 & 0 & 1
\end{pmatrix} 
\quad\text{and}\quad
\begin{pmatrix}
6 & 0 & 2 & 2 & 1 \\
6 & 0 & 2  & -2 & 1 \\
9 & 3 & 1 & -1 & -1 \\
9 & -3 & 1 & 1 & -1 \\
6 & 0 & -2 & 0 & 1
\end{pmatrix} \]
where the simple $A$-modules $S_1,\ldots,S_5$ and the indecomposable projective $A$-modules $P_1,\ldots,P_5$ are labeled in such a way that $S_1$ is the trivial $A$-module and $S_i = \top(P_i)$ for all $i$. 
We have $\sss=(1,1,4,4,6)$, $\ppp=(6,6,9,9,6)$, and $\gcd(\ppp)=3$ equals the order of a $3$-Sylow subgroup of $\SS_5$.
The $A$-module $V=P_4$ is tensor-rich by Proposition~\ref{kernel-via-characters-proposition} and Theorem~\ref{Brauer-result}, since $\chi_V(e)=9\ne \chi_V(g)$ for all $3$-regular $g\ne e$.
We have
\[ M_V = \begin{pmatrix}
1 & 0 & 0 & 2 & 0 \\
0 & 1 & 2 & 0 & 0 \\
2 & 2 & 1 & 4 & 2 \\
2 & 2 & 4 & 1 & 2 \\
2 & 2 & 4 & 4 & 3 
\end{pmatrix}, \quad
L_V = \begin{pmatrix}
8 & 0 & 0 & -2 & 0 \\
0 & 8 & -2 & 0 & 0 \\
-2 & -2 & 8 & -4 & -2 \\
-2 & -2 & -4 & 8 & -2 \\
-2 & -2 & -4 & -4 & 6 
\end{pmatrix}, \quad
\overline{L_V} = \begin{pmatrix}
 8 & -2 & 0 & 0 \\
 -2 & 8 & -4 & -2 \\
 -2 & -4 & 8 & -2 \\
 -2 & -4 & -4 & 6 
\end{pmatrix}.
  \]
Then $K(V)=\mathbf{s}^\perp/\im L_V = (\ZZ/2\ZZ)^3\oplus\ZZ/24\ZZ$ and $\ZZ^4/\im\overline{L_V} = (\ZZ/2\ZZ)^2\oplus\ZZ/4\ZZ \oplus \ZZ/24\ZZ$, which are different.
The eigenvalues of $L_V$ are $0,8,8,10,12$, and $\card{K(V)} =2^3\cdot 24 = 3\cdot 8^2\cdot10\cdot12/5!$, as predicted by Theorem~\ref{Lorenzini-style-formula}.

Now let $V=S_2$ be the ``sign'' representation, which is not tensor-rich since $\chi_V(g)=\chi_V(e)=1$ for some $3$-regular $g\ne e$ by the above character table. Computations in Magma show
\[ M_V = \begin{pmatrix}
0 & 1 & 0 & 0 & 0 \\
1 & 0 & 0 & 0 & 0 \\
0 & 0 & 0 & 1 & 0 \\
0 & 0 & 1 & 0 & 0 \\
0 & 0 & 0 & 0 & 1 
\end{pmatrix}, \quad
L_V = \begin{pmatrix}
1 & -1 & 0 & 0 & 0 \\
-1 & 1 & 0 & 0 & 0 \\
0 & 0 & 1 & -1 & 0 \\
0 & 0 & -1 & 1 & 0 \\
0 & 0 & 0 & 0 & 0 
\end{pmatrix}, \quad
\overline{L_V} = \begin{pmatrix}
1 & 0 & 0 & 0 \\
0 & 1 & -1 & 0 \\
0 & -1 & 1 & 0 \\
0 & 0 & 0 & 0 
\end{pmatrix}. \]
Then $\overline{L_V}$ is singular and $K(V)=\ZZ^2 = \ZZ^4/\im \overline{L_V}$ is infinite.

For $g\in\SS_5$, $\chi_V(g)=1$ if and only if $g$ is an even permutation.
There are $40$ even $3$-regular permutations in $\SS_5$: the $24$ $5$-cycles, the $15$ products of two disjoint $2$-cycles, and the identity permutation.
These permutations generate the alternating group $N=A_5$.
Let $B=\FF[G/N]$. 
Then $V$ is a one-dimensional $B$-module on which $N$ acts by $1$ and $(1,2)N$ acts by $-1$.
As a $B$-module, $V$ is tensor-rich.
\end{example}
\end{verlong}

\begin{example}
\label{Taft-algebra-gcd-example}
Working over an algebraically closed field $\FF$ of characteristic zero,
the generalized Taft Hopf algebra $A=H_{n,m}$ 
from Example~\ref{Taft-algebra-example} has
dimension $mn$.  It has $\ell+1=n$ projective indecomposable representations 
$P_1,\ldots,P_n$, each of dimension $m$, with top $S_i=\top(P_i)$
one-dimensional (see \cite[\S4]{Cibils} and \cite[\S2]{LiZhang}).
Hence in this case, $\gamma=\gcd(\ppp)=m$ and Theorem~\ref{regular-rep-theorem}
yields
\[
K(A) \cong \left( \ZZ/m\ZZ \right)
             \oplus \left( \ZZ/mn\ZZ \right)^{n-2}
\qquad \text{for } n \geq 2.
\]
\end{example}

\begin{example}
\label{asymmetric-Nichols-regular-rep-example}
For Radford's Hopf algebra $A=A(n,m)$  from 
Examples~\ref{asymmetric-Nichols-Hopf-algebra-example},
\ref{asymmetric-Nichols-Cartan-example},
all indecomposable projectives $\{ P_k \}_{k=0}^{n-1}$ are
$2^m$-dimensional, so $\gamma:=\gcd(\ppp)=2^m$ and 
Theorem~\ref{regular-rep-theorem} gives
$$
K(A) \cong \left( \ZZ / 2^m \ZZ \right) \oplus
           \left( \ZZ / n 2^m \ZZ \right)^{n-2}.
$$
\end{example}

\begin{example}
There is a special case
%\footnote{{\color{red} Initially, I (Vic) thought one could be slight more general here, but I was confused by what Humphreys and Jantzen write about it.  Namely, I thought one could replace the prime $p$ by $q=p^r$ for any positive integer $r$, and replace the $p$-power operation $F: x \mapsto x^{[p]}$ on $\frak{g}$ with its $r^{th}$ iterate $F^r$.  I though that all the same assertions hold, replacing $p$ by $q$;  historically, Humphreys dealt with $r=1$ case, and Jantzen the case of general $r$.  However, I queried Humphreys and he responded (email of Feb 22, 2017, saying that the algebra corresponding to ${\frak u}({\frak g})$ for $r>1$ is not just a quotient of U(g) in the same way, and is not even a Hopf algebra!}
of the restricted universal enveloping algebras $A=\frak{u}(\frak{g})$
from Example~\ref{restricted-Lie-algebra} 
where one has all the data needed for Theorem~\ref{regular-rep-theorem}.
Namely, when $\frak g$ is associated to a simple, simply-connected algebraic
group $\GG$ defined and split over $\FF_p$, as in
Humphreys \cite[Chap. 1]{Humphreys}, then there is a natural parametrization of
the simple $A$-modules via the set $X/pX$ where $X \cong \ZZ^{\rank{\frak g}}$ 
is the \emph{weight lattice} for $\GG$ or $\frak g$.  
Although the dimensions of the projective indecomposable $A$-modules $P_i$
are not known completely, they are all divisible by the dimension of
one among them, specifically, the \emph{Steinberg module} of dimension $p^N$
where $N$ is the number of \emph{positive roots}; 
see \cite[\S 10.1]{Humphreys}.  
Consequently, here one has
$$
\begin{array}{rlll}
\gamma&:=\gcd(\ppp)&=p^N&\\
d&:=\dim(A)&=p^{\dim \frak{g}}&\\
\ell+1&:=\#\{\text{simple }A\text{-modules}\}&=\#X/pX&=p^{\rank{\frak g}},
\end{array}
$$
and Theorem~\ref{regular-rep-theorem} implies
\[
K(A) \cong \left( \ZZ/p^N\ZZ \right)
             \oplus \left( \ZZ/p^{\dim \frak{g}}\ZZ \right)^{p^{\rank{\frak g}}-2}.
\]
\end{example} 

\begin{remark}
\label{rmk.sylow-gcd}
All the above examples of Hopf algebras $A$ share a common
interpretation for $\gamma=\gcd(\ppp)$ which we find suggestive.  
Each has a family of $\FF$-subalgebras $B \subset A$, which one is tempted to
call \emph{Sylow subalgebras}, with the following properties:
\begin{enumerate}
\item[(i)] The augmentation ideal $\ker(B \overset{\epsilon}{\rightarrow} \FF)$ is a nil ideal, that
is, it consists entirely of nilpotent elements.
\item[(ii)] $A$ is free as a left $B$-module.
\item[(iii)] $\dim B=\gamma$.
\end{enumerate}
We claim that properties (i) and (ii) already imply that $\dim B$ \emph{divides}
$\gamma$ (cf. \cite[proof of Cor. 8.1.3]{Webb}):  
property (i) implies $B$ has only one simple
module, namely $\epsilon$, whose projective cover must be
$B$ itself, and property (ii) implies that each 
projective $A$-module $P_i$ restricts to a projective $B$-module, 
which must be of form $B^{t}$, so that $\dim B$ divides $\dim P_i$,
and hence divides $\gcd(\{\dim P_i\})=\gamma$.
Thus property (iii) implies that $B$ must be \emph{maximal}
among subalgebras of $A$ having properties (i),(ii).
\begin{itemize}
\item
When $A$ is semisimple, then $B = \FF 1_A$.
\item
When $A=\FF G$ is a group algebra and $\FF$ has characteristic $p$, then $B=\FF H$ is the group algebra for any $p$-Sylow subgroup $H$.  
\item
When $A=H_{n,m}$ is the generalized Taft Hopf algebra, $B$ is the subalgebra $\FF\langle x \rangle$ generated by $x$, or by any of the elements of the form $g^i x$ for $i=0,1,\ldots,n-1$.
\item
When $A=A(n,m)$ is Radford's Hopf algebra, $B$ is the exterior subalgebra $\Lambda[x_1,\ldots,x_m]$ generated by $x_1,\ldots,x_m$, or various isomorphic subalgebras
$\Lambda[g^i x_1, \ldots, g^i x_m]$ for $i \in \ZZ$.
\item
When $A={\frak u}(\frak g)$ is the restricted universal enveloping algebra for the Lie algebra $\frak g$ of
a semisimple algebraic group over $\FF_p$, 
then $B= {\frak u}(\mathfrak{n}_+)$ for
a nilpotent subalgebra $\mathfrak{n}_+$ in a triangular decomposition 
$\mathfrak{g} = \mathfrak{n}_-  \oplus  \mathfrak{h}  \oplus  \mathfrak{n}_+$.
% Jim Humphreys corroborated that $B= {\frak u}(\frak n)$ really has properties (i),(ii)
\end{itemize}
\end{remark}

\begin{question}
\label{Sylow-subalgebra-question}
For which finite-dimensional Hopf algebras $A$ over an algebraically closed field $\FF$ is there a subalgebra $B$
satisfying properties (i),(ii),(iii) above?
\end{question}

%%%%%%%%%%%%%%%%%%%%%%%%%%%%%%%%%%%%%%%%%%%%%%%%%%%%%%%%%
\section{Proof of Theorem~\ref{tensor-rich-equivalences-theorem}}
\label{M-matrix-section}

We recall the statement of the theorem, involving
an $A$-module $V$ of dimension $n$, with 
$L_V=nI_{\ell+1}-M_V$ in $\ZZ^{(\ell+1) \times (\ell+1)}$, 
and its submatrix $\overline{L_V}$ in $\ZZ^{\ell \times \ell}$.
\vskip.1in
\noindent
{\bf Theorem~\ref{tensor-rich-equivalences-theorem}.}
{\it
The following are equivalent for an $A$-module $V$:
\begin{enumerate}
\item[(i)] $\overline{L_V}$ is a nonsingular $M$-matrix. 
\item[(ii)] $\overline{L_V}$ is nonsingular.
\item[(iii)] $L_V$ has rank $\ell$, so nullity $1$.
\item[(iv)] $K(V)$ is finite.
\item[(v)] $V$ is tensor-rich.

\end{enumerate}
}
\vskip.1in
\noindent
The definitions for $V$ to be tensor-rich and for $\overline{L_V}$ to
be a nonsingular $M$-matrix are given below.

\begin{definition}
Let $V$ be an $A$-module.
Say that $V$ is \emph{rich} if $[V:S_i]>0$ for every simple $A$-module $S_i$. 
Say that $V$ is \emph{tensor-rich} if for some positive integer $t$,
the $A$-module $\bigoplus_{k = 0}^t V^{\otimes k}$ is rich.
\end{definition}

\begin{definition}
\label{our-M-matrix-definition}
Let $Q$ be a matrix in $\RR^{\ell \times \ell}$ whose
off-diagonal entries are nonpositive, that is,
$Q_{i,j} \leq 0$ for $i \neq j$.
Then $Q$ is called a \emph{nonsingular $M$-matrix} 
if it is invertible and the entries in $Q^{-1}$ are all nonnegative.
% $Q^{-1} \geq 0$.
\end{definition}

To prove the theorem, we will show the following implications:
$$
\begin{array}{rcccccl}
                       &     &                     &\text{(iv)}                  &                   &   &  \\
                       &     &                    &\Updownarrow &                   &    & \\
\text{(i)} \Rightarrow& \text{(ii)}& \Rightarrow&\text{(iii)}                   & \Rightarrow &\text{(v)}& \Rightarrow \text{(i)},
\end{array}
$$
after first establishing some inequality notation for vectors and matrices.

\begin{definition}
Given $u,v$ in $\RR^m$, write $u \leq v$ (resp. $u < v$) if 
$u_j \leq v_j$ (resp. $u_j < v_j$) for all $j$.
Given matrices $M,N$ in $\RR^{m \times m'}$, similarly write
$M \leq N$ (resp. $M < N$) if $M_{i,j} \leq N_{i,j}$
(resp. $M_{i,j} < N_{i,j}$) for all $i,j$.

Note that $u \leq v$ and $u \neq v$ do not together imply that $u < v$;
similarly for matrices.
\end{definition}

\subsection{The implication $\text{(i)} \Rightarrow \text{(ii)}$}
This is trivial from Definition~\ref{our-M-matrix-definition}.

\subsection{The implication $\text{(ii)} \Rightarrow \text{(iii)}$}
Since $L_V$ is singular (as $\sss^T L_V=0$), if its 
submatrix $\overline{L_V}$ is nonsingular, 
then $L_V$ has rank $\ell$ and nullity $1$.

\subsection{The equivalence $\text{(iii)} \Leftrightarrow \text{(iv)}$}
For a square integer matrix $L_V$, having nullity $1$ 
is equivalent to its integer cokernel $\ZZ^{\ell+1}/\im(L_V)=\ZZ \oplus K(V)$ 
having free rank $1$, that is, to $K(V)$ being finite.

\subsection{The implication $\text{(iii)} \Rightarrow \text{(v)}$}
We prove the contrapositive: not (v) implies not (iii).

To say that (v) fails, i.e., $V$ is \emph{not} tensor-rich, means that
 the composition factors within the various tensor powers 
 $V^{\otimes k}$ form a nonempty proper subset 
$\{S_j\}_{j \in J}$ of the set of simple $A$-modules 
$\{S_i\}_{i=1,2,\ldots,\ell+1}$.
This implies that the McKay matrix $M_V$ has a nontrivial 
block-triangular decomposition, in the sense that
$(M_V)_{i,j}=0$ for $j \in J$ and $i \not\in J$
\begin{vershort}
(otherwise $(M_V)_{i,j}=[S_j \otimes V:S_i]>0$, which implies 
$[V^{\otimes(k+1)}:S_i]>0$ since $[V^{\otimes k}:S_j]>0$ for 
some $k$).
\end{vershort}
\begin{verlong}
\ \ \ \ \footnote{
Assume the contrary. Thus, there exist $i \notin J$ and $j \in J$
such that
$\left(M_V\right)_{i, j} \neq 0$. Consider these $i$ and $j$.
Since $\left(M_V\right)_{i, j} = \left[ S_j \otimes V : S_i \right]$,
this rewrites as $\left[ S_j \otimes V : S_i \right] \neq 0$. Hence,
$S_i$ appears as a composition factor in the $A$-module
$S_j \otimes V$. Therefore,
$\left[ S_j \otimes V \right] = \left[S_i\right] + \left[W\right]$
in $\groth$ for some $A$-module $W$. Consider this $W$.
But we have $\left[V^{\otimes k} : S_j\right] > 0$
for some $k \in \NN$ (since $j \in J$). Consider this $k$. Thus, $S_j$
appears as a composition factor in $V^{\otimes k}$. Hence,
$\left[ V^{\otimes k} \right] = \left[S_j\right] + \left[X\right]$ in
$\groth$ for some $A$-module $X$. Consider this $X$. Now, in $\groth$,
we have
\begin{align}
\left[ V^{\otimes \left(k+1\right)} \right]
&= \left[ V^{\otimes k} \otimes V \right]
= \left[ V^{\otimes k} \right] \left[ V \right]
= \left( \left[S_j\right] + \left[X\right] \right) \left[ V \right]
\qquad
\left(\text{since }
      \left[ V^{\otimes k} \right] = \left[S_j\right] + \left[X\right]
\right) \nonumber \\
&= \underbrace{\left[ S_j \otimes V \right]}_{ = \left[S_i\right] + \left[W\right]}
    + \left[X \otimes V\right]
= \left[S_i\right] + \left[W\right] + \left[X \otimes V\right]
= \left[S_i \oplus W \oplus X \otimes V\right] .
\label{pf.avalanche-finite-implies-tensor-rich.2.pf.1}
\end{align}
But if $Z$ is an $A$-module, then the number $\left[ Z : S_i \right]$
depends only on the class $\left[Z\right] \in \groth$ (but not on $Z$
itself). (Indeed, this follows from \eqref{eq.alg-reps.V:Si2}.) Hence,
from \eqref{pf.avalanche-finite-implies-tensor-rich.2.pf.1}, we obtain
\begin{align*}
\left[ V^{\otimes \left(k+1\right)} : S_i \right]
&= \left[ S_i \oplus W \oplus X \otimes V : S_i \right]
> 0
\end{align*}
(since $S_i$ clearly does appear in
$S_i \oplus W \oplus X \otimes V$). By the definition of $J$, this
shows that $i \in J$; but this contradicts $i \notin J$. This
contradiction completes our proof.}.
\end{verlong}
This will allow us
to apply the following property of nonnegative matrices.

\begin{vershort}
\begin{lemma} 
\label{lem.matrices.large-kernel}
Let $M \geq 0$ be a nonnegative matrix in $\RR^{m \times m}$,
with a nontrivial block-triangular decomposition:
$\varnothing  \subsetneq J \subsetneq \{1,2,\ldots,m\}$
for which $M_{i,j}=0$ when $j \in J, i \not\in J$.

If $M$ has both positive right- and left-eigenvectors 
$v>0, u>0$ for the same eigenvalue $\lambda$,
meaning that
$$
\begin{aligned}
Mv &= \lambda v,\\
u^T M &= \lambda u^T,
\end{aligned}
$$
then its $\lambda$-eigenspace is not simple, that is,  
$\dim \ker(\lambda I_m-M) \geq 2$.
\end{lemma}

\begin{proof}
Introduce $L = \lambda I_m - M$, so that 
right- and left-nullvectors for $L$ (such as $v, u$) are
right- and left-eigenvectors for $M$ with eigenvalue $\lambda$.
Decompose $v=v'+v''$ where $v',v'' \geq 0$ are
defined by
\[
v'_i
= \begin{cases}
    v_i, & \text{ if } i \in J; \\
      0, & \text{ if } i \notin J,
  \end{cases}
\qquad \qquad
%\text{ for all } i \in \left\{1,2,\ldots,m\right\} .
v''_i
= \begin{cases}
      0, & \text{ if } i \in J; \\
    v_i, & \text{ if } i \notin J.
  \end{cases}
%\qquad \text{ for all } i \in \left\{1,2,\ldots,m\right\} .
\]
Since $0=Lv=Lv'+Lv''$, one has $Lv''=-Lv'$.  Also, note that $v >0$ implies that $v',v''$ have disjoint nonempty supports,
and hence are linearly independent.  Thus it only remains to show that $Lv'=0$.  In fact,
it suffices to check that $Lv' \geq 0$, since $Lv'$ has zero dot product with the positive vector $u>0$:
$$ 
u^T (Lv')=(u^TL)v'=0 \cdot v'=0.
$$
We argue each coordinate $(Lv')_i \geq 0$ in cases, 
depending on whether $i$ lies in $J$ or not.  When $i \not\in J$, one has 
$$
\left(Lv'\right)_i
= \sum_{j=1}^m L_{i, j} \left(v'\right)_j
%= \sum_{j=1}^m L_{i, j}  \begin{cases}
%                            v_j, & \text{ if } j \in J; \\
%                              0, & \text{ if } j \notin J
 %                         \end{cases}
 %\qquad \left(\text{by the definition of } v'\right) \\
= \sum_{j \in J}  L_{i, j} v_j
= 0 
$$
using in the last equality the fact that $L_{i,j}=0$ for $i \not\in J, j \in J$.  When $i \in J$, one has
$$
(Lv')_i= -\left(Lv''\right)_i
= -\sum_{j=1}^m L_{i, j} v''_j 
%\begin{verlong}
%= -\sum_{j=1}^m L_{i, j}  \begin{cases}
%                            0, & \text{ if } j \in J; \\
%                           v_j, & \text{ if } j \notin J
%                        \end{cases}
%\qquad \left(\text{by the definition of } v''\right) 
%\end{verlong}
 = \sum_{j \notin J} (-L_{i, j}) v_j \geq 0 ,
$$
using for the last inequality the facts that $L_{ij} \leq 0$ when $i \in J, j \not\in J$,
and that $v_j  \geq 0$.
\end{proof}
\end{vershort}

\begin{verlong}
\begin{lemma} \label{lem.matrices.large-kernel}
Let $M \geq 0$ be a nonnegative matrix in $\RR^{m \times m}$. Let
$\lambda \in \RR$. Set $L = \lambda I_m - M$. Let
$u \in \RR^m$ and $v \in \RR^m$ be two column vectors such that
$u > 0$, $v > 0$, $u^T L = 0$ and $L v = 0$.

Let $J$ be a nonempty proper subset of
$\left\{1, 2, \ldots, m\right\}$. Assume that
\begin{equation}
M_{i, j} = 0 \qquad \text{ for all } i \notin J \text{ and } j \in J .
\label{eq.lem.matrices.large-kernel.0}
\end{equation}
Then, $\dim \left(\ker L\right) \geq 2$.
\end{lemma}

\begin{proof}
Recall that $L = \lambda I_m - M$.
Hence, for all $i \notin J$ and $j \in J$, we have
\begin{equation}
L_{i, j} = \left(\lambda I_m - M\right)_{i, j}
= \lambda \underbrace{\left(I_m\right)_{i, j}}_{\substack{
            = \delta_{i, j} = 0 \\
            \text{(since } i \notin J \text{ and } j \in J \\
            \text{ lead to } i \neq j \text{)} }}
  - \underbrace{M_{i, j}}_{\substack{ = 0 \\
            \text{(by \eqref{eq.lem.matrices.large-kernel.0})}}}
= 0 .
\label{pf.lem.matrices.large-kernel.1}
\end{equation}
On the other hand, for all $i \in J$ and $j \notin J$, we have
\begin{equation}
L_{i, j} = \left(\lambda I_m - M\right)_{i, j}
= \lambda \underbrace{\left(I_m\right)_{i, j}}_{\substack{
            = \delta_{i, j} = 0 \\
            \text{(since } i \in J \text{ and } j \notin J \\
            \text{ lead to } i \neq j \text{)} }}
  - \underbrace{M_{i, j}}_{\substack{ \geq 0 \\
            \text{(since } M \geq 0 \text{)}}}
\leq 0 .
\label{pf.lem.matrices.large-kernel.2}
\end{equation}

Define a column vector $v' \in \RR^m$ by setting
\[
\left(v'\right)_i
= \begin{cases}
    v_i, & \text{ if } i \in J; \\
      0, & \text{ if } i \notin J
  \end{cases}
\qquad \text{ for all } i \in \left\{1,2,\ldots,m\right\} .
\]
Define a column vector $v'' \in \RR^m$ by setting
\[
\left(v''\right)_i
= \begin{cases}
      0, & \text{ if } i \in J; \\
    v_i, & \text{ if } i \notin J
  \end{cases}
\qquad \text{ for all } i \in \left\{1,2,\ldots,m\right\} .
\]
Clearly, $v = v' + v''$. Moreover, the vectors $v$ and $v'$ are
linearly independent\footnote{To see this, recall that the vector $v$
has all its coordinates nonzero (since $v > 0$), whereas the vector
$v'$ has only a nonempty proper subset of its coordinates nonzero
(namely, the coordinates $\left(v'\right)_i$ for $i \in J$).}. Thus,
$\dim \operatorname{span}_{\RR}\left(v, v'\right) = 2$.

From $v = v' + v''$, we obtain $Lv = Lv' + Lv''$, so that
$Lv' + Lv'' = Lv = 0$ and therefore $Lv' = - Lv''$.

Now, we claim that
\begin{equation}
\left(Lv'\right)_i \geq 0
\qquad \text{ for all } i \in \left\{1, 2, \ldots, m \right\} .
\label{pf.lem.matrices.large-kernel.3}
\end{equation}
To prove this, we handle the cases $i \in J$ and $i \notin J$
separately:
\begin{itemize}
\item Let us consider the case when $i \in J$. In this case, recall
that $Lv' = - Lv''$; hence,
$\left(Lv'\right)_i = - \left(Lv''\right)_i$. Since
\begin{align*}
\left(Lv''\right)_i
&= \sum_{j=1}^m L_{i, j} \left(v''\right)_j
 = \sum_{j=1}^m L_{i, j}  \begin{cases}
                              0, & \text{ if } j \in J; \\
                            v_j, & \text{ if } j \notin J
                          \end{cases}
 \qquad \left(\text{by the definition of } v''\right) \\
&= \sum_{j \notin J}
     \underbrace{L_{i, j}}_{\substack{\leq 0 \\
            \text{(by \eqref{pf.lem.matrices.large-kernel.2})}}}
     \underbrace{v_j}_{\substack{> 0 \\
            \text{(since } v > 0 \text{)}}}
\leq 0 ,
\end{align*}
this results in $\left(Lv'\right)_i \geq 0$. Thus,
\eqref{pf.lem.matrices.large-kernel.3} is proven when $i \in J$.
\item Let us now consider the case when $i \notin J$. In this case,
\begin{align*}
\left(Lv'\right)_i
&= \sum_{j=1}^m L_{i, j} \left(v'\right)_j
 = \sum_{j=1}^m L_{i, j}  \begin{cases}
                            v_j, & \text{ if } j \in J; \\
                              0, & \text{ if } j \notin J
                          \end{cases}
 \qquad \left(\text{by the definition of } v'\right) \\
&= \sum_{j \in J}
     \underbrace{L_{i, j}}_{\substack{= 0 \\
            \text{(by \eqref{pf.lem.matrices.large-kernel.1})}}}
     v_j
= 0 .
\end{align*}
Hence, \eqref{pf.lem.matrices.large-kernel.3} is proven when
$i \notin J$.
\end{itemize}

Thus, \eqref{pf.lem.matrices.large-kernel.3} is proven in all cases.
Consequently, $Lv' \geq 0$.

Now, recall that $u > 0$. Hence, if $x \in \RR^m$ is any column
vector satisfying $x \geq 0$ and $u^T x = 0$, then we must have
$x = 0$ (because a sum of nonnegative reals can only be $0$ if all
of its addends are $0$). Applying this to $x = Lv'$, we find
$Lv' = 0$ (since $Lv' \geq 0$ and $\underbrace{u^T L}_{= 0} v' = 0$).
Combined with $Lv = 0$, this shows that both $v$ and $v'$ lie in
$\ker L$. Therefore, $\ker L \supseteq
\operatorname{span}_{\RR}\left(v, v'\right)$. Hence,
$\dim \left(\ker L\right) \geq
\dim \operatorname{span}_{\RR}\left(v, v'\right) = 2$.
\end{proof}
\end{verlong}

This lets us finish the proof that not (v) implies not (iii):
the discussion preceding Lemma~\ref{lem.matrices.large-kernel} shows that
when $V$ is \textbf{not} tensor-rich, one can apply 
Lemma~\ref{lem.matrices.large-kernel} to $M_V$, with the 
roles of $u,v$ played by $\sss, \ppp$, and conclude that
$L_V=nI_{\ell+1}-M_V$ has nullity at least two.

\begin{noncompile}
% was verlong
\[
G = \left\{ j \in \left\{1, 2, \ldots, \ell+1\right\}
            \mid \left[V^{\otimes k} : S_j\right] > 0
                \text{ for some } k \in \NN \right\} .
\]
Then, $G$ is a \textbf{proper} subset of
$\left\{1, 2, \ldots, \ell+1\right\}$ (since $V$ is not tensor-rich).
Furthermore, $G$ is nonempty (because the simple $A$-module 
$\epsilon$ appears in $V^{\otimes 0}$).

Clearly, $\sss > 0$ and $\ppp > 0$. Furthermore,
$L_V = nI_{\ell+1} - M_V$. Proposition~\ref{left-nullvector-prop}
yields $\left(L_V\right)^T \sss = 0$, hence $\sss^T L_V = 0$.
Proposition~\ref{right-nullvector-prop} yields $L_V \ppp = 0$.

Now, we claim that
\begin{equation}
\left(M_V\right)_{i, j} = 0
\qquad \text{ for all } i \notin G \text{ and } j \in G .
\label{pf.avalanche-finite-implies-tensor-rich.2}
\end{equation}

Indeed, let us prove
\eqref{pf.avalanche-finite-implies-tensor-rich.2}: Assume the
contrary. Thus, there exist $i \notin G$ and $j \in G$ such that
$\left(M_V\right)_{i, j} \neq 0$. Consider these $i$ and $j$.
Since $\left(M_V\right)_{i, j} = \left[ S_j \otimes V : S_i \right]$,
this rewrites as $\left[ S_j \otimes V : S_i \right] \neq 0$. Hence,
$S_i$ appears as a composition factor in the $A$-module
$S_j \otimes V$. Therefore,
$\left[ S_j \otimes V \right] = \left[S_i\right] + \left[W\right]$
in $\groth$ for some $A$-module $W$. Consider this $W$.
But we have $\left[V^{\otimes k} : S_j\right] > 0$
for some $k \in \NN$ (since $j \in G$). Consider this $k$. Thus, $S_j$
appears as a composition factor in $V^{\otimes k}$. Hence,
$\left[ V^{\otimes k} \right] = \left[S_j\right] + \left[X\right]$ in
$\groth$ for some $A$-module $X$. Consider this $X$. Now, in $\groth$,
we have
\begin{align}
\left[ V^{\otimes \left(k+1\right)} \right]
&= \left[ V^{\otimes k} \otimes V \right]
= \left[ V^{\otimes k} \right] \left[ V \right]
= \left( \left[S_j\right] + \left[X\right] \right) \left[ V \right]
\qquad
\left(\text{since }
      \left[ V^{\otimes k} \right] = \left[S_j\right] + \left[X\right]
\right) \nonumber \\
&= \underbrace{\left[ S_j \otimes V \right]}_{ = \left[S_i\right] + \left[W\right]}
    + \left[X \otimes V\right]
= \left[S_i\right] + \left[W\right] + \left[X \otimes V\right]
= \left[S_i \oplus W \oplus X \otimes V\right] .
\label{pf.avalanche-finite-implies-tensor-rich.2.pf.1}
\end{align}
But if $Z$ is an $A$-module, then the number $\left[ Z : S_i \right]$
depends only on the class $\left[Z\right] \in \groth$ (but not on $Z$
itself). (Indeed, this follows from \eqref{eq.alg-reps.V:Si2}.) Hence,
from \eqref{pf.avalanche-finite-implies-tensor-rich.2.pf.1}, we obtain
\begin{align*}
\left[ V^{\otimes \left(k+1\right)} : S_i \right]
&= \left[ S_i \oplus W \oplus X \otimes V : S_i \right]
> 0
\end{align*}
(since $S_i$ clearly does appear in
$S_i \oplus W \oplus X \otimes V$). By the definition of $G$, this
shows that $i \in G$; but this contradicts $i \notin G$. This
contradiction completes our proof of
\eqref{pf.avalanche-finite-implies-tensor-rich.2}.

Thanks to \eqref{pf.avalanche-finite-implies-tensor-rich.2}, we can
now apply Lemma~\ref{lem.matrices.large-kernel} to $\ell+1$, $M_V$,
$n$, $L_V$, $\sss$ and $\ppp$ instead of $m$, $M$, $\lambda$, $L$,
$u$ and $v$. As a result, we obtain
$\dim \left(\ker \left(L_V\right)\right) \geq 2$. Hence,
$\rank \left(L_V\right) \leq \ell-1$. Set $L = L_V$; then this
rewrites as $\rank L \leq \ell-1$.

Now, we want to prove that $K\left(V\right)$ is infinite. Because of
\eqref{eq.K(V).def3}, this boils down to showing that the free
$\ZZ$-module $\im L$ has a smaller rank than $\sss^\perp$. But this is
true, because the rank of $\im L$ is
$\rank L \leq \ell-1 < \ell$, whereas the rank of $\sss^\perp$
is $\ell$.
\end{proof}
\end{noncompile}

%%%%%%%%
\subsection{The implication $\text{(v)} \Rightarrow \text{(i)}$}

Here we will use a nontrivial fact which is part\footnote{See \cite{Grinberg-Mmat} for a more self-contained proof of the implication
we are using, namely that if one has a matrix $Q \in \RR^{\ell \times \ell}$ with $Q_{i,j} \leq 0$ for $i \neq j$,
and a vector $x \in \RR^\ell$ with both $x>0$ and $Qx > 0$, then $Q$ is nonsingular.}
of the equivalence of two characterizations for nonsingular $M$-matrices
given by Plemmons \cite[Thm. 1]{Plemmons};
see his conditions $F_{15}, K_{34}$.

\begin{proposition} \label{prop.nonsingular-Mmatrix-is-nonsingular}
A matrix $Q \in \RR^{\ell \times \ell}$ with nonpositive
off-diagonal entries is a nonsingular $M$-matrix
as in Definition~\ref{our-M-matrix-definition} if and only if
there exists $x \in \RR^\ell$ with both $x>0$ and $Qx > 0$.
\end{proposition}

\begin{noncompile}
\begin{proof}
Assume the contrary. Thus, there exists some nonzero vector
$y \in \RR^\ell$ such that $Qy = 0$. Consider this $y$.

Since $Q$ is a nonsingular $M$-matrix,
there exists an $x \in \RR^\ell$ with both $x>0$ and $Qx > 0$.
Consider this $x$.

There exists some $\mu \in \RR$ such that the vector $x - \mu y$ is
nonnegative but has at least one zero entry\footnote{\textit{Proof.}
The vector $y$ is nonzero. Thus, we can WLOG assume that $y$ has at
least one positive coordinate $y_i > 0$
(otherwise, replace $y$ by $-y$).
Then, let $\mu$ be the minimum of the set
$\left\{ x_i / y_i \mid y_i > 0 \right\}$. It is easy to check that
the vector $x - \mu y$ is nonnegative but has at least one zero
entry.}. Consider this $\mu$. Set $z = x - \mu y$. Thus, the vector
$z$ is nonnegative but has at least one zero entry $z_i$. Furthermore,
$Q z = Q \left(x - \mu y\right) = Q x - \mu \underbrace{Q y}_{= 0}
= Q x > 0$.

But
\begin{align*}
\left(Q z\right)_i 
&= \sum_{j=1}^\ell Q_{i,j} z_j
= Q_{i,i} \underbrace{z_i}_{= 0}
  + \sum_{j \neq i}
      \underbrace{Q_{i,j}}_{\substack{\leq 0 \\
        \text{(since } Q \text{ is a nonsingular } M\text{-matrix)}}}
      \underbrace{z_j}_{\substack{\geq 0 \\ \text{(since }
        z \geq 0 \text{)}}}
\leq 0 + \sum_{j \neq i} 0 = 0 .
\end{align*}
This contradicts $Q z > 0$. This contradiction proves that our
assumption was wrong.
\end{proof}
\end{noncompile}

A few more notations are in order.
For $x$ in $\RR^{\ell+1}$, let $\overline{x}$
be the vector in $\RR^\ell$ obtained by
forgetting its last coordinate.
For $M$ in $\RR^{(\ell+1) \times (\ell+1)}$,
let $\overline{M}$ be the matrix in $\RR^{\ell \times \ell}$ obtained
by forgetting its last row and last column.\footnote{This
notation will not conflict with the notation
$\overline{L_V}$ used (e.g.) in
Theorem~\ref{tensor-rich-equivalences-theorem}
because %(as soon as we resume the proof of the latter theorem)
we shall re-index the simple $A$-modules in such a way
that the last row and the last column of $L_V$ are the ones
corresponding to $\epsilon$.}
Let $M_{*,k}$ denote the vector which is the $k$-th column of $M$.

\begin{proposition}
\label{overline-and-product}
For nonnegative matrices $M,N \geq 0$ both in
$\RR^{(\ell+1) \times (\ell+1)}$, one has
$\overline{M} \cdot \overline{N} \leq \overline{MN}$.
\end{proposition}
\begin{proof}
Compare their $(i,j)$-entries for $i,j \in \{1,2,\ldots,\ell\}$:
\[
\begin{aligned}
 \overline{MN}_{i,j} 
  = ( MN )_{i,j} 
 =\sum_{k=1}^{\ell+1} M_{i,k} N_{k,j}  
 &= M_{i,\ell+1} N_{\ell+1,j} + \sum_{k=1}^\ell M_{i,k} N_{k,j} \\
 &= M_{i,\ell+1} N_{\ell+1,j} + \left( \overline{M} \cdot \overline{N}\right) _{i,j}
 \geq \left( \overline{M} \cdot \overline{N}\right) _{i,j}. \qedhere
 \end{aligned}
\]
\end{proof}

The following gives a useful method to produce
nonsingular $M$-matrices, to be applied to $M=M_V$ below.

\begin{proposition}
\label{avalanche-finite-criteria}
Assume one has an eigenvector equation
\[
Mx = \lambda x
\]
with a nonnegative matrix $M \geq 0$ in $\RR^{(\ell+1) \times (\ell+1)}$,
a real scalar $\lambda$, and a
positive eigenvector $x > 0$ in $\RR^{\ell+1}$.
Let $\overline{L}:=\lambda I_{\ell} - \overline{M}$. 

\begin{enumerate}
\item[(i)]
One always has $\lambda \geq 0$, and
\[
\overline{M}\overline{x} \leq \lambda \overline{x}.
\]
Consequently, $\overline{L}\overline{x} \geq 0$.
\item[(ii)]
Under the additional hypothesis that $M$ has positive last column 
$M_{*,\ell+1}>0$, then 
\[
\overline{M}\overline{x} < \lambda \overline{x}.
\]
Consequently, (under this hypothesis)
$\overline{L}$ is a nonsingular $M$-matrix, 
since both $\overline{x}>0$ and $\overline{L} \overline{x} > 0$.

\item[(iii)]
Let $t$ be a positive integer.
Set
$
\overline{y}:=\sum_{k=0}^{t-1} \overline{M}^k \overline{x}
$.
Then, $\overline{y} > 0$.
Under the additional hypothesis (different from (ii))
that the last column of
$\sum_{k=0}^{t-1} M^k$ is strictly positive, 
we also have
\[
\overline{M} \overline{y} < \lambda \overline{y}.
\]
Consequently, (under this hypothesis)
$\overline{L}$ is a nonsingular $M$-matrix, 
since both $\overline{y} > 0$ and $\overline{L} \overline{y} > 0$.
\end{enumerate}

\end{proposition}

\begin{proof}
The nonnegativity $\lambda \geq 0$ follows
from $M x = \lambda x$ since $M \geq 0$ and $x>0$.

For the remaining assertions in (i) and (ii), 
note that the first $\ell$ equations in the system $Mx=\lambda x$ assert
\[
\overline{M}\overline{x} + \ \overline{M_{*,\ell+1}} x_{\ell+1}
= \lambda \overline{x} ,
\qquad
\text{where } M_{*,\ell+1} \text{ is the last column vector of } M .
\]
Since $x_{\ell+1}>0$, and
since the entries of $M_{*,\ell+1}$ are nonnegative (resp. strictly positive) under the hypotheses in (i)
(resp. in (ii)), the remaining assertions in (i) and (ii) follow.

For assertion (iii), note that 
$\overline{y}=\overline{x}+\sum_{k=1}^{t-1} \overline{M}^k \overline{x}$,
and hence $\overline{y}>0$ follows from
the facts that $\overline{x} >0$ and $\overline{M} \geq 0$.
To prove $\overline{M} \overline{y} < \lambda \overline{y}$,
we first prove a weak inequality as follows.  For each $k=0,1,2,\ldots,t-1$,
multiply the inequality in (i) by $\overline{M}^k$, obtaining:
\begin{equation}
\label{x-power-inequality}
\overline{M}^{k+1}\overline{x} 
 \leq \lambda \overline{M}^k \overline{x} .
\end{equation}
Summing this over all $k$, we find
\[
\sum_{k=0}^{t-1}\overline{M}^{k+1}\overline{x} 
 \leq \sum_{k=0}^{t-1} \lambda \overline{M}^k \overline{x} .
\]
In view of the definition of $\overline{y}$, this can be rewritten as
\begin{equation}
\label{y-inequality}
 \overline{M} \overline{y} 
 \leq \lambda \overline{y}.
\end{equation}
It remains to show that for $1 \leq j \leq \ell$, the inequality 
in the $j^{th}$ coordinate of \eqref{y-inequality} is strict.
For the sake of contradiction, assume   
$\left( \overline{M} \overline{y} \right)_j = \lambda \overline{y}_j$.
This forces equalities in the $j$-th coordinate of \eqref{x-power-inequality} for $0 \leq k \leq t-1$: 
\[
\left(\overline{M}^{k+1}\overline{x} \right)_j
 = \lambda \left( \overline{M}^k \overline{x} \right)_j.
\]
This implies via induction on $k$ that 
$
 \left(\overline{M}^{k}\overline{x} \right)_j
=\lambda^{k} \overline{x}_j,
$
for $k=0,1,2,\ldots,t-1$. Summing on $k$ gives
\[
\left(
\left( \sum_{k=0}^{t-1} \overline{M}^k \right) \overline{x}
\right)_j
 =
\left( \sum_{k=0}^{t-1} \lambda^k \right) \overline{x}_j .
\]
However, this contradicts the strict inequality in the $j^{th}$ coordinate in the following:
\begin{equation}
\label{two-inequalities-together}
\left( \sum_{k=0}^{t-1} \overline{M}^k \right) \overline{x}
 \leq \overline{ \left( \sum_{k=0}^{t-1} M^k \right) } \overline{x} 
  < \left( \sum_{k=0}^{t-1} \lambda^k \right) \overline{x} .
\end{equation}
The first (weak) inequality in \eqref{two-inequalities-together} comes
from the fact that $\overline{M}^k \leq \overline{M^k}$
(which follows by induction
from Proposition~\ref{overline-and-product}),
while the second (strict) inequality comes from applying assertion (ii) to the eigenvector equation 
$
\left( \sum_{k=0}^{t-1} M^k \right) x 
=\left( \sum_{k=0}^{t-1} \lambda^k \right) x
$
(which follows from $M x = \lambda x$).
\end{proof}

We return now to our usual context of a finite-dimensional Hopf algebra $A$ over an
algebraically closed field $\FF$, and 
an $A$-module $V$ of dimension $n$.
Recall the matrices $M_V$ and $L_V$ are given by $(M_V)_{i,j}=[S_j \otimes V:S_i]$
and $L_V:=nI_{\ell+1}-M_V$.
For the remainder of this section,
assume one has indexed the simple $A$-modules $\{S_i\}_{i=1,2,\ldots,\ell+1}$ such that
$S_{\ell+1}=\epsilon$ is the trivial $A$-module on $\FF$.
Thus $\overline{M_V}, \overline{L_V}$ 
come from $M_V, L_V$ by removing the row and column 
indexed by $\epsilon$.

Richness of $V$ has an obvious reformulation in terms of $M_V$.
 \begin{proposition}
\label{rich-implies-positive}
$V$ is rich if and only if the McKay matrix $M_V$ has positive last column $(M_V)_{*,\ell+1}>0$.\end{proposition}
\begin{proof}
\noindent
Using 
\eqref{canonical-triv-tensor-isomorphisms} one has
$[V: S_i]
=[\epsilon \otimes V:S_i]=[S_{\ell+1} \otimes V:S_i]=(M_V)_{i,\ell+1}$ .
\end{proof}

\begin{proof}[Proof of $\text{(v)} \Rightarrow \text{(i)}$] 
Assuming $V$ is tensor-rich, there is some $t>0$ for which
$W:=\bigoplus_{k=0}^{t} V^{\otimes k}$ is rich.
Thus $M_W$ has positive last column $(M_W)_{*,\ell+1}>0$.
In $G_0(A)$, one has
$[W]=\sum_{k=0}^t [V]^k$, giving the matrix equation
$
M_W=\sum_{k=0}^t M_V^k.
$
Since $M_V \ppp = n \ppp$ by Proposition~\ref{right-nullvector-prop},
one can apply Proposition~\ref{avalanche-finite-criteria}(iii), 
with $M=M_V, \lambda=n, x=\ppp$, and
conclude that $\overline{L_V}$ is a nonsingular $M$-matrix.
\begin{noncompile}
% was verlong
Proposition~\ref{right-nullvector-prop}
shows $M_V \ppp = n \ppp$ with $\ppp > 0$. Notice that
$\overline{L_V} = n I_\ell - \overline{M_V}$.
% We claim that the hypothesis of 
% Proposition~\ref{avalanche-finite-criteria}(iii) applies to $M_V$.
From $W = \bigoplus_{k=0}^{t-1} V^{\otimes k}$, we obtain
$[W]= \sum_{k=0}^{t-1} [V]^{k}$ in $G_0(A)$,
so that $M_W=\sum_{k=0}^{t-1} M_V^k$
(since $M_V$ is a matrix representing multiplication on the right by
$[V]$ on $\groth$, and similarly $M_W$ represents multiplication by
$[W]$).
But Proposition~\ref{rich-implies-positive} (applied to $W$ instead
of $V$) shows that the McKay matrix $M_W$ has positive last column
$\left(M_W\right)_{*, \ell+1}$. Since $M_W=\sum_{k=0}^{t-1} M_V^k$,
this means that the last column of $\sum_{k=0}^{t-1} M_V^k$ is
positive. Hence, Proposition~\ref{avalanche-finite-criteria}(iii)
(applied to $M_V$, $n$, $\ppp$ and $\overline{L_V}$
instead of $M$, $\lambda$, $x$ and $\overline{L}$) shows that
$\overline{L_V}$ is a nonsingular $M$-matrix.

By Proposition~\ref{prop.nonsingular-Mmatrix-is-nonsingular}, this
shows that $\overline{L_V}$ is nonsingular, and therefore the matrix
$L_V$ has rank $\geq \ell$. Thus, the group $K(V)$ is finite.
\end{noncompile}
\end{proof}

This completes the proof of Theorem~\ref{tensor-rich-equivalences-theorem}.

%%%%%%%%%%%%%%%%%%%%%
\medskip

Theorem~\ref{tensor-rich-equivalences-theorem}
raises certain questions on finite-dimensional Hopf algebras.

\begin{question}
\label{tensor-rich-questions}
Let $A$ be a finite-dimensional Hopf algebra over an
algebraically closed field.
\begin{itemize}
\item[(i)]
How does one test whether $V$ is tensor-rich in terms
of some kind of character theory for $A$?
\item[(ii)]
Can the nullity of $L_V$ be described in terms of the simple
$A$-modules appearing in $V^{\otimes k}$ for $k \geq 1$?
\end{itemize}
\end{question}

\noindent
Section~\ref{group-algebra-section} 
answers Question~\ref{tensor-rich-questions}(i)
for group algebras $A=\FF G$, via Brauer characters.

\subsection{Non-tensor-rich modules as inflations}

Any module $V$ over an algebra $B$ can be regarded as an
inflation of a faithful $B/\Ann_B V$-module.\footnote{The
  \textit{annihilator} of a $B$-module $V$ is defined to
  be the ideal $\left\{ b \in B \mid bV = 0 \right\}$ of
  $B$. It is denoted by $\Ann_B V$. \par
  A $B$-module $V$ is said to be \textit{faithful} if and
  only if $\Ann_B V = 0$.}
A natural question to ask is whether a similar fact
holds for tensor-rich modules over Hopf algebras.
The annihilator of an $A$-module is always an ideal,
not necessarily a Hopf ideal; thus, a subtler
construction is needed.
The answer is given by part (iv) of the following theorem,
communicated to us by Sebastian Burciu who graciously
allowed us to include it in this paper.

\begin{theorem} \label{thm.burciu}
Let $V$ be an $A$-module.
Let $\omega$ be the map $A \to A$ sending each
$a \in A$ to $a - \epsilon(a) 1$.
Let $J_V = \bigcap_{k \geq 0} \Ann_A(V^{\otimes k})$.

\begin{enumerate}
\item[(i)] We have $J_V = \omega(\LKer_V) A$, where
$\LKer_V = \left\{ a \in A
\mid \sum a_1 \otimes a_2 v = a \otimes v \text{ for all } v
\in V \right\}$.

\item[(ii)] The subspace $J_V$ of $A$ is a Hopf ideal of $A$,
and thus $A / J_V$ is a Hopf algebra.

\item[(iii)] If $J_V = 0$, then $V$ is tensor-rich.
%(This is the essence of
%Steinberg's (3). Also, the converse does not hold, or does it?)

\item[(iv)] The $A$-module $V$ is the inflation of an
$A / J_V$-module via the canonical projection $A \to A / J_V$,
and the latter $A / J_V$-module is tensor-rich.

\item[(v)] Let $J'$ be any Hopf ideal of $A$ such that the
$A$-module $V$ is the inflation of an $A / J'$-module via
the canonical projection $A \to A / J'$.
Then, $J' \subseteq J_V$.
\end{enumerate}
\end{theorem}

Note that part (i) of the theorem allows for actually computing
$J_V$, while the definition of $J_V$ itself involves an uncomputable
infinite intersection.

\begin{proof}[Proof of Theorem~\ref{thm.burciu}.]
Part (i) is \cite[Corollary 2.3.7]{Burciu}.

Part (ii) follows from \cite[Theorem 7 (i)]{Passman-Quinn},
since the family $\left( V^{\otimes n} \right)_{n \geq 0}$
of $A$-modules is clearly closed under tensor products.

Part (iii) is essentially \cite[(3)]{Steinberg-complete},
but let us also prove it for the sake of completeness:
Assume that $J_V = 0$.
Consider any simple $A$-module $S_i$ and the corresponding
primitive idempotent $e_i$ of $A$.
The $A$-module $\bigoplus_{k \geq 0} V^{\otimes k}$ is
faithful (since its annihilator is $J_V = 0$).
Thus, $e_i \cdot \bigoplus_{k \geq 0} V^{\otimes k} \neq 0$.
Thus, there exists some $k \geq 0$ such that
$e_i V^{\otimes k} \neq 0$.
Consider this $k$.
But recall (see, e.g., \cite[Prop. 7.4.1 (3)]{Webb})
that
$\Hom_A\left(Ae, W\right) \cong eW$ %sending $\varphi \mapsto \varphi(e)$,
for any $A$-module $W$ and any idempotent $e$ of $A$.
Thus, $\Hom_A\left(Ae_i, V^{\otimes k}\right)
\cong e_i V^{\otimes k} \neq 0$, so that
$\dim \Hom_A\left(Ae_i, V^{\otimes k}\right) > 0$.
Hence,
\[
\left[ V^{\otimes k} : S_i \right]
= \dim \Hom_A\left(P_i, V^{\otimes k}\right)
= \dim \Hom_A\left(Ae_i, V^{\otimes k}\right) > 0 .
\]
Since we have shown this to hold for each $i$, we thus
conclude that $V$ is tensor-rich.

(iv) Since $J_V \subseteq \Ann_A \left( V \right)$, we
see immediately that $V$ is the inflation of an
$A / J_V$-module $V'$.
It remains to show that this $V'$ is tensor-rich.
But this follows from part (iii), applied to $A / J_V$,
$V'$ and $0$ instead of $A$, $V$ and $J_V$:
Indeed, we have
$0 = \bigcap_{k \geq 0} \Ann_{A / J_V}(\left(V'\right)^{\otimes k})$,
since
$\bigcap_{k \geq 0} \Ann_{A / J_V}(\left(V'\right)^{\otimes k})$
is the projection of
$\bigcap_{k \geq 0} \Ann_A(V^{\otimes k}) = J_V$
onto the quotient ring $A / J_V$,
which projection of course is $J_V / J_V = 0$.

(v) We assumed that the
$A$-module $V$ is the inflation of an $A / J'$-module $V'$ via
the canonical projection $A \to A / J'$.
Thus, for each $k \geq 0$, the $A$-module $V^{\otimes k}$
is the inflation of the $A / J'$-module $\left(V'\right)^{\otimes k}$
via this projection.
Hence, for each $k \geq 0$, we have $J' V^{\otimes k} = 0$.
Thus, $J' \subseteq \bigcap_{k \geq 0} \Ann_A(V^{\otimes k}) = J_V$.
\end{proof}

\subsection{Avalanche-finiteness}

We digress slightly to discuss \emph{avalanche-finite
matrices} and \emph{chip-firing}.

\begin{definition}
An integer nonsingular $M$-matrix is called an \emph{avalanche-finite} matrix;
see \cite[\S 2]{BenkartKlivansReiner}.
\end{definition}

The terminology arises because the integer cokernel
$\ZZ^\ell/\im L$ for an avalanche-finite matrix $L$ in $\ZZ^{\ell \times \ell}$
has certain convenient
coset representatives in $(\ZZ_{\geq 0})^\ell$, 
characterized via their behavior with
respect to the dynamics of a game in which one makes moves
(called \emph{chip-firing} or \emph{toppling} or \emph{avalanches})
that subtract columns of $L$.  One such family of coset representatives
are called \emph{recurrent}, and the other such family are called \emph{superstable}; 
see, e.g., \cite[\S 2, Thm. 2.10]{BenkartKlivansReiner}
for their definitions and further discussion.\footnote{This
terminology harkens back to the theory of chip-firing on graphs
(also known as the theory of sandpiles),
where analogous notions have been known for longer -- see, e.g.,
\cite{chipfiring} or \cite{sandpile-ag}.}

Since $L_V=nI_{\ell+1}-M_V$ always has
$\overline{L_V}$ in $\ZZ^{\ell \times \ell}$,
Theorem~\ref{tensor-rich-equivalences-theorem}
has the following immediate consequence.

\begin{proposition}
For a finite-dimensional Hopf algebra $A$ over an algebraically
closed field, any tensor-rich $A$-module $V$ has
$\overline{L_V}$ avalanche-finite.
\end{proposition}

The problem in applying this
to study the critical group is that
$K(V) \cong \sss^\perp/\im L_V$, which is not
always isomorphic to $\ZZ^\ell/\im \overline{L_V}$.
%(see Example~\ref{sym-example}).
Under certain conditions, they are isomorphic, namely when the 
left-nullvector $\sss$ and the right-nullvector 
$\ppp$ both have their $(\ell+1)^{st}$ coordinate equal to $1$;
see \cite[\S 2, Prop. 2.19]{BenkartKlivansReiner}.  
Since we have indexed the simple $A$-modules in such a way that $S_{\ell+1} = \epsilon$ is
the trivial $A$-module, this condition is equivalent
to the $1$-dimensional trivial $A$-module $\epsilon$ being its own
projective cover $P_\epsilon=\epsilon$.  This, in turn,
is known \cite[Cor.~4.2.13]{Tensor} to be equivalent to semisimplicity of $A$.
For example, this is always the case in the 
setting of \cite{BenkartKlivansReiner},
where $A=\FF G$ was a group algebra of a finite
group and $\FF$ had characteristic zero.

%
%\begin{proposition}
%For finite-dimensional Hopf algebras $A$, one has $P_\epsilon = \epsilon$
%if and only if $A$ is semisimple.  
%\footnote{ This is \cite[Cor.~4.2.13]{Tensor}. }
%\end{proposition}
%\begin{proof}
%When $A$ is a semisimple algebra, 
%any simple module is its own projective cover,
%so $P_\epsilon=\epsilon$.
%
%Conversely, assume $A$ has $P_\epsilon = \epsilon$,
%and let $e$ be an idempotent in $A$
%\footnote{We used $e$ for the identity element of a group $G$.
%Should we avoid this conflict of notation?}
%with $Ae \cong P_\epsilon$.  Thus $Ae \cong \epsilon$,
%which means that $ae=\epsilon(a) e$ for all $a$ in $A$,
%so that $e$ is a \emph{left integral} for $A$ in the sense of
%Larson and Sweedler.  Their version of 
%Maschke's Theorem (see, e.g., \cite[Thm. 2.2.1]{Montgomery}) 
%says that $A$ is semisimple if and only if $\epsilon(e) \neq 0$ in $\FF$.
%However the idempotence of $e$ gives 
%$e=e \cdot e=\epsilon(e) e$, so $\epsilon(e)=1 \neq 0$.
%\end{proof}

%(Should we include the following remark?? Vic is not sure.)\\
In the case where $A$ is semisimple, many of the results on chip-firing 
from \cite[\S5]{BenkartKlivansReiner} remain valid, with the same proofs.
For example, \cite[Prop. 5.16]{BenkartKlivansReiner} explains why removing
the last entry from the last column of $M_V$
gives a \emph{burning configuration} for the avalanche-finite matrix $\overline{L_V}$,
which allows one to easily test when a configuration is recurrent.

\begin{noncompile}
There is another situation where one might want to say
Do we want to say anything about the case where there exists some $P_i$ having
$\dim(P_i)=\gamma$, such as seems to hold for the Steinberg representation of 
${\frak u}({\frak g})$ with $\frak g$ semisimple?
Or that example in symmetric groups that Jia computed?
\end{noncompile}

%%%%%%%%%%%%%%%%%%%%%%%%%%%%%%%%%%%%%%%%%%%%%%%%%%%%%%%%%
\section{Tensor-rich group representations}
\label{group-algebra-section}

Brauer already characterized tensor-rich $\FF G$-modules, 
though he did not state it in these terms.  We need
the following fact, well-known when $\FF$ has characteristic zero
(see, e.g., James and Liebeck~\cite[Thm. 13.11]{JamesLiebeck}),
and whose proof works just as well in positive characteristic.

\begin{proposition}
\label{kernel-via-characters-proposition}
Given a finite group $G$ and $n$-dimensional $\FF G$-module $V$, 
a $p$-regular element $g$ in $G$ acts as $1_V$  on $V$ 
if and only if $\chi_V(g)=n$.
\end{proposition}
\begin{proof}
The forward implication is clear.  For the reverse, if one has
$n=\chi_V(g)=\sum_{i=1}^{n} \widehat{\lambda_i}$, then since each 
$|\widehat{\lambda_i}|=1$, the Cauchy-Schwarz inequality implies that the
$\widehat{\lambda_i}$ are all equal, and hence they must all equal $1$, since
they sum to $n$.  But then $\lambda_i=1$ for all $i$, which means that
$g$ acts as $1_V$ on $V$.
\end{proof}

We also need the following fact about Brauer characters; 
see, e.g., \cite[Prop. 10.2.1]{Webb}.

\begin{theorem}
Given a simple $\FF G$-module $S$ having projective cover $P$,
then any $\FF G$-module $V$ has
\[
[V:S]=\dim \Hom_{\FF G}(P,V) 
     = \frac{1}{\card{G}}\sum_{\substack{p-\text{regular }\\g\in G}} 
          \overline{\chi}_P(g) \chi_V(g).
\]
\end{theorem}

We come now to the characterization of tensor-rich finite group representations.

\begin{theorem}(Brauer \cite[Rmk. 4]{Brauer-tensor-rich})
\label{Brauer-result}
For $\FF$ an algebraically closed
field, and $G$ a finite group,
an $\FF G$-module $V$ is tensor-rich if and only if the only $p$-regular element 
acting as $1_V$ on $V$ is the identity element $e$ of $G$.

More precisely, if the only $p$-regular element in $G$ acting as $1_V$
is the identity $e$, and if the Brauer character values $\chi_V(g)$ 
take on exactly $t$ distinct values, then $\bigoplus_{k=0}^{t-1} V^{\otimes k}$ is rich.
\end{theorem}
\begin{proof}
To see the ``only if'' direction of the first sentence, note that if
some $p$-regular element $g \neq e$ acts as $1_V$ on $V$,
then the action of $G$ on $V$ factors through some nontrivial quotient group
$G/N$ with $g \in N$, and the same is true for $G$ acting 
on every tensor power $V^{\otimes k}$.  Note that not every
simple $\FF G$-module can be the inflation of a simple $\FF[G/N]$-module through the quotient
map $G \rightarrow G/N$ in this way, else the columns
indexed by $e$ and by $g$ in the Brauer character table of $G$ would be equal,
contradicting its invertibility.  Therefore not all simple $\FF G$-modules
can appear in the tensor algebra $T(V)$, that is, $V$ cannot be tensor-rich.

To see the ``if'' direction of the first sentence,
it suffices to show the more precise statement 
in the second sentence. So assume that the only $p$-regular element 
acting as $1_V$ on $V$ is $e$, and label the $t$ distinct 
Brauer character values $\chi_V(g)$ as $a_1,a_2,\ldots,a_t$,
where $a_1=\dim(V)=\chi_V(e)$.  Letting $A_j$ denote the set of
$p$-regular elements $g$ for which $\chi_V(g)=a_j$, 
Proposition~\ref{kernel-via-characters-proposition} implies that $A_1=\{e\}$.
Assuming for the sake of contradiction that
$\bigoplus_{k=0}^{t-1} V^{\otimes k}$ is not rich, then there 
exists some simple $\FF G$-module $S$ such that for $k=0,1,\ldots,t-1$,
one has (with $P$ denoting the projective cover of $S$) the
equality
\[
0=[V^{\otimes k}:S]=\dim \Hom_{\FF G}(P,V^{\otimes k}) =
\frac{1}{\card{G}}\sum_{\substack{p-\text{regular }\\g\in G}} 
          \overline{\chi}_P(g) \chi_{V^{\otimes k}}(g)
=\frac{1}{\card{G}}\sum_{j=1}^t a_j^k \sum_{g \in A_j} \overline{\chi}_P(g). 
\]
Multiplying each of these equations by $\card{G}$, 
one can rewrite this as a matrix system
\[
\left[
\begin{matrix}
1 & 1& \cdots & 1\\
a_1&a_2& \cdots & a_t\\
a_1^2&a_2^2& \cdots & a_t^2\\
\vdots& \vdots& \ddots & \vdots\\
a_1^{t-1}&a_2^{t-1}& \cdots & a_t^{t-1}
\end{matrix}
\right]
\left[
\begin{matrix}
\sum_{g \in A_1} \overline{\chi}_P(g) \\
\vdots\\
\sum_{g \in A_t} \overline{\chi}_P(g) 
\end{matrix}
\right]
=
\left[
\begin{matrix}
0\\
0\\
\vdots\\
0
\end{matrix}
\right].
\]
The matrix on the left governing the system is an invertible 
Vandermonde matrix, forcing 
$\sum_{g \in A_j} \overline{\chi}_P(g) = 0$ for each $j=1,2,\ldots,t$.
However, the $j=1$ case contradicts 
$
\sum_{g \in A_1} \overline{\chi}_P(g) = \overline{\chi}_P(e) = \dim(P) \neq 0.
\qedhere
$
\end{proof}

\begin{noncompile}
% was verlong
Theorem~\ref{Brauer-result} has several corollaries.

The first of these provides an alternative route to the results of
Theorem~\ref{tensor-rich-implies-avalanche-finite} (sans the claim
about $\overline{L_V}$) and
Theorem~\ref{avalanche-finite-implies-tensor-rich} in the case when
$A = \FF G$ is a group algebra:

\begin{corollary}
\label{group-reps-are-tensor-rich-iff-critical-group-finite}
An $\FF G$-module $V$ for a finite group $G$
has $K(V)$ finite if and only if $V$ is tensor-rich.
\end{corollary}
\begin{proof}
$K(V)$ is finite if and only if $L_V=nI_{\ell+1}-M_V$ 
(where $n=\dim(V)$) has rank $\ell$ and nullity $1$.  
But Proposition~\ref{left-eigenvector-prop} or \ref{right-eigenvector-prop} 
shows that the nullity of $L_V$ is the number of 
$p$-regular conjugacy classes $g$ with $\chi_V(g)=n$,
or the number for which $g$ acts as $1_V$ by 
Proposition~\ref{kernel-via-characters-proposition}.
This number is $1$ if and only if $V$ satisfies the
hypothesis for tensor-richness in Theorem~\ref{Brauer-result}.
\end{proof}
\end{noncompile}

\begin{corollary}
Faithful representations $V$ of a finite group $G$ are always tensor-rich.
\end{corollary}

In fact, in characteristic zero, Burnside had shown that 
faithfulness of $V$ is the same as the condition in 
Theorem~\ref{Brauer-result} characterizing tensor-richness.
Hence one can always regard a non-faithful $G$-representation of $V$ 
as a faithful (and hence tensor-rich) representation of 
some quotient $G/N$ where $N$ is the kernel of the representation on $V$.
There is a similar reduction in positive characteristic. 

\begin{proposition}
\label{non-tensor-rich-quotient-group-prop}
For a finite group representation $\rho: G \to GL(V)$ over
a field $\FF$ of characteristic $p$, 
\begin{itemize}
\item the subgroup
$N$ of $G$ generated by the $p$-regular elements in $\ker(\rho)$ is always 
normal, and 
\item 
$\rho$ factors through the representation of the quotient
$\overline{\rho}: G/N \rightarrow GL(V)$ which is tensor-rich.
\end{itemize}
\end{proposition}
\begin{proof}
The subgroup $N$ as defined above is normal 
since its generating set is stable under $G$-conjugation.

To show that the representation $\overline{\rho}: G/N \rightarrow GL(V)$ 
is tensor-rich, by Theorem~\ref{Brauer-result} above, it suffices to check that
if a coset $gN$ in $G/N$ is both $p$-regular 
and has $gN$ in $\ker(\overline{\rho})$ (that is, $g \in \ker(\rho)$), 
then $g \in N$.  The $p$-regularity means $g^m \in N$ for some
$m$ with $\gcd(m,p)=1$.  
Recall (e.g., from \cite[proof of Lemma~9.3.4]{Webb}) 
that one can write $g=ab$ uniquely with $a,b$ both powers of $g$
in which $a$ is $p$-regular, but $b$ is \emph{$p$-singular} (that is, 
$b$ has order a power of $p$).
Since $a$ is a power of $g$, one has $a \in \ker(\rho)$, and
therefore also $a \in N$.  Additionally $b^m=a^{-m} g^m$ must also
lie in $N$.  Since $b$ is $p$-singular, say of order $p^d$, one 
has $1 = x m + y p^d$ with $x,y$ in $\ZZ$, and then
$b=b^{xm+yp^d} = (b^m)^x (b^{p^d})^y = (b^m)^x \in N$.  Hence $g=ab \in N$,
as desired.
\end{proof}

Proposition~\ref{non-tensor-rich-quotient-group-prop}
implies the following fact,
which should be contrasted with
Theorem~\ref{thm.burciu}.

\begin{corollary}
\label{group-algebra-quotient-corollary}
A non-tensor-rich $A$-module $V$ for $A=\FF G$
the group algebra of a finite group is
always a $B$-module for a proper Hopf quotient 
$A \twoheadrightarrow B$, namely the group algebra $B=\FF[G/N]$,
where $N$ is the subgroup generated by the $p$-regular elements in
$G$ that act as $1_V$.
\end{corollary}
\begin{proof}
$B=A/I$ where $I$ is the $\FF$-span 
of $\{g-gn\}_{g \in G, n \in N}$, a two-sided ideal and coideal of
$A=\FF G$.
\end{proof}

\begin{remark}
These last few results relate to a result of Rieffel \cite[Cor. 1]{Rieffel},
asserting that an $A$-module $V$ for a finite-dimensional Hopf algebra $A$ that cannot
be factored through a proper Hopf quotient must be a \emph{faithful} representation of the algebra $A$,
in the sense that the ring map $A \rightarrow \End(V)$ is injective.  He also shows that this implies $V$ is tensor-rich.
However, as he notes there, faithfulness of a finite group representation $G \rightarrow \GL(V)$ over $\FF$
is a weaker condition than faithfulness of the $\FF G$-module $V$ in the above sense.
\end{remark}

%%%%%%%%%%%%%%%%%%%%%%%%%%%%%%%
\section{Appendix A: Hopf algebra proofs}
\label{Hopf-appendix}
%%%%%%%%%%%%%%%%%%%%%%%%%%%%%%%

In this section, we collect proofs for some facts stated in
Section~\ref{Hopf-algebra-subsection} about Hopf algebras.
In fact, we shall prove more general versions of these facts.

\begin{vershort}
To achieve this generality, let us disavow two of the assumptions
made in Subsection~\ref{subsect.conventions} and in
Subsection~\ref{subsect.findim-alg}.
First, we shall not assume $\FF$ to be algebraically closed (instead,
$\FF$ can be any field).
Second, we shall not assume $A$ (or any $A$-module) to be
finite-dimensional unless explicitly required.
Other conventions remain in place (in particular, $A$ is still a
Hopf algebra, and $A$-modules always mean left $A$-modules) as long as
they still make sense.
\end{vershort}

\begin{verlong}
To achieve this generality, we will \textbf{not} follow
the conventions and assumptions made in
Subsection~\ref{subsect.conventions}.
In particular, we shall not require $\FF$ to be algebraically closed.
We shall also not require our Hopf algebra $A$ (and its modules) to be
finite-dimensional unless we explicitly state so.

Instead, let us make the following standing assumptions:
We let $A$ be a Hopf algebra over a field $\FF$.
We denote its counit, its coproduct, and its antipode by
$\epsilon$, $\Delta$ and $\alpha$, respectively
(as in Section~\ref{Hopf-algebra-subsection}).
All tensor products are over $\FF$.
We shall use Sweedler notation for comultiplication in $A$,
writing $\sum a_1 \otimes a_2$ for the coproduct
$\Delta\left(a\right)$ of any $a \in A$.
We denote by $\dim V$
the dimension of an $\FF$-vector space $V$.
The word ``module'' always means ``left module''.
\end{verlong}

\begin{vershort}
Some of what has been said in Subsection~\ref{subsect.findim-alg}
still applies verbatim to our new general setting:
The antipode $\alpha$ of $A$ is still an algebra anti-endomorphism and
a coalgebra anti-endomorphism of $A$.
The trivial $A$-module $\epsilon$ is still well-defined, as is the
subspace $V^A$ of $A$-fixed points of any $A$-module $V$.
The tensor product $V \otimes W$ of two $A$-modules $V$ and $W$ is
still defined in the same way (and still satisfies the
associativity law and the canonical isomorphisms
\eqref{canonical-triv-tensor-isomorphisms}); so is the $A$-module
$\Hom_\FF(V, W)$. We shall abbreviate the latter as $\Hom(V, W)$.
The left-dual $V^*$ of an $A$-module $V$ is also well-defined.

However, the antipode $\alpha$ of $A$ is no longer necessarily
bijective. As a consequence, the right-dual $\leftast{V}$ is not
well-defined in general.
Lemma~\ref{lem.Vstarstar} no longer holds (unless both $V$ is
finite-dimensional and $\alpha$ is bijective).
The $\epsilon^* \cong \epsilon$ part of Lemma~\ref{lem.epsilonstar}
still holds, but the $\leftast{\epsilon} \cong \epsilon$ part only
holds if $\alpha$ is bijective.
As for Lemma~\ref{lem.VtimesA}, part (ii) holds in full generality,
while part (i) requires $\alpha$ to be bijective. (The same proof
applies.)
Lemma~\ref{lem.hom-tensor-iso} still holds if at least one of $V$ and
$W$ is finite-dimensional (otherwise, $\Phi$ is merely an
$A$-module homomorphism, not an isomorphism).
Lemma~\ref{lem:AHom} holds in full generality (and the proof in
\cite[Lemma 4.1]{Schneider} still applies).
\end{vershort}

\begin{verlong}
We recall the following basic properties of Hopf algebras:
%(see Section~\ref{Hopf-algebra-subsection} for references about them):
\begin{itemize}
\item The antipode $\alpha : A \to A$ is an algebra
      anti-endomorphism and a coalgebra anti-endomorphism of $A$.
      (This is proven, e.g., in
      \cite[Proposition 4.2.6]{DascalescuEtAl} or in
      \cite[Theorem 1.5]{Schneider}.)
\item If $A$ is finite-dimensional, then the antipode
      $\alpha : A \to A$ is bijective.
      (See, e.g., \cite[Thm. 2.1.3]{Montgomery} or
      \cite[Thm. 7.1.14 (b)]{Radford} or \cite[Prop. 4]{Pareigis}
      or \cite[Thm. 2.3. 2)]{Schneider}
      or \cite[Prop. 5.3.5]{Tensor} for proofs of this fact.)
\item The vector space $\FF$ becomes an $A$-module,
      by letting $A$ act on $\FF$ through the algebra homomorphism
      $\epsilon$.
      This $A$-module is denoted by $\epsilon$, and called the
      \textit{trivial $A$-module}.
\item For each $A$-module $V$, we define its
      \textit{subspace of $A$-fixed points} to be
      \[
      V^A := \{v \in V: av=\epsilon(a)v \text{ for all }a \in A\} .
      \]
\item If $V$ and $W$ are two $A$-modules, then their tensor product
      $V \otimes W$ becomes an $A$-module according to the rule
      $a \left(v \otimes w\right)
       := \sum a_1 v \otimes a_2 w$ for each $a \in A$,
      $v \in V$ and $w \in W$.
      (This can be restated more abstractly as follows:
      The $A$-module structure on $V \otimes W$ is obtained from
      the obvious $A \otimes A$-module structure on $V \otimes W$
      by restriction along the $\FF$-algebra homomorphism
      $\Delta : A \to A \otimes A$.) \par
      This concept of tensor products of $A$-modules satisfies the
      associativity law (more precisely: if $U$, $V$ and $W$ are three
      $A$-modules, then the canonical $\FF$-vector space isomorphism
      $\left(U \otimes V\right) \otimes W
         \to U \otimes \left(V \otimes W\right)$
      is an $A$-module isomorphism), thus allowing us to write
      tensor products of multiple factors without parenthesizing them.
      Furthermore, the canonical isomorphisms
      \eqref{canonical-triv-tensor-isomorphisms} hold for every
      $A$-module $V$.
      However, the tensor product is not generally commutative
      (indeed, the $A$-modules $V \otimes W$ and $W \otimes V$ may
      be non-isomorphic).
\item For any $A$-module $V$, the dual space $\Hom_\FF(V,\FF)$
      becomes an $A$-module according to the rule
      $\left(af\right) \left(v\right)
       := f \left( \alpha(a) v \right)$
      for each $a \in A$, $f \in \Hom_\FF(V,\FF)$ and $v \in V$.
      This $A$-module $\Hom_\FF(V,\FF)$ is denoted by $V^*$, and is
      called the \textit{left-dual} of $V$.
      (This is a well-defined $A$-module because
      $\alpha : A \to A$ is an algebra anti-homomorphism.)
\item If $A$ is finite-dimensional (so that $\alpha$ is bijective),
      then there is also another $A$-module structure on the dual
      space $\Hom_\FF(V,\FF)$ of an $A$-module $V$: Namely, we define
      it by the rule
      $\left(af\right) \left(v\right)
       := f \left( \alpha^{-1}(a) v \right)$
      for each $a \in A$, $f \in \Hom_\FF(V,\FF)$ and $v \in V$.
      This $A$-module $\Hom_\FF(V,\FF)$ is denoted by $\leftast{V}$,
      and is called the \textit{right-dual} of $V$.
      (This is a well-defined $A$-module because
      $\alpha^{-1} : A \to A$ is an algebra anti-homomorphism.)
\item For any two $A$-modules $V$ and $W$, we define an $A$-module
      structure on the vector space $\Hom_\FF(V,W)$ via 
      \[
      (a\varphi)(v) := \sum a_1 \varphi(\alpha(a_2)v)
      \]
      for all $a \in A$, $\varphi \in \Hom_\FF(V,W)$ and $v \in V$.
      (See Lemma~\ref{lem.internal-Hom} below for a proof of the
      fact that this is a well-defined $A$-module structure.)
      Sometimes, we will simply write $\Hom(V, W)$ for this $A$-module
      $\Hom_\FF(V, W)$.
\end{itemize}

The following fact is folklore, but an explicit mention is hard to
find in the literature:

\begin{lemma} \label{lem.internal-Hom}
Let $V$ and $W$ be two $A$-modules.
For any $a \in A$ and $\varphi \in \Hom_\FF(V,W)$, define an
element $a \varphi \in \Hom_\FF(V, W)$ by
\[
(a\varphi)(v) := \sum a_1 \varphi(\alpha(a_2)v)
\qquad \text{ for all } v \in V .
\]
This defines an $A$-module structure on $\Hom_\FF(V, W)$.
\end{lemma}

\begin{proof}[Proof of Lemma~\ref{lem.internal-Hom}.]
The definition of $a \varphi$ (for $a \in A$ and
$\varphi \in \Hom_\FF(V,W)$) can be rewritten as follows:
$a \varphi$ is the image of $a$ under the composition
\[
\xymatrix{
A \ar[r]^-{\Delta} & A \otimes A \ar[r]^{\id \otimes \alpha} &
    A \otimes A \ar[r] & \Hom_\FF(V,W)
} ,
\]
where the last arrow is the $\FF$-linear map
$A \otimes A \to \Hom_\FF(V,W)$ sending each
$b \otimes c \in A \otimes A$ to the map
$V \to W, \  v \mapsto b \varphi\left(c v\right)$.
Thus, $a \varphi$ depends $\FF$-linearly on $a$.
Also, $a \varphi$ depends $\FF$-linearly on $\varphi$ (this is clear
from the definition).

Furthermore, $1 \varphi = \varphi$ for each
$\varphi \in \Hom_\FF(V,W)$.
(This follows easily from $\Delta(1) = 1 \otimes 1$ and
$\alpha(1) = 1$.)

It thus remains to show that
$\left(ab\right) \varphi = a \left(b \varphi\right)$ for any
$a \in A$, $b \in A$ and $\varphi \in \Hom_\FF(V,W)$.

For this purpose, let us fix
$a \in A$, $b \in A$ and $\varphi \in \Hom_\FF(V,W)$.
Also, fix $v \in V$.

Now, the definition of $a \left(b \varphi\right)$ yields
\begin{align*}
\left( a \left(b \varphi\right) \right) \left(v\right)
&= \sum a_1 \underbrace{\left(b \varphi\right) \left( \alpha\left(a_2\right) v \right)}
                       _{\substack{
                          = \sum b_1 \varphi\left( \alpha\left(b_2\right) \alpha\left(a_2\right) v \right) \\
                          \text{(by the definition of } b \varphi \text{)}}} \\
&= \sum \sum a_1 b_1 \varphi\left( \underbrace{\alpha\left(b_2\right) \alpha\left(a_2\right)}
                                              _{\substack{ = \alpha\left( a_2 b_2 \right) \\
                                                    \text{(since } \alpha \text{ is an algebra anti-endomorphism)}}}
                                    v \right) \\
&= \sum \sum a_1 b_1 \varphi\left( \alpha \left( a_2 b_2 \right) v \right)
 = \sum \left(ab\right)_1 \varphi \left( \alpha \left( \left(ab\right)_2 \right) v \right)
\end{align*}
(since one of the axioms of a bialgebra yields
$\sum \sum a_1 b_1 \otimes a_2 b_2 = \sum \left(ab\right)_1 \otimes \left(ab\right)_2$).
Comparing this with
\[
\left( \left( ab \right) \varphi \right) \left(v\right)
= \sum \left(ab\right)_1 \varphi \left( \alpha \left( \left(ab\right)_2 \right) v \right)
\qquad \left(\text{by the definition of } \left( ab \right) \varphi \right) ,
\]
we obtain
$\left( \left( ab \right) \varphi \right) \left(v\right)
= \left( a \left(b \varphi\right) \right) \left(v\right)$.
Since this holds for all $v \in V$, we thus have
$\left(ab\right) \varphi = a \left(b \varphi\right)$. This completes
our proof.
\end{proof}

Let us now generalize Lemma~\ref{lem.VtimesA}:

\begin{lemma}
\label{lem.VtimesA.gen}
Let $V$ be an $A$-module.
\begin{enumerate}
\item[(i)] If $A$ is finite-dimensional, then
$V \otimes A \cong A^{\oplus \dim V}$ as $A$-modules.
\item[(ii)] Also, $A \otimes V \cong A^{\oplus \dim V}$
as $A$-modules (independently of the dimension of $A$).
\end{enumerate}
\end{lemma}

\begin{proof}[Proof of Lemma~\ref{lem.VtimesA.gen}.]
(i) Recall that the antipode $\alpha : A \to A$ is bijective (since
$A$ is finite-dimensional). Hence, its inverse $\alpha^{-1}$ is
well-defined and an algebra anti-homomorphism (since $\alpha$ is an
algebra anti-homomorphism).

The defining property of the antipode $\alpha$ of $A$ shows that
\[
\sum a_1 \alpha(a_2) = \sum \alpha(a_1) a_2 = \epsilon(a) 1
\]
for each $a \in A$. Applying the map $\alpha^{-1}$ to this chain of
equalities, we obtain
\[
\alpha^{-1} \left( \sum a_1 \alpha(a_2) \right)
= \alpha^{-1} \left( \sum \alpha(a_1) a_2 \right)
= \alpha^{-1} (\epsilon(a) 1) .
\]
Since $\alpha^{-1}$ is an algebra anti-homomorphism,
this can be rewritten as
\begin{equation}
\sum a_2 \alpha^{-1} (a_1) = \sum \alpha^{-1} (a_2) a_1
= \epsilon(a) 1 .
\label{pf.lem.VtimesA.beq}
\end{equation}

Let $\underline{V}$ denote $V$ considered as a mere $\FF$-vector
space, without $A$-module structure. Then,
$\underline{V} \otimes A$ is a well-defined $A$-module, isomorphic to
$A^{\oplus \dim V}$. It remains to prove the isomorphism
$V \otimes A \cong \underline{V} \otimes A$.

Define an $\FF$-linear map
$\Phi : \underline{V} \otimes A \to V \otimes A$ by setting
$\Phi\left(v \otimes b\right) = \sum b_1 v \otimes b_2$ for all
$v \in V$ and $b \in A$. Then, $\Phi$ is $A$-linear (as follows from
a straightforward argument using $\sum \sum a_1 b_1 \otimes a_2 b_2
= \sum (ab)_1 \otimes (ab)_2$).

Define an $\FF$-linear map
$\Psi : V \otimes A \to \underline{V} \otimes A$ by setting
$\Psi\left(v \otimes b\right) = \sum \alpha^{-1}(b_1) v \otimes b_2$
for all $v \in V$ and $b \in A$. Then,
every $v \in V$ and $b \in A$ satisfy
\begin{align*}
\Phi\left(\Psi\left(v \otimes b\right)\right)
&= \Phi\left(\sum \alpha^{-1}(b_1) v \otimes b_2\right)
= \sum \underbrace{\Phi\left(\alpha^{-1}(b_1) v \otimes b_2\right)}_{= \sum (b_2)_1 \alpha^{-1}(b_1) v \otimes (b_2)_2}
= \sum \sum (b_2)_1 \alpha^{-1}(b_1) v \otimes (b_2)_2 \\
&= \sum \underbrace{\sum (b_1)_2 \alpha^{-1}((b_1)_1)}_{\substack{= \epsilon(b_1) 1\\ \text{(by \eqref{pf.lem.VtimesA.beq}, applied to } a = b_1\text{)}}}
 v \otimes b_2
= \sum \epsilon(b_1) 1 v \otimes b_2
= v \otimes \left(\sum \epsilon(b_1) b_2\right)
= v \otimes b
\end{align*}
(where the fourth equality sign relied on the coassociativity of $A$
in the form $\sum \sum b_1 \otimes (b_2)_1 \otimes (b_2)_2
= \sum \sum (b_1)_1 \otimes (b_1)_2 \otimes b_2$, and the seventh
relied on $\sum \epsilon(b_1) b_2 = b$). Thus, $\Phi \circ \Psi = \id$.
A similar computation reveals that $\Psi \circ \Phi = \id$, whence we
see that $\Phi$ and $\Psi$ are mutually inverse bijections. Since
$\Phi$ is $A$-linear, they are thus $A$-isomorphisms, and hence
$V \otimes A \cong \underline{V} \otimes A$ is proven.

(ii) A similar
argument shows that
$A \otimes V \cong A \otimes \underline{V} \cong A^{\oplus \dim V}$
via the isomorphisms
$A \otimes \underline{V} \to A \otimes V,
\ b \otimes v \mapsto \sum b_1 \otimes b_2 v$ and
$A \otimes V \to A \otimes \underline{V},
\ b \otimes v \mapsto \sum b_1 \otimes \alpha(b_2) v$.
\end{proof}

\begin{noncompile}
\begin{proof}[Proof of Lemma~\ref{lem.VtimesA.gen} (Vic's old version).]
We prove (i);  the proof of (ii) is similar.  
Let $\underline{V}$ denote $V$ considered as a mere $\FF$-vector
space, without $A$-module structure. Then,
$\underline{V} \otimes A$ is a well-defined $A$-module, isomorphic to
$A^{\oplus \dim V}$. It remains to prove the isomorphism
$V \otimes A \cong \underline{V} \otimes A$, which we do by checking that
these two $\FF$-linear maps $\Phi, \Psi$ are both $A$-linear, and mutually 
inverse bijections:
$$
\begin{array}{rcl}
\underline{V} \otimes A 
 &\overset{\Phi}{\longrightarrow} &V \otimes A \\
 & & \\
v \otimes b &\longmapsto &\sum b_1 v \otimes b_2,\\
\end{array}
\qquad \qquad \qquad
\begin{array}{rcl}
V \otimes A  &\overset{\Psi}{\longrightarrow}& \underline{V} \otimes A\\
 & & \\
v \otimes b& \longmapsto & \sum \alpha^{-1}(b_1) v \otimes b_2.
\end{array}
$$
The $A$-linearity of $\Phi$ can be deduced 
 using 
$\sum \sum a_1 b_1 \otimes a_2 b_2
= \sum (ab)_1 \otimes (ab)_2$.
Then $A$-linearity of $\Psi$ will follow
from $\Psi=\Phi^{-1}$.  
We check $\Phi \circ \Psi = \id$;
proving $\Psi \circ \Phi = \id$ is similar.  
We use the fact that any $a$ in $A$ has
\begin{equation}
\sum a_2 \alpha^{-1} (a_1) = \sum \alpha^{-1} (a_2) a_1
= \epsilon(a) 1,
\label{pf.lem.VtimesA.beq}
\end{equation}
deduced by applying the 
algebra anti-automorphism $\alpha^{-1}$ to the equalities 
$
\sum a_1 \alpha(a_2) = \sum \alpha(a_1) a_2 = \epsilon(a) 1
$
that come from the definition of 
$\alpha$.
One then has
\begin{align*}
\Phi\left(\Psi\left(v \otimes b\right)\right)
&= \Phi\left(\sum \alpha^{-1}(b_1) v \otimes b_2\right)\\
&= \sum 
%\underbrace{
           \Phi\left(\alpha^{-1}(b_1) v \otimes b_2\right)
%}_{= \sum (b_2)_1 \alpha^{-1}(b_1) v \otimes (b_2)_2} 
\\
&= \sum \sum (b_2)_1 \alpha^{-1}(b_1) v \otimes (b_2)_2 \\
&= \sum 
%\underbrace{
            \sum (b_1)_2 \alpha^{-1}((b_1)_1)
%}_{\substack{= \epsilon(b_1) 1\\ 
%     \text{(by \eqref{pf.lem.VtimesA.beq}, applied to } a = b_1\text{)}}}
 v \otimes b_2 \\
&= \sum \epsilon(b_1) 1 v \otimes b_2\\
&= v \otimes \left(\sum \epsilon(b_1) b_2\right)
= v \otimes b
\end{align*}
where the fourth equality used the coassociativity 
$\sum \sum b_1 \otimes (b_2)_1 \otimes (b_2)_2
= \sum \sum (b_1)_1 \otimes (b_1)_2 \otimes b_2$, and the 
fifth equality applied \eqref{pf.lem.VtimesA.beq} with
$a = b_1$. 
% The seventh equality relied on $\sum \epsilon(b_1) b_2 = b$. 
(ii) A similar
argument shows that
$A \otimes V \cong A \otimes \underline{V} \cong A^{\oplus \dim V}$
via the isomorphisms
$A \otimes \underline{V} \to A \otimes V,
\ b \otimes v \mapsto \sum b_1 \otimes b_2 v$ and
$A \otimes V \to A \otimes \underline{V},
\ b \otimes v \mapsto \sum b_1 \otimes \alpha(b_2) v$.
\end{proof}
\end{noncompile}

Next, let us generalize Lemma~\ref{lem.epsilonstar}:

\begin{lemma} \phantomsection
\label{lem.epsilonstar.gen}
\begin{itemize}
\item[(i)] We have an $A$-module isomorphism
$\epsilon^* \cong \epsilon$.
\item[(ii)] Assume that $A$ is finite-dimensional.
Then, we have an $A$-module isomorphism
$\leftast{\epsilon} \cong \epsilon$.
\end{itemize}
\end{lemma}

\begin{proof}[Proof of Lemma~\ref{lem.epsilonstar.gen}.]
(i)
We have $\epsilon \circ \alpha = \epsilon$ (since $\alpha$ is a
coalgebra anti-homomorphism), and thus $\epsilon^* \cong \epsilon$
(via the canonical isomorphism $\FF^* \cong \FF$).

(ii) From $\epsilon \circ \alpha = \epsilon$, we obtain
$\epsilon \circ \alpha^{-1} = \epsilon$, and therefore
$\leftast{\epsilon} \cong \epsilon$.
% By \cite[Proposition 1.46]{GrinbergReiner}, the map
% $\epsilon\circ\alpha$ is the convolutional inverse to $\epsilon$,
% which is $\epsilon$ itself since $\sum \epsilon(a_1)a_2 = \sum a_1\epsilon(a_2) =a$ implies 
% $\sum \epsilon(a_1)\epsilon(a_2)= \epsilon(a)$ for all $a\in A$.
% This shows $\epsilon^*=\epsilon$.
% %Applying $S^{-1}$ to $\sum S(a_1)a_2 = u(\epsilon(a)) = \sum a_1S(a_2)$ gives $\sum S^{-1}(a_2)a_1 = u(\epsilon(a)) = \sum a_2 S^{-1}(a_1)$.
% Since $\alpha^{-1}$ is the inverse of $\id_A$ under the \emph{twisted convolution map} (the usual convolution map composed with the twist map), adapting the proof of~\cite[Proposition 1.46]{GrinbergReiner} we have that $\epsilon\circ \alpha^{-1}$ is the twisted convolutional inverse of $\epsilon$, which is again $\epsilon$ itself.
\end{proof}

Next, we restate and prove Lemma~\ref{lem.Vstarstar}:

\begin{lemma}
\label{lem.Vstarstar.gen}
Assume that $A$ is finite-dimensional.
Let $V$ be a finite-dimensional $A$-module.
We have canonical $A$-module isomorphisms
$\leftast{(V^*)} \cong V \cong (\leftast{V})^*$.
\end{lemma}

\begin{proof}[Proof of Lemma~\ref{lem.Vstarstar.gen}.]
There is a linear isomorphism $\phi:V\to\leftast{(V^*)}$ defined by
$\phi(v)(f) = f(v)$ for all $v\in V$ and $f\in V^*$.
This isomorphism is $A$-equivariant, since each $a\in A$,
$v \in V$ and $f \in V^\ast$ satisfy
\[ (a\phi(v))(f) = \phi(v)(\alpha^{-1}(a)f) = (\alpha^{-1}(a)f)(v) = f(\alpha(\alpha^{-1}(a))v) = f(av) = \phi(av)(f). \]
This proves $\leftast{(V^*)} \cong V$.
The proof of $(\leftast{V})^* \cong V$ is similar
(again, the same $\phi$ works).
\end{proof}

Next, we shall show a generalization of
Lemma~\ref{lem.hom-tensor-iso}:

\begin{lemma}
\label{lem.hom-tensor-iso.gen}
Let $V$ and $W$ be two $A$-modules.
Consider the $\FF$-linear map
\begin{equation}
\Phi: W \otimes V^* \to \Hom_\FF(V,W)
\end{equation}
sending each $w \otimes f$ (with $w\in W$ and $f\in V^*$)
to the linear map $\varphi \in \Hom_\FF(V,W)$
that is defined by
$\varphi(v)=f(v)w$ for all $v\in V$.

\begin{itemize}
\item[(i)] This map $\Phi$ is an $A$-module homomorphism.
\item[(ii)] Assume that at least one of the vector spaces $V$ and $W$
is finite-dimensional. Then, $\Phi$ is an $A$-module isomorphism,
and therefore $W \otimes V^* \cong \Hom_\FF(V,W)$ as $A$-modules.
\end{itemize}
\end{lemma}

\begin{proof}[Proof of Lemma~\ref{lem.hom-tensor-iso.gen}.]
(i) Every $v \in V$, $w \in W$,
$f \in V^*$ and $a \in A$ satisfy
\[ \begin{aligned}
\Phi(a(w\otimes f))(v) &= \sum \Phi(a_1w \otimes a_2f)(v) = \sum (a_2f)(v) a_1w = \sum f(\alpha(a_2)v)a_1w \\
&= \sum a_1 f(\alpha(a_2)v) w = \sum a_1\Phi(w\otimes f)(\alpha(a_2)v) = (a\Phi(w\otimes f))(v).
\end{aligned} \]
Hence, every $w \in W$, $f \in V^*$ and $a \in A$ satisfy
$\Phi(a(w\otimes f)) = a\Phi(w\otimes f)$.
By linearity, this entails that every $t \in W \otimes V^*$ and
$a \in A$ satisfy $\Phi(at) = a\Phi(t)$.
In other words, the map $\Phi$ is $A$-equivariant.

(ii) We have assumed that at least one of the vector spaces $V$ and
$W$ is finite-dimensional.
Thus, we know from linear algebra that $\Phi$ is a vector space
isomorphism.\footnote{For instance, its inverse can be constructed as
follows:
\begin{itemize}
\item If $V$ is finite-dimensional, then we can pick a basis
$\{v_i\}$ of the $\FF$-vector space $V$, and the
corresponding dual basis $\{v_i^*\}$ of $V^*$. Then, the inverse of
$\Phi$ can be defined by
$\Phi^{-1}(h) = \sum_i h(v_i)\otimes v_i^*$ for all $h\in\Hom_\FF(V,W)$.
\item If $W$ is finite-dimensional, then we can pick a basis
$\{w_j\}$ of the $\FF$-vector space $W$, and the corresponding dual
basis $\{w_j^*\}$ of $W^*$.
Then, the inverse of $\Phi$ can be defined by
$\Phi^{-1}(h) = \sum_j w_j \otimes \left(w_j^\ast \circ h\right)$
for all $h \in \Hom_\FF\left(V, W\right)$.
\end{itemize}
}
Hence, $\Phi$ is an $A$-module isomorphism (since
Lemma~\ref{lem.hom-tensor-iso.gen} (i) shows that $\Phi$ is an
$A$-module homomorphism).
\end{proof}

\begin{corollary} \label{cor.Vstar=Hom.gen}
For any $A$-module $V$, we have $V^* \cong\Hom_\FF(V,\epsilon)$.
\end{corollary}

\begin{proof}[Proof of Corollary~\ref{cor.Vstar=Hom.gen}.]
Clearly, the $A$-module $\epsilon$ is finite-dimensional. Hence,
Lemma~\ref{lem.hom-tensor-iso.gen} (ii) (applied to $W = \epsilon$)
shows that the map $\Phi: \epsilon \otimes V^* \to
\Hom_\FF(V, \epsilon)$
(defined as in Lemma~\ref{lem.hom-tensor-iso.gen})
is an $A$-module isomorphism.
Thus, $\Hom_\FF(V, \epsilon) \cong \epsilon \otimes V^* \cong V^*$
(by \eqref{canonical-triv-tensor-isomorphisms}, applied to $V^*$
instead of $V$)
as $A$-modules.
\end{proof}

The following lemma generalizes Lemma~\ref{lem:AHom}:

\begin{lemma}\label{lem:AHom.gen}
Let $V$ and $W$ be two $A$-modules. Then,
$\Hom_A (V, W) = \Hom_\FF (V, W)^A$.
\end{lemma}

The following proof of Lemma~\ref{lem:AHom.gen} is taken from
\cite[Lemma 4.1]{Schneider} (where the lemma is stated in far lesser
generality, but the proof equally applies in the general setting):

\begin{proof}[Proof of Lemma~\ref{lem:AHom.gen}.]
If $\varphi\in\Hom_A(V,W)$ then $\varphi\in \Hom_\FF(V,W)^A$,
since each $a \in A$ and $v \in V$ satisfy
\[
(a\varphi)(v)
= \sum a_1 \underbrace{\varphi(\alpha(a_2)v)}
                      _{\substack{
                        = \alpha(a_2) \varphi(v) \\
                        \text{(since }
                        \varphi\in\Hom_A(V,W) \text{)}}}
= \sum a_1 \alpha(a_2) \varphi(v)
= \epsilon(a) \varphi(v).
\]
This proves $\Hom_A (V, W) \subseteq \Hom_\FF (V, W)^A$.

Conversely, let $\psi\in \Hom_\FF(V,W)^A$. Thus,
\begin{equation}
\epsilon\left(  b\right)  \psi=b\psi\qquad\text{for each }b\in A .
\label{pf.lem:AHom.back.1}
\end{equation}

Using $a=\sum\epsilon\left(  a_{1}\right)  a_{2}$, we find
\begin{align*}
\psi\left(  av\right)    & =\psi\left(  \sum\epsilon\left(  a_{1}\right)
a_{2}v\right)  =\sum\underbrace{\epsilon\left(  a_{1}\right)  \psi\left(
a_{2}v\right)  }_{=\left(  \epsilon\left(  a_{1}\right)  \psi\right)  \left(
a_{2}v\right)  }=\sum\underbrace{\left(  \epsilon\left(  a_{1}\right)
\psi\right)  }_{\substack{=a_{1}\psi\\\text{(by (\ref{pf.lem:AHom.back.1}),
applied to }b=a_{1}\text{)}}}\left(  a_{2}v\right)  \\
& =\sum\underbrace{\left(  a_{1}\psi\right)  \left(  a_{2}v\right)
}_{\substack{=\sum\left(  a_{1}\right)  _{1}\psi\left(  \alpha\left(  \left(
a_{1}\right)  _{2}\right)  a_{2}v\right)  \\\text{(by the definition of the
}A\text{-action}\\\text{on }\Hom_{\mathbb{F}}\left(
V,W\right)  \text{)}}}=\sum\sum\left(  a_{1}\right)  _{1}\psi\left(
\alpha\left(  \left(  a_{1}\right)  _{2}\right)  a_{2}v\right)  \\
& =\sum\sum a_{1}\psi\left(  \alpha\left(  \left(  a_{2}\right)  _{1}\right)
\left(  a_{2}\right)  _{2}v\right)  \\
& \ \ \ \ \ \ \ \ \ \ \left(
\begin{array}
[c]{c}%
\text{by the coassociativity law }\sum\sum\left(  a_{1}\right)  _{1}%
\otimes\left(  a_{1}\right)  _{2}\otimes a_{2}=\sum\sum a_{1}\otimes\left(
a_{2}\right)  _{1}\otimes\left(  a_{2}\right)  _{2}\\
\text{in the coalgebra }A
\end{array}
\right)  \\
& =\sum a_{1}\psi\left(  \underbrace{\sum\alpha\left(  \left(  a_{2}\right)
_{1}\right)  \left(  a_{2}\right)  _{2}}_{\substack{=\epsilon\left(
a_{2}\right)  1\\\text{(since }\sum\alpha\left(  b_{1}\right)  b_{2}%
=\epsilon\left(  b\right)  1\\\text{for every }b\in A\text{)}}}v\right)  =\sum
a_{1}\psi\left(  \epsilon\left(  a_{2}\right)  v\right)  =\underbrace{\sum
a_{1}\epsilon\left(  a_{2}\right)  }_{=a}\psi\left(  v\right)  \\
& =a\psi\left(  v\right)  .
\end{align*}
Since this holds for all $a$ and $v$, we thus conclude that
$\psi$ lies in $\Hom_A(V,W)$.
Hence, $\Hom_\FF (V, W)^A \subseteq \Hom_A (V, W)$ is proven.
\end{proof}

\begin{noncompile}
\begin{proof}[Proof of Lemma~\ref{lem:AHom.gen} (Jia's writeup).]
If $\varphi\in\Hom_A(V,W)$ then $\varphi\in \Hom_\FF(V,W)^A$,
since each $a \in A$ and $v \in V$ satisfy
\[
(a\varphi)(v) = \sum a_1 \varphi(\alpha(a_2)v) 
= \sum a_1 \alpha(a_2) \varphi(v) 
= \epsilon(a) \varphi(v). 
\]

Conversely, let $\psi\in \Hom_\FF(V,W)^A$. Let $a \in A$ and
$v \in V$. Then
\[ \psi(av) =\sum \epsilon(a_1)\psi(a_2v) = \sum (a_1\psi)(a_2v)= \sum a_1\psi(\alpha(a_2)a_3v) = \sum a_1\epsilon(a_2) \psi(v) = a\psi(v)\]
where the first and last equalities use $a = \sum a_1\epsilon(a_2) = \sum \epsilon(a_1)a_2$ and the second uses $\psi\in \Hom_\FF(V,W)^A$.
\end{proof}
\end{noncompile}

\begin{remark}
If $A$ is a Hopf algebra, then $\alpha\left(A\right)$ is a Hopf
subalgebra of $A$.
As a consequence of our proof of Lemma~\ref{lem:AHom.gen}, we obtain the
following curious fact: If $V$ and $W$ are two $A$-modules, then
$\Hom_{\alpha\left(A\right)} (V, W) = \Hom_\FF (V, W)^A
= \Hom_A \left(V, W\right)$.
In fact, in our proof of Lemma~\ref{lem:AHom.gen}, we actually showed the
two inclusions
$\Hom_{\alpha\left(A\right)} \left(V, W\right)
\subseteq \Hom_\FF \left(V, W\right) ^A
\subseteq \Hom_A \left(V, W\right)$.
But this chain of inclusions clearly is an
equality, since its last term is contained in its first term.
\end{remark}
\end{verlong}

Next, let us generalize Lemma~\ref{lem:AHom2}:

\begin{lemma}\label{lem:AHom2.gen}
Let $V$ and $W$ be two $A$-modules such that
$W$ is finite-dimensional. Then,
$\Hom_A (V, W) \cong \Hom_A \left(W^* \otimes V, \epsilon\right)$.
\end{lemma}

\begin{vershort}
\begin{proof}
One has an $\FF$-vector space isomorphism
$\phi : \Hom(V, W) \rightarrow \Hom(W^* \otimes V,\FF)$
sending $f$ in $\Hom_\FF (V, W)$
to the functional $\phi(f)$ satisfying
$
\phi(f)(g \otimes v) = g\left(f(v)\right).
$
Indeed,
$\phi$ is the composition of the standard
isomorphisms 
$$\Hom_\FF (V, W) \to \Hom_\FF (V, (W^*)^*)
= \Hom_\FF (V, \Hom_\FF(W^*, \FF))
\to \Hom_\FF \left(W^* \otimes V, \FF\right)
= \left(W^* \otimes V\right)^*,$$ 
where the first arrow is induced
by the isomorphism $W \to (W^*)^*$ arising from the
finite-dimensionality of $W$.
This $\phi$ is not, in general, an $A$-module isomorphism,
but we claim that it restricts to an isomorphism
$\Hom_A (V, W) \to \Hom_A \left(W^* \otimes V, \epsilon\right)$,
which would prove this lemma.
That is, we will show $f \in \Hom_\FF (V, W)$ belongs to
$\Hom_A (V, W)$ if and only if its image $\phi(f)$ belongs to
$\Hom_A \left(W^* \otimes V, \epsilon\right)= \Hom\left(W^* \otimes V, \epsilon\right)^A$. 
First observe that, for each
$a \in A$, $g \in W^*$ and $v \in V$,
\begin{equation}
\phi(f)\left(a(g \otimes v)\right)
= \phi(f)\left(\sum a_1 g \otimes a_2 v\right)
= \sum (a_1 g)\left( f\left(a_2 v\right)\right) 
= \sum g(\alpha(a_1) f(a_2v)).
\label{pf.lem:AHom2.3}
\end{equation}

\vskip.1in
\noindent
{\sf The forward implication.}
Assuming $f \in \Hom_A (V, W)$, we check  
$\phi(f) \in \Hom_A \left(W^* \otimes V, \epsilon\right)$ as follows:
\[ 
\phi(f)\left(a(g \otimes v)\right)
= g\left( \sum \alpha(a_1) a_2  f(v) \right)
= \epsilon(a) g\left( f(v) \right) 
= \epsilon(a) \phi(f) \left(g \otimes v\right)
\]
where the first equality used \eqref{pf.lem:AHom2.3} and
$f \in \Hom_A (V, W)$ and the second equality
used the definition of $\alpha$.

\vskip.1in
\noindent
{\sf The backward implication.}
We show $f \in \Hom_A (V, W)$, assuming 
$\phi(f) \in \Hom_A \left(W^* \otimes V, \epsilon\right)$ so that
\begin{equation*}
\phi(f)\left(a(g \otimes v)\right)
= \epsilon(a) \phi(f) \left(g \otimes v\right)
= \epsilon(a) g\left( f\left(v\right) \right).
\end{equation*}
Comparing this with \eqref{pf.lem:AHom2.3}, we obtain 
$
g\left( \sum \alpha(a_1) f\left(a_2 v\right) \right)
= g \left( \epsilon(a) f\left(v\right) \right)
$
for \textbf{all} $g \in W^*$, and hence
\begin{equation}
\sum \alpha(a_1) f\left(a_2 v\right)
= \epsilon(a) f\left(v\right).
\label{pf.lem:AHom2.2}
\end{equation}
\begin{verlong}
(since an element $w \in W$ is uniquely determined by its images under
all $g \in W^*$).
\end{verlong}
We use this to calculate that, for any $b \in A$ and $v \in V$, one has
\begin{multline*}
f(bv)
= f\left(\sum \epsilon(b_1) b_2 v\right)
= \sum \epsilon(b_1) f\left(b_2 v\right)
%= \sum \epsilon(b_1) 1 f\left(b_2 v\right)
= \sum \sum (b_1)_1 \alpha((b_1)_2) f\left(b_2 v\right) \\
= \sum b_1 \sum \alpha((b_2)_1) f\left((b_2)_2 v\right) 
= \sum b_1 \cdot \epsilon(b_2) f\left(v\right) 
= b f\left(v\right) 
\end{multline*}
where the third equality used the defining property of $\alpha$ applied to
$b_1$, the fourth equality used the coassociativity
$\sum \sum (b_1)_1 \otimes (b_1)_2 \otimes b_2 
= \sum b_1 \otimes \sum (b_2)_1 \otimes (b_2)_2$,
and the fifth applied 
\eqref{pf.lem:AHom2.2} with  $a = b_2$.
\end{proof}
\end{vershort}

\begin{verlong}
\begin{proof}[Proof of Lemma~\ref{lem:AHom2.gen}.]
For every $f \in \Hom_\FF (V, W)$, define a linear functional
$\phi(f) \in \left(W^* \otimes V \right)^*$ by setting
\[
\phi(f)(g \otimes v) = g\left(f(v)\right)
\qquad \text{for all } g \in W^* \text{ and } v \in V .
\]
Thus, we have defined an $\FF$-linear map
$\phi : \Hom_\FF (V, W) \to \left(W^* \otimes V \right)^*$. This
$\FF$-linear map is an isomorphism of vector spaces\footnote{Indeed,
this map is the composition of the standard
isomorphisms $\Hom_\FF (V, W) \to \Hom_\FF (V, (W^*)^*)
= \Hom_\FF (V, \Hom_\FF(W^*, \FF))
\to \Hom_\FF \left(W^* \otimes V, \FF\right)
= \left(W^* \otimes V\right)^*$. Here, the isomorphism
$W \to (W^*)^*$ (arising from the finite-dimensionality of $W$)
was used for the first arrow.}. It is not, in general, an
$A$-module isomorphism. Nevertheless, we claim that it restricts to an
isomorphism
$\Hom_A (V, W) \to \Hom_A \left(W^* \otimes V, \epsilon\right)$.
This claim (once proven) will immediately yield Lemma~\ref{lem:AHom2.gen};
thus, it suffices to prove this claim. In other words, it suffices to
prove that a map $f \in \Hom_\FF (V, W)$ belongs to
$\Hom_A (V, W)$ if and only if its image $\phi(f)$ belongs to
$\Hom_A \left(W^* \otimes V, \epsilon\right)$. We shall
prove the two directions of this equivalence separately:

$\Longrightarrow$: Assume that $f \in \Hom_A (V, W)$. Recall that
$\sum \alpha(a_1) a_2 = \epsilon(a)1$ for each $a \in A$. Now, for each
$a \in A$, $g \in W^*$ and $v \in V$, we have
\begin{align}
\phi(f)\left(a(g \otimes v)\right)
&= \phi(f)\left(\sum a_1 g \otimes a_2 v\right)
= \sum (a_1g) \left(f\left(a_2 v\right)\right)
= \sum g \left(\alpha(a_1) f\left(a_2 v\right)\right) \\
&= g \left( \sum \alpha(a_1)
        \underbrace{f\left(a_2 v\right)}_{\substack{
            = a_2 f\left(v\right) \\ \text{(since } f \in \Hom_A (V, W) \text{)} }}
     \right)
\label{pf.lem:AHom2.1} \\
&= g\left( \underbrace{\sum \alpha(a_1) a_2}_{= \epsilon(a)1}
  f\left(v\right) \right)
= \epsilon(a) \underbrace{g\left( f\left(v\right) \right)}_{=\phi(f) \left(g \otimes v\right)}
= \epsilon(a) \phi(f) \left(g \otimes v\right) .
\nonumber
\end{align}
Thus, $\phi(f)$ belongs to
$\Hom_A \left(W^* \otimes V, \epsilon\right)$.
This proves the $\Longrightarrow$ direction.

$\Longleftarrow$: Assume that
$\phi(f)$ belongs to $\Hom_A \left(W^* \otimes V, \epsilon\right)$.

Let $a \in A$, $g \in W^*$ and $v \in V$. Then, as in
\eqref{pf.lem:AHom2.1}, we find
\begin{equation}
\phi(f)\left(a(g \otimes v)\right)
= g \left(\sum \alpha(a_1) f\left(a_2 v\right)\right) .
\label{pf.lem:AHom2.3}
\end{equation}
But $\phi(f)$ belongs to
$\Hom_A \left(W^* \otimes V, \epsilon\right)$. Hence,
\begin{equation*}
\phi(f)\left(a(g \otimes v)\right)
= \epsilon(a) \underbrace{\phi(f) \left(g \otimes v\right)}_{=g\left( f\left(v\right)\right)}
= \epsilon(a) g\left( f\left(v\right) \right)
= g\left( \epsilon(a) f\left(v\right) \right) .
\end{equation*}
Comparing this with \eqref{pf.lem:AHom2.3}, we obtain
\[
g\left( \sum \alpha(a_1) f\left(a_2 v\right) \right)
= g \left( \epsilon(a) f\left(v\right) \right) .
\]
Since this holds for \textbf{all} $g \in W^*$, we thus have
\begin{equation}
\sum \alpha(a_1) f\left(a_2 v\right)
= \epsilon(a) f\left(v\right)
\label{pf.lem:AHom2.2}
\end{equation}
(since an element $w \in W$ is uniquely determined by its images under
all $g \in W^*$).

Now, let $b \in A$ and $v \in V$. Then, using the axiom
$b = \sum \epsilon(b_1) b_2$, we find
\begin{align*}
f(bv)
&= f\left(\sum \epsilon(b_1) b_2 v\right)
= \sum \epsilon(b_1) f\left(b_2 v\right)
= \sum \epsilon(b_1) 1 f\left(b_2 v\right)
= \sum \sum (b_1)_1 \alpha((b_1)_2) f\left(b_2 v\right) \\
& \qquad \left(\text{by the antipode axiom }
  \epsilon(a) 1 = \sum a_1 \alpha(a_2)\text{, applied to }
  a = b_1 \right) \\
&= \sum b_1 \underbrace{\sum \alpha((b_2)_1) f\left((b_2)_2 v\right)}_{\substack{= \epsilon(b_2) f\left(v\right) \\ \text{(by \eqref{pf.lem:AHom2.2}, applied to } a = b_2 \text{)}}} \\
& \qquad \left(\text{by the coassociativity law }
  \sum \sum (b_1)_1 \otimes (b_1)_2 \otimes b_2 = \sum b_1 \otimes \sum (b_2)_1 \otimes (b_2)_2 \right) \\
&= \sum b_1 \epsilon(b_2) f\left(v\right) = b f\left(v\right) .
\end{align*}
Hence, $f$ is $A$-linear, i.e., belongs to $\Hom_A (V, W)$.
This proves the $\Longleftarrow$ direction. Hence,
Lemma~\ref{lem:AHom2.gen} is proven.
\end{proof}
\end{verlong}

%We say $Y$ is a \emph{left dual} of $X$ and write $Y=X^*$, 
%or equivalently, say $X$ is a \emph{right dual} of $Y$ and write $X=\leftast{Y}$, 
%if there exist morphisms $\ev_X: Y\otimes X\to\mathbf1$ and 
%$\coev_X:\mathbf1\to X\otimes Y$, called the \emph{evaluation} and \emph{coeval%uation}, such that the following two composition morphisms are 
%both identity morphisms:
%\[ X\xrightarrow{\coev_X\otimes\id_X} (X\otimes Y)\otimes X 
%  \xrightarrow{a_{X,Y,X}} X\otimes(Y\otimes X) \xrightarrow{\id_X\otimes\ev_X} %X,\]
%\[ Y \xrightarrow{\id_Y \otimes \coev_X} Y\otimes (X\otimes Y) 
%  \xrightarrow{a^{-1}_{Y,X,Y}} (Y\otimes X)\otimes Y \xrightarrow{\ev_X\otimes\%id_Y} Y.\]

%If $U$ and $V$ both have left [right] duals then $U\otimes V$ has a left dual $%V^*\otimes U^*$ [$\leftast{V} \otimes \leftast{U}$].

We next generalize Lemma~\ref{lem:tensor-dual}:

\begin{lemma}\label{lem:tensor-dual.gen}
Let $U$ and $V$ be $A$-modules, at least one of which is
finite-dimensional.
\begin{itemize}
\item[(i)] We have $(U\otimes V)^* \cong V^*\otimes U^*$.
\item[(ii)] Assume that $A$ is finite-dimensional. Then,
$\leftast{(U\otimes V)} \cong \leftast{V}\otimes \leftast{U}$.
\end{itemize}
\end{lemma}

\begin{vershort}
\begin{proof}
It is straightforward to check $A$-equivariance for
the usual $\FF$-vector space isomorphism
$V^*\otimes U^* \to (U\otimes V)^*$
(or $\leftast{V}\otimes \leftast{U} \to \leftast{(U\otimes V)}$)
that sends $g\otimes f$ to the functional
mapping $u\otimes v \mapsto f(u)g(v)$.
\end{proof}
\end{vershort}

\begin{verlong}
\begin{proof}[Proof of Lemma~\ref{lem:tensor-dual.gen}.]
(i)
There is a linear map $\phi: V^*\otimes U^*\to (U\otimes V)^*$
given by $\phi(g\otimes f)(u\otimes v) = f(u)g(v)$ for all $u\in U$,
$v\in V$, $f\in U^*$, and $g\in V^*$.
This linear map $\phi$ is an $\FF$-vector space isomorphism (since at
least one of $U$ and $V$ is finite-dimensional)\footnote{The quickest
way to prove this is to recall that both duals and tensor products
commute with finite direct sums, so we can reduce the proof to the
case when one of $U$ and $V$ is $\FF$; but this case is obvious.}.
We shall now show that this isomorphism is $A$-equivariant.
Indeed, for any $a\in A$, $g \in V^*$, $f \in U^*$, $u \in U$ and
$v \in V$, we have
\begin{align*}
\phi\left(\underbrace{a(g\otimes f)}_{= \sum a_1 g \otimes a_2 f} \right) (u\otimes v)
&= \sum \phi(a_1g \otimes a_2f)(u\otimes v)
= \sum \underbrace{(a_2f)(u)}_{= f(\alpha(a_2)u)} \underbrace{(a_1g)(v)}_{= g(\alpha(a_1) v)}
= \sum f(\alpha(a_2)u) g(\alpha(a_1)v) \\
&= \sum \phi( g \otimes f )(\alpha(a_2)u\otimes \alpha(a_1)v)
= \phi(g\otimes f)(\alpha(a)(u\otimes v))
= (a\phi(g\otimes f))(u\otimes v),
\end{align*}
where the fifth equality sign relied on the fact that
$\sum \alpha(a_2) \otimes \alpha(a_1) = \Delta (\alpha(a))$ (which is
part of what it means for $\alpha$ to be a coalgebra
anti-homomorphism).
Hence, $\phi\left(a (g \otimes f)\right) = a \phi(g \otimes f)$ for
all $a\in A$, $g \in V^*$ and $f \in U^*$
(because we have just shown that the two linear maps
$\phi\left(a (g \otimes f)\right)$ and
$a \phi(g \otimes f)$ are equal on all pure tensors).
Therefore, $\phi\left(a t\right) = a \phi(t)$ for all $a \in A$
and $t \in V^* \otimes U^*$ (by linearity).
In other words, the isomorphism
$\phi: V^*\otimes U^*\to (U\otimes V)^*$ is $A$-equivariant.
Therefore $(U\otimes V)^* \cong V^*\otimes U^*$.

(ii) The proof for
$\leftast{(U\otimes V)} \cong \leftast{V}\otimes \leftast{U}$
is similar.\footnote{
Indeed, the isomorphism is provided by the same map $\phi$. The
computation is also identical, except that $\alpha$ is replaced by
$\alpha^{-1}$ (which, too, is a coalgebra anti-homomorphism).
The finite-dimensionality of $A$ is used to ensure that $\alpha^{-1}$
exists (and the right-duals are well-defined).
% There is a linear isomorphism $\psi: \leftast{V}\otimes \leftast{U}\to \leftast{(U\otimes V)}$
% given by $\psi(g\otimes f)(u\otimes v) = f(u)g(v)$ for all $u\in U$,
% $v\in V$, $f\in\leftast{U}$, and $g\in\leftast{V}$. 
% This isomorphism is $A$-equivariant since for any $a\in A$ we have
% \[ \begin{aligned}
% \psi(a(g\otimes f))(u\otimes v) &= \sum \psi(a_1g \otimes a_2f)(u\otimes v) = \sum a_2f(u) \otimes a_1g(v) = \sum f(\alpha^{-1}(a_2)u) \otimes g(\alpha^{-1}(a_1)v) \\
% &= \sum \psi( g \otimes f )(\alpha^{-1}(a_2)u\otimes \alpha^{-1}(a_1)v) = \psi(g\otimes f)(\alpha^{-1}(a)(u\otimes v)) =  (a\psi(g\otimes f))(u\otimes v). 
% \end{aligned}\]
}
\end{proof}
\end{verlong}

\begin{verlong}
Here is another basic property of Hopf algebras:

\begin{lemma}\label{lem:coevaluation}
Let $V$ be a finite-dimensional $A$-module.
Let $\{v_i\}$ be a basis for the $\FF$-vector space $V$,
and $\{v_i^*\}$ the dual basis for $V^*$.
Identify $V \otimes \epsilon$ with $V$ and $\epsilon \otimes V^*$ with
$V^*$ as usual.
Then
\begin{enumerate}
\item[(i)]
$\sum_i v_i\otimes v_i^*(v)=v$ for all $v\in V$,
\item[(ii)]
$\sum_i v^*(v_i)\otimes v_i^* = v^*$ for all $v^*\in V^*$, and
\item[(iii)]
$a\sum_i v_i\otimes v_i^*=\epsilon(a) \sum_i v_i\otimes v_i^*$ for all $a\in A$.
\end{enumerate}
\end{lemma}
\begin{proof}[Proof of Lemma~\ref{lem:coevaluation}.]
(i) and (ii) are just restatements of the identities
$\sum_i v_i^*(v) v_i = v$ and $\sum_i v^*(v_i) v_i^* = v^*$, which
are known facts from linear algebra.

(iii) The identity map $\id_V$ belongs to
$\Hom_A \left(V,V\right) = \Hom_\FF \left(V,V\right)^A$ (by
Lemma~\ref{lem:AHom.gen}). But
Lemma~\ref{lem.hom-tensor-iso.gen} (ii) provides an $A$-module
isomorphism $\Phi: V \otimes V^* \overset{\cong}{\rightarrow} \Hom_\FF(V,V)$.
The image of the element $\sum_i v_i\otimes v_i^* \in V \otimes V^*$
under this isomorphism $\Phi$ is $\id_V$ (because each $v \in V$
satisfies
$\left(\Phi\left(\sum_i v_i\otimes v_i^*\right)\right)\left(v\right)
= \sum_i v_i^*\left(v\right) v_i = v = \id_V \left(v\right)$),
which belongs to
$\Hom_A \left(V,V\right) = \Hom_\FF \left(V,V\right)^A$. Hence,
$\sum_i v_i\otimes v_i^*$ must belong to
$\left(V \otimes V^*\right)^A$. In other words,
$a\sum_i v_i\otimes v_i^*=\epsilon(a) \sum_i v_i\otimes v_i^*$ for all $a\in A$.
This proves part (iii).
% (iii) Writing $v=\sum_j c_j v_j$ and $v^*=\sum_j d_j v_j^*$, with $c_j,d_j\in\FF$, and using $v_i^*(v_j)=\delta_{i,j}$, we have 
% \[ \sum_i v_i\otimes v_i^*(v) = \sum_i c_iv_i = v \quad\text{and}\quad 
% \sum_i v^*(v_i)\otimes v_i^*=\sum_i d_iv_i^*=v^*.\]
% Thus $\sum_i v_i\otimes v_i^* \in V\otimes V^* \cong\Hom(V,V)$ is the identity morphism, on which $a$ acts by $\epsilon(a)$ by Lemma~\ref{lem:AHom.gen}.
\end{proof}
\end{verlong}

The next lemma generalizes Lemma~\ref{lem:tensor-dual-four-isos}:

\begin{lemma}
\label{lem:tensor-dual-four-isos.gen}
Let $U$, $V$, and $W$ be $A$-modules such that $V$ and $W$ are
finite-dimensional.
Then, one has isomorphisms
\begin{align}
\Hom_A(U\otimes V, W) \xrightarrow\sim \Hom_A(U, W\otimes V^*),
\label{eq:TensorDual1.gen}\\
\Hom_A(V^*\otimes U, W) \xrightarrow\sim \Hom_A(U, V\otimes W).
\label{eq:TensorDual2.gen}
\end{align}

Assume furthermore that $A$ is finite-dimensional.
Then, one has isomorphisms
\begin{align}
\Hom_A(U\otimes \leftast{V}, W) \xrightarrow\sim \Hom_A(U, W\otimes V),
\label{eq:TensorDual3.gen} \\
\Hom_A(V\otimes U, W) \xrightarrow\sim \Hom_A(U, \leftast{V}\otimes W).
\label{eq:TensorDual4.gen}
\end{align}
\end{lemma}

\begin{vershort}
\begin{proof}
One only needs to check \eqref{eq:TensorDual1.gen} and \eqref{eq:TensorDual2.gen},
since then replacing $V$ by $\leftast{V}$ yields \eqref{eq:TensorDual3.gen} and \eqref{eq:TensorDual4.gen}.

To verify \eqref{eq:TensorDual1.gen}, first note that one has
$A$-module isomorphisms
\begin{equation}
\label{some-A-module-isos-for-TensorDual1}
\Hom(U\otimes V, W)
\cong W \otimes (U \otimes V)^*
\cong W\otimes V^*\otimes U^*
\cong \Hom(U, W\otimes V^*)
\end{equation}
in which the two outer isomorphisms come from the appropriate
generalization of Lemma~\ref{lem.hom-tensor-iso},
and the middle from Lemma~\ref{lem:tensor-dual.gen}.
Now apply $X \mapsto X^A$ to the left and right side of
\eqref{some-A-module-isos-for-TensorDual1},
and use Lemma~\ref{lem:AHom} (in its generalized form),
yielding \eqref{eq:TensorDual1.gen}.

To verify \eqref{eq:TensorDual2.gen}, note that
$$
\Hom_A(V^*\otimes U, W)
\cong \Hom_A(W^*\otimes V^*\otimes U, \epsilon)
\cong \Hom_A((V\otimes W)^* \otimes U, \epsilon)
\cong \Hom_A(U, V\otimes W)
$$
where the two outer isomorphisms come from Lemma~\ref{lem:AHom2.gen},
and the middle from Lemma~\ref{lem:tensor-dual.gen}.
\end{proof}
\end{vershort}

\begin{verlong}
\begin{proof}[Proof of Lemma~\ref{lem:tensor-dual-four-isos.gen}.]
Notice that $V^*$, $W \otimes V^*$ and $V \otimes W$ are
finite-dimensional (since $V$ and $W$ are finite-dimensional).

One only needs to check \eqref{eq:TensorDual1.gen} and \eqref{eq:TensorDual2.gen},
since then replacing $V$ by $\leftast{V}$ yields
\eqref{eq:TensorDual3.gen} and \eqref{eq:TensorDual4.gen}
(thanks to Lemma~\ref{lem.Vstarstar.gen}).

We have
\begin{align*}
\Hom(U\otimes V, W)
&\cong W \otimes (U \otimes V)^*
\qquad \left(\text{by Lemma~\ref{lem.hom-tensor-iso.gen} (ii)}\right) \\
&\cong W\otimes V^*\otimes U^*
\qquad \left(\text{by Lemma~\ref{lem:tensor-dual.gen} (i)}\right) \\
&\cong \Hom(U, W\otimes V^*)
\qquad \left(\text{by Lemma~\ref{lem.hom-tensor-iso.gen} (ii)}\right) .
\end{align*}
Now, \eqref{eq:TensorDual1.gen} holds since
%combining \eqref{hom-tensor-iso}, Lemma~\ref{lem:AHom.gen}, and Lemma~\ref{lem:tensor-dual.gen} gives 
\begin{align*}
\Hom_A(U\otimes V, W)
&= \Hom(U\otimes V, W)^A
\qquad \left(\text{by Lemma~\ref{lem:AHom.gen}}\right) \\
&\cong \Hom(U, W\otimes V^*)^A
\qquad \left(\text{since } \Hom(U\otimes V, W) \cong \Hom(U, W\otimes V^*) \right) \\
&= \Hom_A(U, W\otimes V^*)
\qquad \left(\text{by Lemma~\ref{lem:AHom.gen}}\right) .
\end{align*}

Furthermore, \eqref{eq:TensorDual2.gen} holds since
\begin{align*}
\Hom_A(V^*\otimes U, W)
&\cong \Hom_A(W^*\otimes V^*\otimes U, \epsilon)
\qquad \left(\text{by Lemma~\ref{lem:AHom2.gen}}\right) \\
&\cong \Hom_A((V\otimes W)^* \otimes U, \epsilon)
\qquad \left(\text{since Lemma~\ref{lem:tensor-dual.gen} (i) yields } (V\otimes W)^* \cong W^* \otimes V^*\right) \\
&\cong \Hom_A(U, V\otimes W)
\qquad \left(\text{by Lemma~\ref{lem:AHom2.gen}}\right) .
\end{align*}

{\small \textbf{Remark:}
Here is an alternative proof of \eqref{eq:TensorDual2.gen}:
Fix a basis $\{v_i\}$
for the $\FF$-vector space $V$, and the corresponding dual
basis $\{v_i^*\}$ of $V^*$.

For each $f\in \Hom(V^*\otimes U, W)$, we define $\phi(f)\in\Hom(U, V\otimes W)$ by 
\[ \phi(f)(u)= \sum_i v_i\otimes f(v_i^*\otimes u), \quad \forall u\in U.\]
This defines a linear map $\phi : \Hom(V^*\otimes U, W) \to \Hom(U, V\otimes W)$.
(Its definition uses a choice of bases, but $\phi$ itself is independent
of this choice, since the tensor $\sum_i v_i \otimes v_i^*$
does not depend on the choice of basis $\{v_i\}$.)

Conversely, if $g\in\Hom(U, V\otimes W)$ then define $\psi(g)\in\Hom(V^*\otimes U, W)$ by 
\[ \psi(g)(v^*\otimes u) = (v^*\otimes\id)(g(u)), \quad \forall u\in U,\ \forall v^*\in V^*.\]
This defines a linear map $\psi : \Hom(U, V\otimes W) \to \Hom(V^*\otimes U, W)$.

The maps $\phi$ and $\psi$ are inverses of each other since Lemma~\ref{lem:coevaluation} implies
\[ \phi(\psi(g))(u) = \sum_i v_i\otimes \psi(g)(v_i^*\otimes u) = \sum_i v_i\otimes (v_i^*\otimes\id)(g(u)) = g(u), \]
\[ \psi(\phi(f))(v^*\otimes u) = (v^*\otimes\id)(\phi(f)(u)) = \sum_i v^*(v_i) f(v_i^*\otimes u) = f(v^*\otimes u). \]
It thus remains to show that
$\phi \left( \Hom_A(V^*\otimes U, W) \right)
\subseteq \Hom_A(U, V\otimes W)$
and
$\psi \left( \Hom_A(U, V\otimes W) \right)
\subseteq \Hom_A(V^*\otimes U, W)$.

Suppose $f \in \Hom(V^*\otimes U, W) $ is $A$-linear.
Let $a\in A$ and $u\in U$. Then
\[ a\phi(f)(u) = \sum_i \sum a_1v_i\otimes a_2f(v^*_i\otimes u) = \sum_i \sum a_1v_i\otimes f(a_2'v^*_i\otimes a_2''u).\]
Combining this with Lemma~\ref{lem:coevaluation} we have, for any $v^*\in V^*$, 
\[\begin{aligned}
(v^*\otimes\id)(a\phi(f)(u)) & = \sum_i \sum f(v^*(a_1v_i) a_2'v^*_i\otimes a_2''u) = \sum f(\epsilon(a_1)v^*\otimes a_2u) \\
&= f(v^*\otimes au) = \sum_i v^*(v_i) f(v^*_i\otimes au) = (v^*\otimes\id)(\phi(f)(au)).
\end{aligned}\]
This implies $a\phi(f)(u)=\phi(f)(au)$. Thus $\phi(f)$ is $A$-linear.
Hence,
$\phi \left( \Hom_A(V^*\otimes U, W) \right)
\subseteq \Hom_A(U, V\otimes W)$.

Now suppose $g \in \Hom(U, V \otimes W)$ is $A$-linear.
Let $a\in A$, $u\in U$, and $v^*\in V^*$.
Writing $g(u) = \sum_i v^{(i)}\otimes w^{(i)}$, we have
\[\begin{aligned}
\psi(g)(a(v^*\otimes u)) &= \sum \psi(g)(a_1v^*\otimes a_2u) = \sum (a_1v^*\otimes\id)(g(a_2u)) \\
& = \sum (a_1v^*\otimes\id)(a_2g(u)) = \sum (a_1v^*\otimes\id) \sum_i a_2'v^{(i)}\otimes a_2''w^{(i)} \\
& = \sum \sum_i v^*(\alpha(a_1)a_2'v^{(i)}) a_2''w^{(i)} = \sum \sum_i v^*(\epsilon(a_1) v^{(i)}) a_2w^{(i)} \\
& = \sum \sum_i v^*(v^{(i)}) \epsilon(a_1)a_2w^{(i)} = \sum_i v^*(v^{(i)}) aw^{(i)} \\
& = a \sum_i v^*(v^{(i)}) w^{(i)} =  a((v^*\otimes\id)(g(u))) = a\psi(g)(v^*\otimes u). 
\end{aligned}\]
This shows that $\psi(g)$ is $A$-linear.
Thus,
$\psi \left( \Hom_A(U, V\otimes W) \right)
\subseteq \Hom_A(V^*\otimes U, W)$.
Therefore \eqref{eq:TensorDual2.gen} holds.
% \footnote{
% The proof for \eqref{eq:TensorDual2.gen} should also work for \eqref{eq:TensorDual1.gen}.
% As an alternative approach to \eqref{eq:TensorDual2.gen}, I (Jia) tried to prove $\psi$ and $\phi$ are $A$-equivariant, but am still not sure whether this is true.
% [DG] Nope, this is not true. For $U = \epsilon$ and $W = \epsilon$,
% it would be claiming that $(V^*)^* \cong V$ via the canonical
% double-dual isomorphism. I'm not sure if there is *any*
% isomorphism $(V^*)^* \cong V$ ...
% }

The advantage of this alternative proof of \eqref{eq:TensorDual2.gen}
is that it gets by without using the finite-dimensionality of $W$.
Hence, the isomorphism \eqref{eq:TensorDual2.gen} holds even if $W$
is not finite-dimensional (as long as $V$ is finite-dimensional).

The isomorphism \eqref{eq:TensorDual1.gen} also holds in this
generality (i.e., requiring only $V$ to be finite-dimensional, not
$W$).
Here is an outline of how to prove this:
It clearly suffices to prove the isomorphism
$\Hom(U\otimes V, W) \cong \Hom(U, W\otimes V^*)$ (because this
isomorphism was the crucial step in our above proof).
In view of the isomorphism
$W\otimes V^* \cong \Hom(V, W)$
(a consequence of Lemma~\ref{lem.hom-tensor-iso.gen} (ii)),
this boils down to proving the isomorphism
$\Hom(U\otimes V, W) \cong \Hom(U, \Hom(V, W))$.
But this isomorphism is well-known:
It is provided by the $\FF$-vector space isomorphism
\begin{align*}
\Hom(U\otimes V, W) &\to \Hom(U, \Hom(V, W)) , \\
F &\mapsto \left(\text{the map } U \to \Hom(V, W) \text{ sending each }
                u \in U \text{ to the map } V \to W \right. \\
  & \qquad \qquad \left.
                \text{ sending each } v \in V \text{ to }
                F \left(u \otimes v\right) \right) ,
\end{align*}
whose $A$-equivariance can be proven by a somewhat tedious but
straightforward computation.\footnote{The computation uses the fact
that $\alpha$ is a coalgebra anti-homomorphism.}
}
\end{proof}
\end{verlong}

\begin{vershort}
Finally, we shall prove Proposition~\ref{prop.AA} as part of the
following fact:
\end{vershort}

\begin{verlong}
Let us now extend Proposition~\ref{prop.AA}.
To wit, Proposition~\ref{prop.AA} will follow from part (ii) of the
following fact:
\end{verlong}

\begin{proposition} \label{prop.AA.gen}
Let $A$ be a finite-dimensional Hopf algebra.
\begin{enumerate}
\item[(i)] We have $\dim \left(A^A\right) = 1$.
\item[(ii)] Let $V$ be a finite-dimensional $A$-module. Then,
$\dim \Hom_A\left(V, A\right) = \dim V$.
\end{enumerate}
\end{proposition}

Proposition~\ref{prop.AA.gen} (i) is actually the well-known fact (see,
e.g., \cite[Thm. 10.2.2 (a)]{Radford}) that the vector space of
left integrals of the finite-dimensional Hopf algebra $A$ is
$1$-dimensional.
Nevertheless, we shall give a proof, as it is easy using what has
been done before.

\begin{vershort}
\begin{proof}
Let $V$ be an $A$-module.
Lemma~\ref{lem.hom-tensor-iso} applied with $W = A$ shows that
$A \otimes V^* \cong \Hom_\FF(V,A)$ as $A$-modules. But
Lemma~\ref{lem.VtimesA} (ii) (applied to $V^*$ instead of $V$) yields
$A \otimes V^* \cong A^{\oplus \dim (V^*)} = A^{\oplus \dim V}$ as
$A$-modules. Hence,
$\Hom_\FF(V,A) \cong A \otimes V^* \cong A^{\oplus \dim V}$ as
$A$-modules. Now, Lemma~\ref{lem:AHom} applied to $W=A$ yields
$\FF$-vector space isomorphisms
$$
\Hom_A (V, A)
= \Hom_\FF (V, A)^A \cong \left(A^{\oplus \dim V}\right)^A
\cong \left(A^A\right)^{\oplus \dim V}
$$
where the first isomorphism comes from
$\Hom_\FF(V,A) \cong A^{\oplus \dim V}$, and
the second holds because $W \mapsto W^A$ preserves direct sums.
Taking dimensions, we thus find
\begin{equation}
\dim \Hom_A (V, A)
%= \dim \left( \left(A^A\right)^{\oplus \dim V} \right)
= \dim \left(A^A\right) \dim V .
\label{pf.prop.AA.vic.1}
\end{equation}
It only remains then to check that
$\dim \left(A^A\right) = 1$,
since then substituting it into
\eqref{pf.prop.AA.vic.1} would give Proposition~\ref{prop.AA.gen} (ii)
(and Proposition~\ref{prop.AA.gen} (i) would be proven as well).
To this end, use the $\FF$-vector space isomorphism
$A \cong \Hom_A (A, A)$ to obtain
$$
\dim A = \dim \Hom_A (A, A) = \dim \left(A^A\right) \dim A ,
$$
where the second equality follows
from \eqref{pf.prop.AA.vic.1} applied to $V = A$. 
Canceling $\dim A$ from the ends gives
the desired equality $1 = \dim \left(A^A\right)$, completing the proof.
\end{proof}
\end{vershort}

\begin{verlong}
\begin{proof}[Proof of Proposition~\ref{prop.AA.gen}.]
Let $V$ be a finite-dimensional $A$-module.
Lemma~\ref{lem.hom-tensor-iso.gen} (ii) (applied to $W = A$) shows that
$A \otimes V^* \cong \Hom_\FF(V,A)$ as $A$-modules. But
Lemma~\ref{lem.VtimesA.gen} (ii) (applied to $V^*$ instead of $V$) yields
$A \otimes V^* \cong A^{\oplus \dim (V^*)} = A^{\oplus \dim V}$ as
$A$-modules. Hence,
$\Hom_\FF(V,A) \cong A \otimes V^* \cong A^{\oplus \dim V}$ as
$A$-modules. Now, Lemma~\ref{lem:AHom.gen} (applied to $W=A$) yields
\begin{align*}
\Hom_A (V, A)
&= \Hom_\FF (V, A)^A \cong \left(A^{\oplus \dim V}\right)^A
\qquad \left(\text{since } \Hom_\FF(V,A) \cong A^{\oplus \dim V}
       \right) \\
&\cong \left(A^A\right)^{\oplus \dim V}
\qquad \left(\text{since the functor } W \mapsto W^A
                \text{ preserves direct sums} \right)
\end{align*}
as $\FF$-vector spaces. Taking dimensions, we thus find
\begin{equation}
\dim \Hom_A (V, A)
= \dim \left( \left(A^A\right)^{\oplus \dim V} \right)
= \dim \left(A^A\right) \dim V .
\label{pf.prop.AA.1}
\end{equation}

Now, forget that we fixed $V$. It is well-known that
$\Hom_A (A, A) \cong A$ as $\FF$-vector spaces (indeed, the map
$\Hom_A (A, A) \to A, \  f \mapsto f\left(1\right)$ is an
isomorphism). Hence,
$\dim A = \dim \Hom_A (A, A) = \dim \left(A^A\right) \dim A$
(by \eqref{pf.prop.AA.1}, applied to $V = A$). We can cancel $\dim A$
from this equality (since $\dim A > 0$), and thus obtain
$1 = \dim \left(A^A\right)$. This proves
Proposition~\ref{prop.AA.gen} (i).

(ii) From \eqref{pf.prop.AA.1}, we obtain
$\dim \Hom_A (V, A)
= \underbrace{\dim \left(A^A\right)}_{=1} \dim V = \dim V$.
\end{proof}
\end{verlong}

\begin{verlong}
We record a curious corollary of Proposition~\ref{prop.AA.gen}, which we
will not use below but we find worth observing.

\begin{corollary} \label{cor.AA.Ae}
Let $e$ be an idempotent in the finite-dimensional Hopf algebra $A$.
Then, $\dim\left(Ae\right) = \dim\left(eA\right)$.
\end{corollary}

\begin{proof}
A well-known fact (see, e.g., \cite[Prop. 7.4.1 (3)]{Webb}) says that
$\dim \Hom_A\left(Ae, U\right) = \dim \left(eU\right)$
for any $A$-module $U$. Applying this to $U = A$, we obtain
$\dim \Hom_A\left(Ae, A\right) = \dim \left(eA\right)$.
But Proposition~\ref{prop.AA.gen} (ii) (applied to $V = Ae$) yields
$\dim \Hom_A\left(Ae, A\right) = \dim \left(Ae\right)$.
Hence,
$\dim \left(Ae\right)
= \dim \Hom_A\left(Ae, A\right) = \dim \left(eA\right)$.
\end{proof}
\end{verlong}

\begin{verlong}
%%%%%%%%%%%%%%%%%%%%%%%%%%%%%%%%%%%%%%%%%%%%%%%%%%%%%%%%%
\section{Appendix B: an elementary proof of Lemma~\ref{lem.coker}}
\label{sect.smith}
In this section, we shall give a second proof of
Lemma~\ref{lem.coker}, using nothing but basic linear algebra.
Besides its elementary nature, this proof has the additional
advantage of providing an explicit construction of the
isomorphism claimed in Lemma~\ref{lem.coker} under some circumstances
(e.g., if two entries of $\sss$ equal $1$).

We prepare by showing some lemmas about $\ZZ$-modules.

\begin{lemma}
\label{lem.coker.unimodular-row-GLm}
Let $m$ be a positive integer. Let $\left(e_1,e_2,\ldots,e_m\right)$
be the standard basis of the $\ZZ$-module $\ZZ^m$ (consisting of
column vectors). Let $w \in \ZZ^m$ be a column
vector. Let $g = \gcd(w)$. (Here and below, $\gcd(w)$ denotes
the greatest common divisor of the entries of $w$.) Then,
there exists some $B \in \GL_m\left(\ZZ\right)$ such that
$Bw = g e_1$.
\end{lemma}

\begin{proof}
This is a well-known fact, but let us sketch a proof. An
operation on vectors in $\ZZ^m$ (that is, a map $\ZZ^m \to \ZZ^m$)
is called an \textit{elementary operation} if and only if
\begin{itemize}
\item it is a \textit{negation operation}, which means that it
  multiplies an entry of the vector by $-1$; \textbf{or}
\item it is an \textit{addition operation}, which means that it
  adds an integer multiple of an entry of the vector to another entry;
  \textbf{or}
\item it is a \textit{swap operation}, which means that it swaps two
  entries of the vector.
\end{itemize}
It is well-known that each composition of elementary operations can
be rewritten as left multiplication by some matrix in
$\GL_m\left(\ZZ\right)$ (i.e., as the map $v \mapsto Bv$ for some
$B \in \GL_m\left(\ZZ\right)$). Hence, if we can prove that the vector
$g e_1$ can be obtained from $w$ by a sequence of
elementary operations, then we will be done.

But proving this is easy: Start with the vector $w$. Then, apply
negation operations to turn all its entries nonnegative. Then, apply
addition operations (specifically, subtracting entries from other
entries) to ensure that at most one of its entries is
nonzero\footnote{This can be done as follows: As long as our vector
has (at least) two nonzero entries, we can apply an addition operation
(namely, subtracting the smaller of these two entries from the
larger) to obtain a new vector, whose entries are still all
nonnegative, but whose sum of entries is smaller than that of the
previous vector. We can repeat this step until no two nonzero entries
remain (which is destined to happen, since the sum of entries cannot
keep decreasing forever).}.
Finally, apply a swap operation to ensure that this nonzero
entry (if it exists)
is the first entry (if it was not already). The resulting vector
has the form $p e_1$ for some $p \in \ZZ$. Consider this $p$.
But elementary
operations do not change the greatest common divisor of the entries of
a vector. Hence, $\gcd\left(p e_1\right) = \gcd(w)$ (since we obtained
$p e_1$ from $w$ by elementary operations). Since $p$ is nonnegative,
we have $\gcd\left(p e_1\right) = p$, so that
$p = \gcd\left(p e_1\right) = \gcd(w) = g$ and thus $p e_1 = g e_1$.
Hence, we have obtained the vector $g e_1$ from $w$ by a sequence of
elementary operations (since we have obtained the vector $p e_1$ in
this way). As we have said above, this proves the lemma.
\end{proof}

\begin{lemma}
\label{lem.coker.0}
Let $m$ be a positive integer. Let $d \in \ZZ$, and let $w \in \ZZ^m$
be a column vector. Let $\gamma = \gcd\left(d, \gcd(w)\right)$. Then,
\[
\ZZ^m / \left( d \ZZ^m + \ZZ w \right)
\cong \left( \ZZ / \gamma \ZZ \right)
       \oplus \left( \ZZ / d \ZZ \right)^{m-1} .
\]
\end{lemma}

\begin{proof}
Let $g = \gcd(w)$. Thus, $\gcd(d, g) = \gcd(d, \gcd(w)) = \gamma$.
Now, $\ZZ d + \ZZ g
= \ZZ \underbrace{\gcd(d, g)}_{= \gamma} = \ZZ \gamma$.

Let $\left(e_1,e_2,\ldots,e_m\right)$
be the standard basis of the $\ZZ$-module $\ZZ^m$ (consisting of
column vectors). Then,
$\ZZ d e_1 + \ZZ g e_1
= \underbrace{\left(\ZZ d + \ZZ g\right)}_{= \ZZ \gamma} e_1
= \ZZ \gamma e_1$.

Lemma~\ref{lem.coker.unimodular-row-GLm} shows that
there exists some $B \in \GL_m\left(\ZZ\right)$ such that
$Bw = g e_1$. Consider this $B$. Left multiplication by $B$ is an
automorphism of the $\ZZ$-module $\ZZ^m$ (since
$B \in \GL_m\left(\ZZ\right)$) and sends the submodule $d \ZZ^m$
to $d \ZZ^m$
while sending the submodule $\ZZ w$ to $\ZZ \underbrace{B w}_{= g e_1}
= \ZZ g e_1$. Hence, it induces an isomorphism
$\ZZ^m / \left( d \ZZ^m + \ZZ w \right)
\to \ZZ^m / \left( d \ZZ^m + \ZZ g e_1 \right)$. Thus,
\begin{align*}
\ZZ^m / \left( d \ZZ^m + \ZZ w \right)
&\cong \ZZ^m / \left( d \ZZ^m + \ZZ g e_1 \right)
= \ZZ^m / \left( \left(\ZZ d e_1 + \ZZ d e_2 + \cdots + \ZZ d e_m\right)
                    + \ZZ g e_1 \right) \\
&\qquad \left(\text{since } d \ZZ^m = \ZZ d e_1 + \ZZ d e_2 + \cdots + \ZZ d e_m \right) \\
&= \ZZ^m / \left( \underbrace{\left(\ZZ d e_1 + \ZZ g e_1\right)}_{= \ZZ \gamma e_1}
                    + \left(\ZZ d e_2 + \ZZ d e_3 + \cdots + \ZZ d e_m\right) \right) \\
&= \ZZ^m / \left( \ZZ \gamma e_1
                    + \left(\ZZ d e_2 + \ZZ d e_3 + \cdots + \ZZ d e_m\right) \right) \\
&= \left( \ZZ e_1 \oplus \ZZ e_2 \oplus \cdots \oplus \ZZ e_m \right)
    / \left( \ZZ \gamma e_1 \oplus \ZZ d e_2 \oplus \ZZ d e_3 \oplus
                \cdots \oplus \ZZ d e_m \right) \\
&\cong \underbrace{\left( \ZZ e_1 / \ZZ \gamma e_1 \right)}_{\cong \ZZ / \gamma \ZZ}
    \oplus \underbrace{\left( \ZZ e_2 / \ZZ d e_2 \right)
        \oplus \left( \ZZ e_3 / \ZZ d e_3 \right) \oplus \cdots
        \oplus \left( \ZZ e_m / \ZZ d e_m \right)}_{
                \cong \left( \ZZ / d \ZZ \right)^{m-1} } \\
&\cong \left( \ZZ / \gamma \ZZ \right)
       \oplus \left( \ZZ / d \ZZ \right)^{m-1} .
\qedhere
\end{align*}
\end{proof}

\begin{lemma}
\label{lem.coker.1}
Let $m > 1$ be an integer. Let $u \in \ZZ^m$ and $v \in \ZZ^m$
be two column vectors such that $v_1 = 1$.
Let $\gamma = \gcd(u)$, and set
$d = v^T u \in \ZZ$. Let
$L = d I_m - u v^T \in \ZZ^{m \times m}$. Then,
\[
\ZZ^m / \im L
\cong \ZZ \oplus \left(\ZZ / \gamma \ZZ\right)
      \oplus \left(\ZZ / d \ZZ\right)^{m-2} .
\]
\end{lemma}

\begin{proof}
Let $\left(e_1,e_2,\ldots,e_m\right)$
be the standard basis of the $\ZZ$-module $\ZZ^m$.
Recall that $L = d I_m - u v^T$. Hence, for each
$i \in \left\{1, 2, \ldots, m\right\}$, we have
\begin{align*}
L e_i &= \left(d I_m - u v^T\right) e_i
= d e_i - u \underbrace{v^T e_i}_{= v_i}
= d e_i - u v_i .
\end{align*}
Applying this to $i = 1$, we obtain
$L e_1 = d e_1 - u \underbrace{v_1}_{=1} = d e_1 - u$.

Now, for each $i \in \left\{1, 2, \ldots, m\right\}$, define a vector
$q_i \in \ZZ^m$ by $q_i = e_i - v_i e_1$. Then, each
$i \in \left\{1, 2, \ldots, m\right\}$ satisfies
\begin{align}
L q_i
&= L\left(e_i - v_i e_1\right)
= \underbrace{L e_i}_{= d e_i - u v_i}
    - v_i \underbrace{L e_1}_{= d e_1 - u}
= \left(d e_i - u v_i\right) - v_i \left(d e_1 - u\right)
\nonumber \\
& = d \underbrace{\left(e_i - v_i e_1\right)}_{= q_i} = d q_i .
\label{pf.lem.coker.1.3}
\end{align}
% and therefore $d q_i = L q_i \in \im L$.
Note that the definition of $q_1$ yields
$q_1 = e_1 - \underbrace{v_1}_{= 1} e_1 = e_1 - e_1 = 0$.
% Hence,
% $\spann\left(q_1, q_2, \ldots, q_m\right)
% = \spann\left(q_2, q_3, \ldots, q_m\right)$.
Hence,
$\sum_{i=1}^m u_i q_i
= \sum_{i=2}^m u_i q_i + u_1 \underbrace{q_1}_{= 0}
= \sum_{i=2}^m u_i q_i$. Thus,
\begin{align}
\sum_{i=2}^m u_i q_i
&= \sum_{i=1}^m u_i \underbrace{q_i}_{= e_i - v_i e_1}
= \sum_{i=1}^m u_i \left(e_i - v_i e_1\right)
= \underbrace{\sum_{i=1}^m u_i e_i}_{= u}
     - \underbrace{\sum_{i=1}^m u_i v_i}_{= v^T u = d} e_1
\nonumber \\
&= u - d e_1 = - \underbrace{\left(d e_1 - u\right)}_{= L e_1}
= - L e_1 .
\label{pf.lem.coker.1.4}
\end{align}

Let $u'$ be the vector $\left[u_2, u_3, \ldots, u_m\right]^T
\in \ZZ^{m-1}$. Recall that
\[
d = v^T u
= \underbrace{v_1}_{= 1} u_1 + v_2 u_2 + v_3 u_3 + \cdots + v_m u_m
= u_1 + \underbrace{v_2 u_2 + v_3 u_3 + \cdots + v_m u_m}_{
                     \substack{\in \gcd(u') \ZZ \\
                       \text{(since } u_i \in \gcd(u') \ZZ
                       \text{ for all } i > 1 \text{)}}}
\in u_1 + \gcd(u') \ZZ .
\]
Hence, $d \equiv u_1 \mod \gcd(u') \ZZ$.
Thus, $\gcd(d, \gcd(u')) = \gcd(u_1, \gcd(u')) = \gcd(u) = \gamma$.

We shall identify the $\ZZ$-module $\ZZ^{m-1}$ with the
$\ZZ$-submodule $\spann\left(e_2, e_3, \ldots, e_m\right)$ of
$\ZZ^m$ by equating each vector $\left[p_2, p_3, \ldots, p_m\right]^T
\in \ZZ^{m-1}$ with $p_2 e_2 + p_3 e_3 + \cdots + p_m e_m
\in \spann\left(e_2, e_3, \ldots, e_m\right)$.
From this point of view, we have
$u' = u_2 e_2 + u_3 e_3 + \cdots + u_m e_m$ and
\begin{align}
\ZZ^m / \left(d \ZZ^{m-1} + \ZZ u'\right)
&\cong \ZZ \oplus
   \underbrace{\left(\ZZ^{m-1} / \left(d \ZZ^{m-1} + \ZZ u'\right)
                    \right)}_{\substack{
                  \cong \left( \ZZ / \gamma \ZZ \right)
                          \oplus \left( \ZZ / d \ZZ \right)^{m-2} \\
                  \text{(by Lemma~\ref{lem.coker.0}, applied to }
                    m-1 \text{ and } u' \text{ instead of } m
                    \text{ and } w \text{)}}} \nonumber \\
&\cong \ZZ \oplus \left( \ZZ / \gamma \ZZ \right)
            \oplus \left( \ZZ / d \ZZ \right)^{m-2} .
\label{pf.lem.coker.1.5}
\end{align}

Now, we define a $\ZZ$-linear map $\Phi : \ZZ^m \to \ZZ^m$ by setting
\[
\Phi\left(e_i\right)
= \begin{cases}
    e_1, & \text{if } i = 1; \\
    q_i, & \text{if } i > 1
  \end{cases}
\qquad \text{for all } i \in \left\{1,2,\ldots,m\right\} .
\]
This map $\Phi$ is tantamount to left multiplication by the matrix
$\left(\begin{array}{ccccc}
            1 &   -v_2 &   -v_3 & \cdots &   -v_m \\
            0 &      1 &      0 & \cdots &      0 \\
            0 &      0 &      1 & \cdots &      0 \\
       \vdots & \vdots & \vdots & \ddots & \vdots \\
            0 &      0 &      0 & \cdots &      1
       \end{array} \right)$
(since $q_i = e_i - v_i e_1$),
which is invertible over $\ZZ$ (since it is upper-unitriangular).
Hence, $\Phi$ is a $\ZZ$-module isomorphism.

Next, we are going to show that
$\Phi\left(d \ZZ^{m-1} + \ZZ u'\right) = \im L$.

Indeed, for each $i \in \left\{2, 3, \ldots, m\right\}$, we have
\begin{equation}
\Phi\left(d e_i\right)
= d \underbrace{\Phi\left(e_i\right)}_{\substack{= q_i \\
                    \text{(since } i > 1 \text{)}}}
= d q_i = L q_i
\qquad \left(\text{by \eqref{pf.lem.coker.1.3}}\right) .
\label{pf.lem.coker.1.5a}
\end{equation}
Furthermore, $u' = u_2 e_2 + u_3 e_3 + \cdots + u_m e_m
= \sum_{i=2}^m u_i e_i$, so that
\begin{equation}
\Phi\left(u'\right)
= \Phi\left(\sum_{i=2}^m u_i e_i\right)
= \sum_{i=2}^m u_i
        \underbrace{\Phi\left(e_i\right)}_{\substack{= q_i \\
                    \text{(since } i > 1 \text{)}}}
= \sum_{i=2}^m u_i q_i = - L e_1
\qquad \left(\text{by \eqref{pf.lem.coker.1.4}}\right) .
\label{pf.lem.coker.1.5b}
\end{equation}

But the map $\Phi$ is an isomorphism. Thus,
\begin{align*}
\ZZ^m
&= \Phi\left(\ZZ^m\right)
= \Phi \left( \spann \left(e_1, e_2, e_3, \ldots, e_m\right) \right) \\
&= \spann\left(\Phi\left(e_1\right), \Phi\left(e_2\right),
              \Phi\left(e_3\right), \ldots, \Phi\left(e_m\right)\right)
= \spann\left(e_1, q_2, q_3, \ldots, q_m\right)
\end{align*}
(since the definition of $\Phi$ yields
$\Phi\left(e_1\right) = e_1$ and $\Phi\left(e_i\right) = q_i$ for
each $i > 1$). Hence,
\begin{equation}
L \left( \ZZ^m \right)
= L \left( \spann\left(e_1, q_2, q_3, \ldots, q_m\right) \right)
= \spann\left(L e_1, L q_2, L q_3, \ldots, L q_m\right) .
\label{pf.lem.coker.1.7}
\end{equation}

On the other hand,
\begin{align*}
d \ZZ^{m-1} + \ZZ u'
&= \underbrace{\ZZ u'}_{= \spann\left(u'\right)}
    + \underbrace{d \ZZ^{m-1}}_{\substack{
          = d \spann\left(e_2, e_3, \ldots, e_m\right) \\
          = \spann\left(de_2, de_3, \ldots, de_m\right)}}
= \spann\left(u'\right) + \spann\left(de_2, de_3, \ldots, de_m\right)
\\
&= \spann\left(u', de_2, de_3, \ldots, de_m\right)
\end{align*}
and thus
\begin{align*}
\Phi\left(d \ZZ^{m-1} + \ZZ u'\right)
&= \Phi\left( \spann\left(u', de_2, de_3, \ldots, de_m\right) \right)
\\
&= \spann\left(\Phi\left(u'\right), \Phi\left(de_2\right),
                 \Phi\left(de_3\right), \ldots,
                 \Phi\left(de_m\right) \right) \\
&= \spann\left(-L e_1, L q_2, L q_3, \ldots, L q_m\right)
\qquad \left(\text{by \eqref{pf.lem.coker.1.5b} and \eqref{pf.lem.coker.1.5a}}\right)
\\
&= \spann\left(L e_1, L q_2, L q_3, \ldots, L q_m\right)
= L \left( \ZZ^m \right)
\qquad \left(\text{by \eqref{pf.lem.coker.1.7}}\right) \\
&= \im L .
\end{align*}
Thus, the isomorphism
$\Phi : \ZZ^m \to \ZZ^m$ induces an isomorphism
$\ZZ^m / \left(d \ZZ^{m-1} + \ZZ u'\right)
\to \ZZ^m / \im L$. Hence,
\[
\ZZ^m / \im L
\cong \ZZ^m / \left(d \ZZ^{m-1} + \ZZ u'\right)
\cong \ZZ \oplus \left( \ZZ / \gamma \ZZ \right)
            \oplus \left( \ZZ / d \ZZ \right)^{m-2}
\qquad \left(\text{by \eqref{pf.lem.coker.1.5}}\right) .
\qedhere
\]
\end{proof}

\begin{proof}[Proof of Lemma~\ref{lem.coker}.]
Lemma~\ref{lem.coker} would follow by applying
Lemma~\ref{lem.coker.1} to $m = \ell+1$, $u = \ppp$, $v = \sss$,
except for a minor inconvenience:
Lemma~\ref{lem.coker} assumes $\sss_{\ell+1} = 1$, whereas
Lemma~\ref{lem.coker.1} assumes $v_1 = 1$.
However, this is merely a notational distinction, and
can be straightened out by relabeling the indices
(we leave the details to the reader).
\end{proof}
\end{verlong}

\begin{verlong}
%%%%%%%%%%%%%%%%%%%%%%%%%%%%%%%
\section{Appendix C:  Some more observations on gcds}
\label{gcd-alg-appendix}
%%%%%%%%%%%%%%%%%%%%%%%%%%%%%%%

In this short appendix, we relate the gcds of the entries of the
vectors $\sss$ and
$\ppp$ to expansions of $[A]$ in the Grothendieck groups $K_0(A)$ and
$G_0(A)$. This relation (which does not require $A$ to be a Hopf
algebra) was found as a side result in our study of the critical
group, but did not turn out to be useful for the latter.
We record it here merely to avoid losing it.

\begin{proposition}
\label{Grothendieck-divisibility-prop}
Let $A$ be a finite-dimensional $\FF$-algebra,
with $\FF$ algebraically closed. (We do not require $A$ to be a
Hopf algebra here.) Let $P_i$, $S_i$, $\ppp$, $\sss$ and $C$ be as in
Subsection~\ref{subsect.findim-alg}.

\begin{enumerate}
\item[(i)]
The class $[A]$ of the left-regular $A$-module lies in
$\gcd(\sss) \cdot K_0(A)$.
\item[(ii)]
When $A \cong A^{\opp}$ as rings, 
the class $[A]$ of the left-regular $A$-module
lies in $\gcd(\ppp) \cdot G_0(A)$.
\end{enumerate}
\end{proposition}
\begin{proof}
Assertion (i) follows since \eqref{left-regular-decomposition}
shows 
$
[A]= \sum_{i=1}^{\ell+1} \dim(S_i) [P_i]
$
in $K_0(A)$.

(ii) We have $[A] = \sum_{i=1}^{\ell+1} [A:S_i] [S_i]$ in $G_0(A)$.
Hence, it is enough to show that $\gcd(\ppp) \mid [A:S_i]$ for each
$i \in \left\{1,2,\ldots,\ell+1\right\}$. So let us fix $i$. Choose
a primitive idempotent $e_i \in A$ such that $P_i \cong A e_i$. The
hypothesis $A \cong A^{\opp}$ shows that there is a ring isomorphism
$\phi : A \to A^{\opp}$. The image of the primitive idempotent $e_i$
under this isomorphism $\phi$ must be another primitive idempotent
$e_j$ of $A$, and furthermore we have $\phi\left(e_iA\right) =
Ae_j$, whence $\dim (e_iA) = \dim (Ae_j)$. Now,
\eqref{eq.alg-reps.V:Si} yields
\[
[A:S_i]
=\dim \Hom_A(P_i,A)\\
=\dim \Hom_A(Ae_i,A)\\
=\dim (e_iA) = \dim (Ae_j) = \dim P_j\text{ for some }P_j.
\]
Here the third equality 
used the $\FF$-linear isomorphism $\Hom_A(Ae_i,V) \cong e_iV$
(defined for each $A$-module $V$)
which sends $\varphi$ to $\varphi(e_i)$.
Since $P_j$ is projective, we have $\gcd(\ppp) \mid \dim P_j
= [A:S_i]$, as desired.
\end{proof}

\begin{example}
The matrix algebra $A=\Mat_n(\FF)$ is semisimple, with only 
one simple $A$-module $S_1(=P_1)$ having $\dim(S_1)=n$.
Thus in this case, $n=\gcd(\sss)=\gcd(\ppp)$, and indeed, 
$[A]=n[S_1]=n[P_1]$ in $G_0(A)(=K_0(A))$.
\end{example}

Each finite-dimensional Hopf algebra $A$ satisfies $A \cong A^{\opp}$
as rings (via the antipode $\alpha : A \to A^{\opp}$); therefore,
Proposition~\ref{Grothendieck-divisibility-prop} (ii) can be applied
to any such $A$. For example, we obtain the following:

\begin{example}
Consider again the generalized Taft Hopf algebra $A=H_{n,m}$
from Example~\ref{Taft-algebra-example}. As we know from
Example~\ref{Taft-algebra-gcd-example}, we have $\gcd(\ppp)=m$.
Hence, Proposition~\ref{Grothendieck-divisibility-prop} (ii) shows that
$[A]$ lies in $m G_0(A)$.
\end{example}
\end{verlong}

\section*{Acknowledgements}
The authors thank Georgia Benkart, Sebastian Burciu, Pavel Etingof, Jim Humphreys, Radha Kessar, Peter J. McNamara, Hans-J\"urgen Schneider, Peter Webb, and Sarah Witherspoon 
for helpful conversations and references.
%regarding Proposition~\ref{group-algebra-gcd},
%and thank Jim Humphreys for extensive information about
%restricted universal enveloping algebras.  They also thank Peter J. McNamara 
%for pointing out an instructive example. 
% http://mathoverflow.net/questions/261254/is-the-cartan-matrix-of-a-finite-dimensional-hopf-algebra-invertible-over-the

\end{document}